\documentclass[12pt, reqno]{amsart}
\usepackage[utf8]{inputenc}
\usepackage{url}
\usepackage{amsmath,color}
\usepackage{amsfonts}
\usepackage{amssymb}
\usepackage{amsthm}
\usepackage{multirow}
\usepackage{graphicx} 
\usepackage{lipsum}

\newtheorem{theorem}{Theorem}
\newtheorem{remark}[theorem]{Remark}
\newtheorem{lemma}[theorem]{Lemma}
\newtheorem{corollary}[theorem]{Corollary}
\newtheorem{proposition}[theorem]{Proposition}

\newcommand{\tmop}[1]{\ensuremath{\operatorname{#1}}}
\numberwithin{equation}{section}
\numberwithin{theorem}{section}

\numberwithin{equation}{section}

\def\sumplus{\sideset{}{^+}\sum}
\def\sumstar{\sideset{}{^*}\sum}

\def\sumprime{\sideset{}{'}\sum}
\def\pamod{\!\!\!\!\pmod}

\newcommand{\fixmehide}[1]{}
\newcommand{\fixmelater}[1]{}
\newcommand{\fixmedone}[1]{}
\newcommand{\fixmehidden}[1]{}

\title{Weighted central limit theorems for central values of $L$-functions}

\author{Hung M. Bui, Natalie Evans, Stephen Lester and Kyle Pratt}
\address{Department of Mathematics, University of Manchester, Manchester M13 9PL, UK}
\email{hung.bui@manchester.ac.uk}
\address{Department of Mathematics, King’s College London, London WC2R 2LS, UK}
\email{natalie.evans@kcl.ac.uk}
\email{steve.lester@kcl.ac.uk}
\address{All Souls College, Oxford OX1 4AL, UK}
\email{kyle.pratt@maths.ox.ac.uk}

\subjclass[2010]{11F41, 11F66, 11F11, 11M06 \\ \indent \textit{Keywords and phrases}: central limit theorem, mollifier, central values, simultaneous non-vanishing}

\allowdisplaybreaks

\begin{document}

\maketitle

\begin{abstract}
We establish a central limit theorem for the central values of Dirichlet $L$-functions with respect to a weighted measure on the set of primitive characters modulo $q$ as $q \rightarrow \infty$. Under the Generalized Riemann Hypothesis (GRH), we also prove a weighted central limit theorem for the joint distribution of the central $L$-values corresponding to twists of two distinct primitive Hecke eigenforms. As applications, we obtain (under GRH) positive proportions of twists for which the central $L$-values simultaneously grow or shrink with $q$ as well as a positive proportion of twists for which linear combinations of the central $L$-values are nonzero.
\end{abstract}

\section{Introduction}

Understanding the behavior of central $L$-values is an important topic of study in number theory, with profound connections to problems in arithmetic as well as other areas of mathematics. Since central $L$-values are often difficult to study individually, a fruitful approach is to embed the $L$-values within a wider family and examine their statistical properties. A fundamental example of a family of $L$-functions is those attached to primitive Dirichlet characters modulo $q$, and one may ask how
the central values of these $L$-functions are distributed when varying over such characters as $q \rightarrow \infty$. 

Selberg's central limit theorem \cite{Selberg-contributions, Selberg-old-new} for the Riemann zeta function states that
\[
\frac{1}{T} \tmop{meas}\bigg\{ t \in [T,2T] : \frac{\log |\zeta(\tfrac12+it)|}{ \sqrt{\tfrac12\log \log T}} \in (a,b) \bigg\}=\frac{1}{\sqrt{2\pi}} \int_a^b e^{-u^2/2} \, du+O\left( \frac{(\log \log \log T)^2}{\sqrt{\log \log T}}\right)
\]
as $T \rightarrow \infty$ and is emblematic of what one might expect to be true for a family of $L$-functions. Similarly, a folklore conjecture predicts that as $\chi$ ranges over primitive Dirichlet characters modulo $q$, the value $\log |L(\tfrac12, \chi)|$ has a Gaussian limiting distribution with mean $0$ and variance $\tfrac12 \log \log q$ as $q \rightarrow \infty$ (see \cite{KS00, selberg-dir} for related discussions). Proving such a result remains completely out of reach, as it would imply $100\%$ of these central $L$-values are non-zero, which is a well-known open conjecture. The problem becomes even more difficult when considering central values of higher degree $L$-functions such as $L$-functions associated to twists of automorphic forms.

In this article we overcome the barrier of the vanishing of the central value by introducing a weight which accounts for when this value is zero. Our main results establish central limit theorems with respect to this weighted measure on the set of primitive characters modulo $q$ for the central values of Dirichlet $L$-functions as $q \rightarrow \infty$, as well as for the joint distribution of the central $L$-values corresponding to twists of two distinct primitive Hecke eigenforms as $q \rightarrow \infty$. The latter result is conditional on the assumption of GRH.

\subsection{Main Results}
Let $\varphi^{\star}(q)$ denote the number of primitive characters modulo $q$; we always work with prime $q$ for technical simplicity. Throughout, we write $\sum_{\chi \pmod q}^*$ to indicate that the summation is restricted to primitive characters, $\chi\ne\chi_0$. Given a complex-valued function $F$ on the set of primitive characters modulo $q$, we define 
\[
\varphi_F^{\star}(q)=\sumstar_{\chi \pamod q} F(\chi).
\]
Let $\mu_F$ be the complex measure on the set of primitive characters modulo $q$ given by
\[
\mu_F(S)= \frac{1}{\varphi_F^{\star}(q)}\sumstar_{\chi \in S} F(\chi), \qquad S \subset \{ \chi \pamod q\}.
\]
For example, if $F=1$ then $\varphi_F^\star(q)=\varphi^{\star}(q)$ and $\mu_F$ is the usual counting measure. 
To account for the vanishing of the central $L$-value we will choose our weight $F$ so that $F(\chi)=0$ whenever the central value vanishes.
Moreover, 
to capture the typical behavior of the $L$-function we would like that $F \approx 1$ as to not bias our measure. 
Our approach takes $F$ to be the central $L$-value multiplied by a mollifier, which dampens the extreme behavior of the central $L$-values.

Let us now introduce our mollifier. The precise definition is technical, but mainly the technicalities arise to ensure the mollifier behaves, on average, like an Euler product.
Let $\lambda(n)=(-1)^{\Omega(n)}$ be the Liouville function, where $\Omega(n)=\sum_{p^a || n } a$ and $p^a||n$ means that $p^a|n$ and $p^{a+1} \nmid n$.
Define the multiplicative function $\nu(n)$ by $\nu(p^a)=\frac{1}{a!}$. Also, let $\eta>0$ be a sufficiently small constant. For each $0\leq j\leq J$ let
$
\theta_j= \eta \frac{e^j}{(\log \log q)^{5}},  \ell_j=2 \lfloor  \theta_j^{-3/4} \rfloor , 
$
where $J$ is chosen so that $ \eta \le \theta_J \le e \eta $ (so $J \asymp \log \log \log q$). Let
$y=q^{\theta_0}$ and  $x=q^{\theta_J}$.
Set $I_0=(c_0, y]$, where $c_0$ is fixed and sufficiently large, and for $1\leq j\leq J$ let $I_j=(q^{\theta_{j-1}}, q^{\theta_j}]$. 
We then define
\begin{equation*} 
\mathcal M_j(\chi)=\sum_{\substack{p|n \Rightarrow p \in I_j \\ \Omega(n) \le \ell_j}}\frac{ \lambda(n)\nu(n) \chi(n)}{\sqrt{n}},
\qquad 
\mathcal M(\chi)=\prod_{j=0}^{J} \mathcal M_j(\chi).
\end{equation*}
The Dirichlet polynomial $\mathcal{M}(\chi)$ will be our mollifier.

We investigate the distribution of $\log |L(\tfrac12,\chi)|$ as $\chi$ varies over primitive characters modulo $q$ with respect to $\mu_{\mathcal W}$, where $\mathcal W(\chi)=L(\tfrac12,\chi)\mathcal M(\chi)$. The weight function $\mathcal W(\chi)$ can be interpreted as a truncated Hadamard product over the low-lying zeros of $L(s,\chi)$ with ordinates $\le 1/\log x$ in magnitude times $L(1,\chi^2)^{1/2}$, which we will justify later; so while the weight knows about the central value its knowledge should typically be restricted to a bounded number of low-lying zeros, such as a possible zero at the central point.
Additionally, we will see that $\varphi_{\mathcal W}^{\star}(q)\asymp q$, and will also prove the following proposition, which shows that our weight is typically not very large.
\begin{proposition}\label{upperboundsecondDirichlet}
Uniformly for $\alpha,\beta \in \mathbb C$ with $|\alpha|,|\beta|\ll (\log q)^{-1}$ we have that
\[
\sideset{}{^*}\sum_{\chi \pamod q}L(\tfrac12+\alpha, \chi)L(\tfrac12+\beta,\overline{\chi})|\mathcal{M}(\chi)|^2\ll q.
\]

\end{proposition}

 Our first main result establishes a central limit theorem for the logarithm of the central values of Dirichlet $L$-functions with respect to $\mu_{\mathcal W}$. 

\begin{theorem}\label{thm:CLTDirichlet} Let $a,b \in \mathbb R$ with $a<b$. We have as $q \rightarrow \infty$ that
\begin{align*}
\mu_{\mathcal W}\bigg( \bigg\{ \chi \pamod q, \chi \neq \chi_0:  \frac{\log |L(\tfrac12,\chi)|}{\sqrt{\tfrac12 \log \log q}} \in (a,b) \bigg\}  \bigg) =\frac{1}{\sqrt{2\pi}} \int_{a}^b e^{-u^2/2} \, du +o(1).
\end{align*}
\end{theorem}

The proof we give yields an upper bound $O_\varepsilon((\log \log q)^{-1/4+\varepsilon})$ on the rate of convergence, and it is possible to improve this to $O_\varepsilon((\log \log q)^{-1/2+\varepsilon})$. There is also some flexibility in the choice of the weighted measure. For example, the analogue of Theorem \ref{thm:CLTDirichlet} with $\mathcal W(\chi)$ replaced by $L(\tfrac12,\chi) \mathcal M_0(\chi)$ holds (however the conclusion of the analogue of Proposition \ref{upperboundsecondDirichlet} no longer holds). Using results on the fourth moment of Dirichlet $L$-functions \cite{young, Hou16, zacharias}, it would be possible to prove an analogue of this result with $|\mathcal W(\chi)|^2$ in place of $\mathcal W(\chi)$.

Our next result establishes a weighted central limit theorem for the joint distribution of central $L$-values of twists of two distinct automorphic forms.
Let $f$ and $g$ be fixed, distinct weight $\kappa$ newforms on $\Gamma_0(N)$ and write
\begin{align*}
L(s,f)&=\sum_{n=1}^{\infty}\frac{\lambda_f(n)}{n^s}=\prod_{p}\left(1-\frac{\alpha_{f,1}(p)}{p^s} \right)^{-1}\left(1-\frac{\alpha_{f,2}(p)}{p^s} \right)^{-1}, \qquad  \tmop{Re}(s)>1,
\end{align*}
where $\lambda_f(n)$ denotes the $n$th Hecke eigenvalue of $f$ and $\alpha_{f,1}(p),\alpha_{f,2}(p)$ are the Satake parameters of $f$; that is, they are complex numbers which satisfy $\alpha_{f,1}(p)\alpha_{f,2}(p)=1$ and $\alpha_{f,1}(p)+\alpha_{f,2}(p)=\lambda_f(p)$.
Also define $L_{\infty}(s,f)=(2\pi)^{-s} \Gamma(s+\frac{\kappa-1}{2})$ and $\Lambda(s,f)=N^{s/2}L_{\infty}(s,f)L(s,f)$.
The functional equation is
\[
\Lambda(s,f)=\varepsilon(f) \Lambda(1-s,f),
\]
where $\varepsilon(f) \in \{\pm 1\}$ is the root number. 

As before, we choose our weight to be the product of the central $L$-values and a mollifier. Define the completely multiplicative functions $w_j(n)$ and $a_{f,j}(n)$ by
\begin{equation} \label{eq:adef}
\begin{split}
w_j(p)=\frac{1}{p^{\frac{1}{\theta_j \log q}}} \left(1-\frac{\log p}{\theta_j \log q} \right), \qquad a_{f,j}(p)&=\lambda_f(p) w_j(p).
\end{split}
\end{equation}
Let
\begin{equation}\label{eq:mollifierdef}
M_{f,j}(\chi)=\sum_{\substack{p|n \Rightarrow p \in I_j \\ \Omega(n) \le \ell_j}}\frac{  a_{f,J}(n) \lambda(n)\nu(n) \chi(n)}{\sqrt{n}}, 
M_f(\chi)=\prod_{j=0}^{J} M_{f,j}(\chi), \text{ and }   M(\chi)=M_f(\chi)M_g(\overline \chi),
\end{equation}
where $M_g$ is defined completely analogously to $M_f$. We take our weight to be
$
W(\chi)=L(\tfrac12, f \otimes \chi) L(\tfrac12, g \otimes \overline \chi) M(\chi)$.
We prove in Section \ref{sec:random} that $\varphi_W^{\star}(q) \asymp q$. Additionally, we can bound the moments of $W(\chi)$ under GRH.

\begin{proposition} \label{prop:mollified}
Assume GRH.
Let $k >0$ and suppose $\eta\lceil k \rceil$ is sufficiently small. Then we have
\begin{equation}\label{eq:momentbound}
\sumstar_{\chi \pamod q} |W( \chi)|^{2k} \ll_{f,g,k} q.
\end{equation}
\end{proposition}
Here we assume that GRH holds for $L(s,f\otimes \chi),L(s,g\otimes \chi), L(s,\tmop{Sym}^2 f \otimes \chi) $, $L(s,\tmop{Sym}^2 g \otimes \chi)$ and $L(s,\chi)$ for all characters $\chi$ modulo $q$.
 We now state our second main result.

\begin{theorem} \label{thm:cltjoint}
Assume GRH. Let $\varepsilon>0$ and suppose that $\eta=\eta(\varepsilon)$ is sufficiently small. Then for any intervals $I_1, I_2 \subset \mathbb R$ we have that
\[
\begin{split}
&\mu_W\bigg(\bigg\{ \chi \pamod q, \chi\neq \chi_0 : \bigg(\frac{\log |L(\tfrac12,f \otimes \chi)|}{\sqrt{\frac12 \log \log q}},\frac{\log |L(\tfrac12, g\otimes \chi)|}{\sqrt{\frac12 \log \log q}}\bigg) \in I_1 \times I_2 \bigg\}\bigg)
\\
&\qquad \qquad \qquad \qquad \qquad = \frac{1}{2\pi} \int_{I_1 \times I_2} e^{-\frac12(u^2+v^2)} \, du \, dv+O_{\varepsilon,f,g}\left(  (\log \log q)^{-1/2+\varepsilon}\right).
\end{split}
\]
\end{theorem}
 In the proof of Theorem \ref{thm:cltjoint}, the only place GRH is required is to bound the moments of $W(\chi)$. If we instead assume the conclusion of Proposition \ref{prop:mollified} with $k=1$ holds, then we can obtain (with some additional work) the same conclusion as in Theorem \ref{thm:cltjoint}, except with an error term of size $O_{f,g}(\frac{\sqrt{\log \log \log q}}{(\log \log q)^{1/4}})$.

\subsection{Applications} Since our weight does not bias our measure to a great extent we can gain insight into the typical behavior of the central $L$-values. Building on methods of Rohrlich \cite{Rohrlich1,Rohrlich2}, Chinta \cite{Chinta} has shown that $L(\tfrac12,f\otimes \chi) \neq 0$ for $100\%$ of primitive characters $\chi \pmod q$. Our first application goes beyond non-vanishing of twists, and shows that a positive proportion of $\chi$ give rise to simultaneous values that either grow or shrink with $q$.

\begin{corollary} \label{cor:simul large vals}
Assume GRH. Let $c>0$ be fixed. There are $\gg_{f,g,c} q$ characters $\chi \pmod{q}$ such that, for each such $\chi$, we simultaneously have
\begin{align*}
|L(\tfrac12, f \otimes \chi)| > \exp(c \sqrt{\log\log q})
\quad \text{ and } \quad
|L(\tfrac12, g \otimes \chi)|>\exp(c \sqrt{\log\log q}).
\end{align*}
Additionally, there are $\gg_{f,g,c} q$ characters $\chi \pmod{q}$ such that, for each such $\chi$, we simultaneously have
\begin{align*}
0<|L(\tfrac12, f \otimes \chi)| < \exp(-c \sqrt{\log\log q})
\quad \text{ and } \quad
0<|L(\tfrac12, g \otimes \chi)|<\exp(-c \sqrt{\log\log q}).
\end{align*}
\end{corollary}

Corollary \ref{cor:simul large vals} follows from combining Proposition \ref{prop:mollified} and Theorem \ref{thm:cltjoint} and shows that the twists $|L(\tfrac12, f \otimes \chi)|$ and $|L(\tfrac12, g \otimes \chi)|$ simultaneously obtain somewhat large values for a positive proportion of twists, and it is best possible. It is possible to obtain larger central values, but these large values do not appear for a positive proportion of twists. For instance, there are values of $|L(\tfrac12, f \otimes \chi)|$ as large as 
$
\exp (c\, \sqrt{\frac{\log q}{\log \log q}}),$
which can be proved via the resonance method (see \cite[Theorem 1.11]{BFK}), and it is even possible for the angle of the central value to be constrained. Blomer et. al. also used the resonance method \cite[Theorem 1.12]{BFK} to show that there exist non-trivial
characters $\chi \pmod q$ such that
\begin{align*}
|L(\tfrac12, f \otimes \chi)L(\tfrac12, g \otimes \chi)| > \exp \left(c \sqrt{\frac{\log q}{\log \log q}} \right),
\end{align*}
but again these large values do not appear for a positive proportion of twists.

Famously, a conjecture of Lehmer predicts the Ramanujan $\tau$-function never vanishes and more generally one may wonder how often the Fourier coefficients of automorphic forms vanish (for non co-compact spaces). For fundamental Fourier coefficients of half-integral weight cusp forms this question is directly related to understanding nonvanishing of linear combinations of central $L$-values by Waldspurger's Theorem\footnote{Here one expands the cusp form in terms of a Hecke basis and applies Waldspurger's Theorem for each Hecke eigenform. It is an open problem, even under GRH, to establish a positive proportion of nonvanishing for fundamental Fourier coefficients of half-integral weight forms, whereas this is known under GRH for Hecke eigenforms of level $4$  (see \cite{LR21}).}. Furthermore, fundamental Fourier coefficients of Siegel cusp forms can also be expressed in terms of linear combinations of central $L$-values by a classical construction of Eichler and Zagier in certain cases (see \cite[Lemma 5.1]{JLS}) and more generally by the refined Gan--Gross--Prasad Conjecture\footnote{A proof of the conjecture has been announced by Furusawa and Morimoto (RIMS conference ``Analytic, geometric and $p$-adic aspects of automorphic forms and $L$-functions", January 2020)} (see \cite[Theorem 1.13]{DPSS15}).
Motivated by these relationships, it is natural to wonder how often linear combinations of central $L$-values vanish. 

Since combining Theorem \ref{thm:cltjoint} and Proposition \ref{prop:mollified} yields a positive proportion of characters for which one central $L$-value is large while the other is small, under GRH we obtain a positive proportion of twists for which $aL(\tfrac12,f\otimes \chi)+bL(\tfrac12,g\otimes \chi) \neq 0$, for $a,b \in \mathbb C$. 
Moreover, this argument gives the following result.

\begin{corollary} \label{cor:nonvanishing}
Assume GRH. Then there exist $\gg_{f,g} q$ characters $ \chi \pmod q$ such that for any $\{ a_\chi\}_{\chi \pmod q}, \{b_\chi\}_{\chi \pmod q} \subset \mathbb C$ with $|a_{\chi}| , |b_{\chi} |\asymp 1$ and any fixed $r_1,r_2>0$  we have 
\[a_{\chi} |L(\tfrac12,f \otimes \chi)|^{r_1}+b_{\chi} |L(\tfrac12, g \otimes \chi)|^{r_2} \neq 0.\]
\end{corollary}

Our last application shows that there exists a sparse set of characters $\chi \pmod q$ for which the the central $L$-values are relatively close together in magnitude.
\begin{corollary} \label{cor:ratios} Assume GRH. Let $\varepsilon>0$. Then for $(\log \log q)^{\varepsilon} \le \Lambda \le \sqrt{\log \log q}$ we have

\scalebox{0.9}{
$
\# \bigg\{ \chi \pmod q, \chi \neq \chi_0 : L(\tfrac12,g \otimes \chi) \neq 0, \, e^{-\Lambda}  \le \bigg| \displaystyle \frac{L(\frac12,f\otimes \chi)}{L(\frac12, g \otimes \chi)} \bigg| \le e^{\Lambda} \bigg\} \gg_{f,g,\varepsilon} q\,  \bigg( \frac{\Lambda}{(\log \log q)^{1/2}} \bigg)^{1+\varepsilon}.
$}
\end{corollary}
We expect that the above bound should be optimal up to the factor of $\Lambda^{\varepsilon}(\log \log q)^{-\varepsilon/2}$.  It would be interesting to determine the smallest $\Lambda$ for which the set above is nonempty and by analogy with the conjecture that the set $\{\zeta(\tfrac12+it)\}_{t \in \mathbb R}$ is dense in $\mathbb C$ one might wonder whether this persists even for $\Lambda=o(1)$, thereby balancing the size of the central $L$-values and potentially allowing for linear combinations to vanish.

Finally let us mention that if in addition to GRH we also assume the Ramanujan-Petersson Conjecture our arguments carry over with only a few modifications needed to the case of $\tmop{GL}(2)$ Maass cusp forms. Let $\phi_1,\phi_2$ be distinct Maass newforms of level $N$. Assuming GRH for $L(s,\phi_1\otimes \chi),L(s,\phi_2\otimes \chi), L(s,\tmop{Sym}^2 \phi_1 \otimes \chi) $, $L(s,\tmop{Sym}^2 \phi_2 \otimes \chi)$, $L(s,\chi)$ for all characters $\chi$ modulo $q$ and $|\lambda_{\phi_j}(p)|\le 2$ for $j=1,2$ 
we obtain analogues of all of the above corollaries for  central values of $L$-functions attached to twists of $\phi_1,\phi_2$. In particular, this shows that under these hypotheses
that $L(\tfrac12, \phi_1 \otimes \chi) \neq 0$ and $L(\tfrac12,\phi_2 \otimes \chi) \neq 0$ simultaneously for $\gg_{\phi_1,\phi_2} q$ characters $\chi \pmod q$.

\subsection{Discussion of past work} 
Selberg's \cite{Selberg-old-new} work extends in great generality. For instance, he was able to prove a central limit theorem for the joint distribution of $\log L(\tfrac12+it,\chi_1), \ldots, \log L(\tfrac12+it,\chi_N)$ with $t \in [T,2T]$ and $\chi_1,\ldots,\chi_N$ distinct primitive characters $\pmod q$, where the limit is taken as $T \rightarrow \infty$. Furthermore, Bombieri and Hejhal \cite{Bombieri-Hejhal} showed that this method extends to higher degree $L$-functions under plausible hypotheses, such as a sufficiently strong zero density estimate. 

Hough \cite{Hough} adapted Selberg's method to study the distribution of central values of families of $L$-functions (see also \cite{darbar-lumley}).  For the family of quadratic Dirichlet characters, he established a one-sided central limit theorem; that is, he bounded the proportion of fundamental discriminants, in magnitude $<X$, for which the logarithm of the normalized central $L$-value is larger than a given  $V>0$ by $\le (1+o(1)) \frac{1}{\sqrt{2\pi}} \int_{V}^{\infty} e^{-u^2/2} \, du$ as $X \rightarrow \infty$. Similar to Selberg's work, a major ingredient in Hough's argument is an analogous zero density estimate. Accounting for the possible vanishing of the central $L$-value remains, however, a major obstacle. Hough was able to prove a central limit theorem for these central $L$-values under a certain spacing hypothesis on the distribution of the low-lying zeros of the $L$-functions, which follows from GRH and the Density Conjecture (see \cite{ILS}).
 Further progress towards proving a central limit theorem for central $L$-values using Selberg's approach appears dependent upon improving our knowledge on the low-lying zeros of the family.  For example, even though Chinta's result \cite{Chinta} shows that $L(\tfrac12,f \otimes \chi) \neq 0$ for $100\%$ of characters $\chi$ modulo $q$ it does not provide sufficient bounds to control the effect of possible extremely low-lying zeros on the distribution of the central $L$-values.
 
 In breakthrough work, Radziwi{\l}{\l} and Soundararajan \cite{Radziwill-Soundararajan} found a new method to study central $L$-values of families. They constructed a mollifier which has roughly the shape of an Euler product, yet still controls the extreme values at the central point. Using this approach, they proved a one-sided central limit theorem for central $L$-values for the family of quadratic twists of an elliptic curve. In comparison to Hough's work, Radziwi{\l}{\l} and Soundararajan required only a first moment, whereas Hough used a second moment.   Their work has sparked many recent innovations such as a new proof of Selberg's central limit theorem \cite{RS17}, bounds for moments of $L$-functions \cite{heap-radziwill-soundararajan, heap-soundararajan, BFK21} with applications to non-vanishing at the central point \cite{DFL20,LR21}, and progress towards the Fyodorov-Hiary-Keating Conjecture \cite{FHK12,FK14,Naj18,ABBRS19,Har19,ABR20}. Radziwi{\l}{\l} and Soundararajan \cite{Radziwill-Soundararajan-rh-notes} have also made further progress towards establishing a central limit theorem for central $L$-values. For quadratic Dirichlet characters, they showed that the proportion of fundamental discriminants $< X$ in absolute values for which the logarithm of the normalized central $L$-values lies in an interval $I$ is $\ge \tfrac78(1+o(1))\frac{1}{\sqrt{2\pi}} \int_{I} e^{-u^2/2} \, du$.
 A remarkable feature of their work is that the leading constant in the lower bound
matches the best known proportion of non-vanishing \cite{SoundNonvanishing}. Recently, Fazzari \cite{Fazzari1,Fazzari2} has proved several weighted central limit theorems for the Riemann zeta function. Under RH, he proved central limit theorems for $\log |\zeta(\tfrac12+it)|$ with respect to the measures on $[T,2T]$ given by $|\zeta(\tfrac12+it)|^{2k} \, dt$  for $k=1,2$, and for $k \ge 3$ under RH and an additional assumption on moments of the zeta function. In contrast to the prior work of Soundararajan and Radziwi{\l}{\l} \cite{Radziwill-Soundararajan-rh-notes}, the main objective in this paper is establishing an asymptotic for the weighted distribution of the central values whereas they focus on removing the weight and optimizing the constant in the lower bound by using a more refined mollifier. One additional difference in our work is that by assuming GRH we only require a first moment, however this is at the cost of obtaining a rather small constant in the lower bound.

\subsection{Outline of the proof}

We discuss the proof of Theorem \ref{thm:cltjoint}, since the proof of Theorem \ref{thm:CLTDirichlet} is easier. 
In both cases the main inputs into our argument are estimates for twisted first moment and upper bound for a mollified second moment. The proof of Theorem \ref{thm:cltjoint} also uses some well-known analytic properties of Rankin-Selberg $L$-functions to handle the joint distribution.

Recall that for intervals $I_1, I_2 $ we wish to prove an asymptotic formula for
\begin{align*}
\sumstar_{\chi \pamod q} W(\chi)  \mathbf 1_{I_1 \times I_2} \bigg(\frac{\log |L(\tfrac12, f \otimes \chi)|}{\sqrt{\tfrac12 \log \log q}}, \frac{\log |L(\tfrac12, g \otimes \chi)|}{\sqrt{\tfrac12 \log \log q}}\bigg),
\end{align*}
where $\mathbf 1_S$ denotes the indicator function of the set $S$. The primary complications arise from the presence of $\log |L(\tfrac12, f\otimes \chi)|$ and $\log |L(\tfrac12, g \otimes \chi)|$; these are difficult to control due to the possible existence of very low-lying zeros, which we cannot rule out. We would like to replace the logarithms of the $L$-functions by more tractable expressions. The basic strategy we follow is due to Radziwi{\l\l} and Soundararajan \cite{RS17}, who were the first to use this kind of Euler product-like mollifier at the central point. Since $M(\chi)=M_f(\chi)M_g(\chi)$ is a mollifier we expect
\begin{align*}
 L(\tfrac12, f \otimes \chi)  M_f(\chi) \approx 1 \qquad \text{ and } \qquad W(\chi)= L(\tfrac12, f \otimes \chi)  L(\tfrac12,g \otimes \overline \chi) M(\chi)\approx 1
\end{align*}
for most characters $\chi$. The mollifier is constructed so that
\begin{align*}
M_f(\chi) \approx \prod_{p\leq x} \left(1 - \frac{\lambda_f(p)\chi(p)}{p^{1/2}} \right),
\end{align*}
for most characters $\chi$, where $x$ is a small power of $q$. Taking these two approximations together implies
\begin{align*}
\log |L(\tfrac12, f \otimes \chi)| \approx \tmop{Re} \bigg(\sum_{p\leq x}\frac{\lambda_f(p)}{p^{1/2}}\chi(p)\bigg) ,
\end{align*}
and similarly for $\log |L(\tfrac12, g \otimes \chi)|$, so that the logarithms of the $L$-functions may be approximated by finite sums over primes. Using the structure of the mollifier in this way, we reduce the proof of Theorem \ref{thm:cltjoint} to computing 
\begin{align*}
\sumstar_{\chi \pamod q} W(\chi)  \mathbf 1_{I_1 \times I_2} \bigg(\frac{\tmop{Re} \big(\sum_{p \le x}  \frac{\lambda_f(p)}{p^{1/2}}\chi(p)\big)} {\sqrt{\tfrac12 \log \log q}} \ ,\frac{\tmop{Re}\big(\sum_{p \le x}  \frac{\lambda_g(p)}{p^{1/2}} \chi(p)\big)} {\sqrt{\tfrac12 \log \log q}}  \bigg).
\end{align*}
This in turn leads to expressions that are in terms of a twisted first moment of $L(\tfrac12,f \otimes \chi)L(\tfrac12, g\otimes \overline \chi)$. Such a twisted first moment has been computed in work of Blomer et. al. \cite{BFK} and we apply this result. 

There are a number of difficulties involved in making the above strategy rigorous and quantitatively strong. There are two principal issues. 

First, after applying the formula for the twisted first moment we are left with an unwieldy expression for the main term. To evaluate this expression we introduce a random $L$-function in which the Dirichlet characters $\chi(n)$ are modelled by random variables $X(n)$, and then match our expression with the random analogue. Here
$
X(n)= \prod_{p^a || n} X(p)^a,
$
and $\{X(p)\}_p$ is a sequence of i.i.d. uniformly distributed random variables on the unit circle. Comparison with a random model allows us to sidestep a number of technical points that would otherwise require involved effort to resolve. In particular, using the independence of the random variables $\{ X(p)\}_p$ reduces many of the computations to ``local'' ones evaluated at each prime. 

Second, we need to restrict ourselves to ``typical'' sets $\mathcal{S}$ of characters modulo $q$, where the complement $\mathcal{S}^c$ has size $O(q/\log \log q)$, say. One such set, for example, is the set of characters with $|W(\chi)| \le \log \log q$. In order to control the error involved in restricting to such sets we require an estimate for
\begin{align*}
\sum_{\chi \in \mathcal{S}^c} |W(\chi)|.
\end{align*}
Using Cauchy-Schwarz's inequality we obtain
\begin{align*}
\sum_{\chi \in \mathcal{S}^c} |W(\chi)| &\leq (\# \mathcal{S}^c)^{1/2} \bigg(\ \sumstar_{\chi \pamod q} |W(\chi)|^2  \bigg)^{1/2}.
\end{align*}
The saving comes from the small size of $\mathcal{S}^c$, so it suffices to show that
\begin{align*}
\sumstar_{\chi \pamod q} |W(\chi)|^2 \ll q.
\end{align*}
Establishing this result is roughly on the level of difficulty of proving an upper bound of the correct order of magnitude for a mollified eighth moment of Dirichlet $L$-functions. Such a feat is well beyond the range of unconditional techniques, and it is here that we require GRH in order to make progress. Our arguments follow along the lines of \cite{LR21}, which builds on the key works of Soundararajan \cite{S}, Radziwi{\l}{\l} and Soundararajan \cite{Radziwill-Soundararajan}, and Harper \cite{Har19}.

It is desirable to prove weighted central limit theorems with nonnegative weights instead of our complex-valued weights $W(\chi)$. For instance, if $W(\chi)$ were nonnegative, then the measure $\mu_W(S)$ would be a genuine probability measure. As mentioned above, it would be possible to carry out such a program in the case of Dirichlet $L$-functions, in which case one has access to unconditional results on the twisted fourth moment. Assuming GRH, it is also possible to achieve this for twists of holomorphic newforms, where GRH is required to obtain an upper bound for the mollified fourth moment of $L(\tfrac12,f\otimes \chi)$. But we do not know how to do this for the twists $L(\frac{1}{2},f \otimes \chi)L(\tfrac12,g \otimes \overline \chi)$, since this seems to require asymptotic evaluation of a second moment, which is well beyond what is currently possible.

If we were working with a nonnegative weight $W(\chi) \ge 0$ we could control $\mu_W(S^c)$ in many places using Chebyshev's inequality and a twisted first moment, and one could prove a weighted central limit theorem for the approximating Dirichlet polynomial. However, to pass to the $L$-function we need to bound the \textit{weighted} measure of the set of characters with $W(\chi) \ge \log \log q$, and this requires input beyond a first moment of $W(\chi)$. This is because if the bulk of the contribution to the first moment of $W(\chi)$ were to come from the set of characters with $W(\chi) \ge \log \log q$ (which cannot be ruled out by a first moment alone) then the weighted measure of this set would be $\asymp \varphi_W^{\star}(q)$.

\subsection{The structure of $\mathcal W(\chi)$}
As we mentioned earlier our mollifier is constructed so that it mimics an Euler product and it is not too hard to prove that for all primitive characters $\chi \pmod q$ outside a set of size $\ll q e^{-1/(3\eta)}$ that
\[
\mathcal M(\chi)=\exp\bigg(-\sum_{c_0< p \le x}\frac{\chi(p)}{\sqrt{p}}\bigg)(1+O(e^{-(3\eta)^{-3/4}})).
\]
By the explicit formula, we can transform the sum over primes into a sum over zeros of $L(s,\chi)$. Using the formula above along with the hybrid Euler-Hadamard product for $L(s,\chi)$ \cite[Theorem 1]{bui-keating} it is not hard to see that outside a set of $\ll q e^{-1/(3\eta)}$ non-principal characters $\chi \pmod q$ with $\chi^2 \neq \chi_0$ we have
\[
|\mathcal W(\chi)| \asymp |L(1,\chi^2)|^{1/2} \exp\bigg(-\tmop{Re} \sum_{\rho_{\chi}} U((\tfrac12-\rho_{\chi}) \log x) \bigg),
\]
where the sum is over the nontrivial zeros of $L(s,\chi)$, $U(z)=\int_0^{\infty} u(t) E_1(z\log t) \, dt$, $E_1(z)=\int_{z}^{\infty} e^{-w} \frac{dw}{w}$ and $u(t)$ is an arbitrary nonnegative Schwartz function with compact support on $[e^{1-1/x},e]$ with unit mass.
Here the term $L(1,\chi^2)^{1/2}$ arises from the contribution of the squares of primes in \cite[Theorem 1]{bui-keating}. Following the discussion in \cite{gonek-keating-hughes} one can interpret the second factor on the right hand side as a truncated Hadamard product over zeros with ordinates $\le 1/\log x$ in magnitude.
Additionally, since $x$ is a power of $q$, in view of the Density Conjecture and the rapid decay of $U$ we expect that typically the sum above will be effectively restricted to a bounded number of zeros none of which is exceptionally close to the central point, so that for most $\chi$ we expect that $|\mathcal W(\chi)| \asymp |L(1,\chi^2)|$.

\subsection{Organization of the paper.}
In Section \ref{sec:twisted} we match the characteristic function of \\ $\big(\tmop{Re}\big( \sum_{p \le x}  \frac{\lambda_f(p)\chi(p)}{p^{1/2}}\big) \ ,\tmop{Re}\big(\sum_{p \le x}  \frac{\lambda_g(p)\chi(p)}{p^{1/2}}\big)\big)$ with that of a random model, the latter of which is evaluated in Section \ref{sec:random}. The proofs of Theorem \ref{thm:cltjoint} and Corollary \ref{cor:ratios} are given in Section \ref{sec:proofcltjoint}. Proposition \ref{prop:mollified} is proved in Section \ref{upperbounds}. We prove Theorem \ref{thm:CLTDirichlet} in Section \ref{sec:twistedD} and Proposition \ref{upperboundsecondDirichlet} in Section \ref{sec:upperboundsD}.

\section{Joint distribution - Reduction to a random model} \label{sec:twisted}

We first recall the twisted first moment of the product of twists of automorphic $L$-functions \cite[Theorem 5.1]{BFK}.  Write
\[
L_f(\chi)=L(\tfrac12, f\otimes \chi), \qquad L(\chi)=L_f(\chi)L_g(\overline \chi).
\]

\begin{lemma}\label{BFKlemma}
Suppose $1 \le n_1, n_2 \le L$ and $(n_1n_2/(n_1,n_2)^2,N)=1$. Then 
\[
\begin{split}
&\frac{1}{\varphi^{\star}(q)}\ \sumstar_{\chi \pamod q} L(\chi) \chi(n_1) \overline{\chi(n_2)} =\frac12\sum_{m_1n_1=m_2n_2} \frac{\lambda_f(m_1)\lambda_g(m_2)}{\sqrt{m_1m_2}} V\left( \frac{m_1m_2}{q^2N^2} \right)\\
&\qquad\qquad+\frac{\varepsilon(f)\varepsilon(g)}{2}\sum_{m_1n_2=m_2n_1} \frac{\lambda_f(m_1)\lambda_g(m_2)}{\sqrt{m_1m_2}} V\left( \frac{m_1m_2}{q^2N^2} \right)+O_\varepsilon\left(L^{3/2} q^{-1/144+\varepsilon}\right),
\end{split}
\]
where
\[
V(\xi)= \frac{1}{2\pi i} \int_{(2)} \frac{L_{\infty}(s+\tfrac12,f)L_{\infty}(s+\tfrac12,g)}{L_{\infty}(\tfrac12,f)L_{\infty}(\tfrac12,g)} \frac{(\cos( \frac{\pi s}{12}))^{-48}}{s} \xi^{-s} \, ds .
\]
\end{lemma}

\begin{remark}
\emph{
The function $V(\xi)$ is approximately 1 for small values of $\xi$, and decays like $\xi^{-3}$ as $\xi \rightarrow \infty$ (see \cite[Lemma 2.19]{BFK}).
}
\end{remark}

\subsection{The random model}\label{subsectionrandom}
Let $\{X(p)\}_p$ be i.i.d. uniformly distributed random variables on the unit circle. Define
\[
X(n)= \prod_{p^a || n} X(p)^a.
\]
Observe that if $m,n \le q$ then by orthogonality of characters
\begin{equation} \label{eq:trivial}
\frac{1}{\varphi(q)}\sum_{\chi \pamod q} \chi(m) \overline {\chi(n)}=\mathbb E\Big(X(m) \overline {X(n)}\Big).
\end{equation}
Let
\[
L(X)=\frac{1}{2} \Big( L^{f,g}(X)+\varepsilon(f)\varepsilon(g)\, L^{g,f}(X) \Big),
\]
where
\begin{align*}
L^{f,g}(X)&=\sum_{m_1,m_2 \ge 1} \frac{\lambda_f(m_1)\lambda_g(m_2)}{\sqrt{m_1m_2}} X(m_1) \overline{X(m_2)}\,V\left(\frac{m_1m_2}{q^2N^2} \right).
\end{align*}
Then by Lemma \ref{BFKlemma}, \eqref{eq:trivial}, and bounding the contribution from $\chi= \chi_0$ trivially, we get
\begin{align} \label{eq:twisted}
&\frac{1}{\varphi^{\star}(q)}\ \sumstar_{\chi \pamod q} L(\chi)\chi(n_1) \overline{\chi(n_2)} =\mathbb E \left( 
L(X) X(n_1) \overline{X(n_2)}
\right) +O_\varepsilon\left(L^{3/2} q^{-1/144+\varepsilon} \right).
\end{align}

Let
\[
P_f(\chi)= \tmop{Re} \bigg(\sum_{c_0<p \le y}\frac{\lambda_f(p)w_J(p)}{\sqrt{p}}\chi(p)\bigg)
\quad\text{ and }\quad
P_f(X)= \tmop{Re}\bigg( \sum_{c_0<p \le y}\frac{\lambda_f(p)w_J(p)}{\sqrt{p}}X(p)\bigg).
\]
We also define
\[
M_{f,j}(X)=\sum_{\substack{p|n \Rightarrow p \in I_j \\ \Omega(n) \le \ell_j}}\frac{  a_{f,J}(n) \lambda(n)\nu(n) X(n)}{\sqrt{n}}, 
\ M_f(X)=\prod_{j=0}^{J} M_{f,j}(X), \   M(X)=M_f(X)M_g(\overline X)
\]
and $W(X)=L(X)M(X)$.
\begin{lemma} \label{lem:moments}
Let $j,k$ be nonnegative integers.
Suppose $y^{j}, y^{k} \le q^{1/1000}$. Then there exists $\delta>0$ such that
\begin{equation} \label{eq:moments1}
\frac{1}{\varphi^{\star}(q)}\ \sumstar_{\chi \pamod q} W(\chi) P_f(\chi)^{j}P_g(\chi)^{k}=\mathbb E \Big( W(X) P_f(X)^{j}P_g(X)^k \Big)+O \left( q^{-\delta}\right).
\end{equation}
Additionally, for $y^k \le q^{1/3}$ we have
\begin{equation} \label{eq:moments2}
\frac{1}{\varphi(q)} \sum_{\chi \pamod q}  P_f(\chi)^{2k} = \mathbb E\Big(  P_f(X)^{2k}\Big) \le k! \bigg(\sum_{c_0<p \le y} \frac{\lambda_f(p)^2}{p} \bigg)^k.
\end{equation}
\end{lemma}
\begin{proof}
Each of $P_f(\chi), P_g(\chi)$ is a Dirichlet polynomial of length $y$ having coefficients $\leq 1$ in magnitude (since $c_0$ is sufficiently large). Therefore, we derive the bounds $|P_f(\chi)^j| \leq y^j$ and $|P_g(\chi)^k|\leq y^k$. Similarly, since $\eta>0$ is sufficiently small, $M(\chi)$ is a Dirichlet polynomial of length $q^{1/1000}$ with bounded coefficients.
Using \eqref{eq:trivial} and \eqref{eq:twisted} we obtain \eqref{eq:moments1} with an error term of size $\ll q^{-\delta}$.

For \eqref{eq:moments2}, first observe that since $y^k \le q^{1/3}$ we have by \eqref{eq:trivial} that
\[
\frac{1}{\varphi(q)}  \sum_{\chi \pamod q}  P_f(\chi)^{2k}=\mathbb E\Big(  P_f(X)^{2k}\Big).
\]
Let $a(p)=\lambda_f(p)w_J(p)$ and $a(n)=\prod_{p^r || n} a(p)^r$.
Using the fact that $|\tmop{Re}(z)| \le |z|$ it follows that
\begin{align*}
\mathbb E\Big(  P_f(X)^{2k}\Big)\le \sum_{\substack{c_0<p_1, \ldots, p_k \le y \\ c_0<q_1, \ldots, q_k \le y}} \frac{a(p_1) \cdots a(p_k)a(q_1) \cdots a(q_k)}{\sqrt{p_1 \cdots p_k q_1 \cdots q_k}} \mathbb E \left( X(p_1) \cdots X(p_k) \overline{ X(q_1) \cdots X(q_k)} \right).
\end{align*}
Observe for $n$ with $\Omega(n)=k$ we have that $\sum_{p_1\cdots p_k=n} 1=k! \nu(n)$. Consequently, the sum on the right hand side above equals
\[
(k!)^2 \sum_{\substack{p|n \Rightarrow c_0<p \le y \\ \Omega(n)=k}} \frac{a(n)^2\nu(n)^2}{n}  \le (k!)^2 \sum_{\substack{p|n \Rightarrow c_0<p \le y \\ \Omega(n)=k}} \frac{a(n)^2\nu(n)}{n}  \le k! \bigg( \sum_{c_0<p \le y} \frac{\lambda_f(p)^2}{p} \bigg)^k.
\]
Combining the two estimates above completes the proof of \eqref{eq:moments2}.
\end{proof}

Before proceeding to the next lemma let us introduce some further notation.  Recall that $\nu(n)$ denotes the multiplicative function with $\nu(p^a)=\frac{1}{a!}$. We write $\nu_{j}(n)=(\nu \ast \cdots \ast \nu)(n)$ for the $j$-fold convolution of $
\nu$. Observe that $\nu_j(p^a)=\frac{j^a}{a!}$, which follows from a simple induction argument. Also, for a positive integer $\ell$ let
\begin{equation} \label{eq:nuconvdef}
\nu_{j;\ell}(n)=\sum_{\substack{n_1 \cdots n_j=n \\ \Omega(n_1), \ldots, \Omega(n_j) \le \ell}} \nu(n_1) \cdots \nu(n_j).
\end{equation}
\begin{lemma} \label{lem:mollifiermomentbd}
Let $k \in \mathbb N$ be fixed. We have for each $0\leq j\leq J$ that
\[
\mathbb E\Big(|M_{f,j}(X)|^{2k}\Big) =\prod_{p \in I_j}\left( 1+O\left(\frac{1}{p}\right)\right).
\]
\end{lemma}
\begin{proof}
We have 
\[
M_{f,j}(X)^k=\sum_{p|n \Rightarrow p \in I_j} \frac{a_{f,J}(n)\lambda(n) \nu_{k,\ell_j}(n) X(n)}{\sqrt{n}}.
\]
Consequently, since $\nu_{k,\ell_j}(n) \le \nu_{k}(n)$, using the fact that $\nu_k(p^a)=k^a/a!$ we obtain 
\begin{equation} \label{eq:mfmoment}
\mathbb E\Big(|M_{f,j}(X)|^{2k}\Big)=\sum_{p|n \Rightarrow p \in I_j} \frac{|a_{f,J}(n) \nu_{k,\ell_j}(n)|^2}{n} \le \prod_{p \in I_j} \left(\sum_{a\geq0}  \frac{ \lambda_f(p)^{2a} w_J(p)^{2a} k^{2a}}{(a!)^2 p^a}\right),
\end{equation}
and the lemma follows.
\end{proof}

\begin{lemma}\label{lem:momentbd}
We have
\begin{equation*}
\mathbb E\Big( |L(X)|^4 \Big) \ll (\log q)^{O(1)}.
\end{equation*}
\end{lemma}

\begin{remark} \label{rem:momentbd}
\emph{
By applying Cauchy-Schwarz's inequality twice, Lemmas \ref{lem:mollifiermomentbd} and \ref{lem:momentbd} give that  $\mathbb{E} (|L(X)M(X)|^2 ) \ll (\log q)^{O(1)}$.} 
\end{remark}

\begin{proof}
Define $V^\dagger(m_1,\ldots, m_8)=\prod_{i=1}^4 V\left(\frac{m_{2i-1}m_{2i}}{q^2N^2}\right)$. Then, expanding the power, we have that
\[
\begin{split}
\mathbb{E}  \Big(|L(X)|^4\Big)
&\ll \sum_{m_1,\ldots,m_8\ge 1}  \prod_{i=1}^8 \frac{\tau (m_i)}{\sqrt{m_i}} \left| \mathbb{E} \left(X(m_1m_3m_5m_7) \overline{X(m_2m_4m_6m_8)} \right) V^\dagger (m_1,\ldots,m_8)\right|\\
&\ll\sum_{n\leq q^{4+\varepsilon}}\frac{\tau_8(n)^2}{n}
\ll(\log q)^{O(1)},
\end{split}
\]
where $\tau_{\ell}(n)$ denotes the $\ell$-fold divisor function.
\end{proof}

\begin{lemma} \label{lem:largedev}
Let $V \ge 1$. Suppose $y^{V^2/9} \le q^{1/3}$. Then
\[
\# \bigg\{ \chi \pamod q : |P_f(\chi)| \ge V \sqrt{\tfrac12 \log \log y} \bigg\} \ll q e^{-V^2/9}.
\]
\end{lemma}
\begin{remark} \label{rem:random}
\emph{
By a similar argument that we will omit, we obtain for any $V\ge 1$ that
\[
\mathbb P\bigg( |P_f(X)| \ge V \sqrt{\tfrac12 \log \log y} \bigg) \ll e^{-V^2/9}.
\]}
\end{remark}
\begin{proof}
For $y^k \le q^{1/3}$, using Lemma \ref{lem:moments} and Chebyshev's inequality we have that 

\[
\begin{split}
\# \left\{ \chi \pamod q : |P_f(\chi)| \ge V \sqrt{\tfrac12 \log \log y} \right\}&\le \frac{1}{V^{2k} (\frac12 \log \log y)^k} \sum_{\chi\pamod q} P_f(\chi)^{2k} \\
&\ll q \frac{8^kk!}{ V^{2k}} \ll q \left( \frac{9 k}{e V^2}\right)^k,
\end{split}
\]
by Stirling's formula. Now take $k=\lfloor V^2/9 \rfloor$.
\end{proof}

Let $C>0$ be fixed and sufficiently large. Given $Z \ge 1$
let 
\begin{equation}\label{definitionS}
\mathcal S=\{\chi \pamod q : |P_f(\chi)|, |P_g(\chi)|  \le C Z \log \log y\}.
\end{equation}

\begin{lemma} \label{lem:matchrandom}
Assume GRH. Suppose $y^{\frac{4C^2 Z^2}{9} \log \log y} \leq q^{1/3}$. For $u,v \in \mathbb R$ with $|u|,|v| \leq Z$ we have that
\begin{equation*}
\begin{split}
&\frac{1}{\varphi^{\star}(q)} \sumstar_{\chi \in \mathcal S} L(\chi)M(\chi) \exp\Big(iuP_f(\chi)+ivP_g(\chi)\Big)\\
&\qquad\qquad =\mathbb E \left( L(X)M(X) \exp\Big(iuP_f(X)+ ivP_g(X) \Big)\right)+O\left((\log q)^{-10} \right).
\end{split}
\end{equation*}
\end{lemma}
\begin{remark}
\emph{ We assume GRH so that we may apply Proposition \ref{prop:mollified} with $k=1$.}
\end{remark}

\begin{proof}
For $s \in \mathbb C$ with $|s| \le S/e^2$ we have that
\begin{equation}\label{exptruncation}
e^s=\sum_{0\le j \le S} \frac{s^j}{j!}+O\left( e^{-S}\right).
\end{equation}
Take $S=15CZ^2 \log \log y$. Then for $\chi \in \mathcal S$ we get
\begin{equation}\label{eq:taylor}
\exp\Big(iu P_f(\chi) +ivP_g(\chi)\Big)=\sum_{0\le j\le S} \frac{i^j}{j!}\sum_{k=0}^j \binom{j}{k} u^k v^{j-k} P_f(\chi)^kP_g(\chi)^{j-k}+O\left( e^{-S}\right).
\end{equation}
Also, for $k \le j \le S$, we have
\[
\begin{split}
\sumstar_{\chi \in \mathcal{S}^{c}}  L(\chi)&M(\chi)  P_f(\chi)^kP_g(\chi)^{j-k}\\
\ll & \bigg(\ \,\sumstar_{\chi \pamod q}|L(\chi)M(\chi)|^2\bigg)^{1/2}\bigg(\, \sumstar_{\chi \in \mathcal{S}^{c}}|P_f(\chi)^k P_g(\chi)^{j-k}|^2\bigg)^{1/2}\\
\ll & q^{1/2} (\# \mathcal{S}^{c})^{1/4} \bigg(\ \,\sumstar_{\chi \pamod q} |P_f(\chi)|^{8k}\bigg)^{1/8} \bigg(\ \,\sumstar_{\chi \pamod q} |P_g(\chi)|^{8(j-k)}\bigg)^{1/8},
\end{split}
\]
where we have applied Cauchy-Schwarz's inequality and Proposition \ref{prop:mollified}. Therefore, by Lemma \ref{lem:moments} and Stirling's formula
\begin{align*}
\sumstar_{\chi \in \mathcal{S}^{c}} L(\chi)M(\chi)  P_f(\chi)^kP_g(\chi)^{j-k}
&\ll q^{3/4} (\# \mathcal{S}^{c})^{1/4}  (4 \log \log y)^{j/2} \Big((4k)! (4(j-k))!\Big)^{1/8}\\
&\ll q^{3/4} (\# \mathcal{S}^{c})^{1/4}  (16 \log \log y)^{j/2}\sqrt{k!(j-k)!}.
\end{align*}
Hence, for $k\le j\le S$ we obtain that
\begin{align} \label{eq:cs}
\sumstar_{\chi \in S} L(\chi)M(\chi)  P_f(\chi)^kP_g(\chi)^{j-k}=
&\sumstar_{\chi \pamod q} L(\chi)M(\chi)P_f(\chi)^kP_g(\chi)^{j-k}\\
&\qquad+O\left( q^{3/4} (\#\mathcal S^{c})^{1/4} (16 \log \log y)^{j/2}  \sqrt{k! (j-k)!} \right).\nonumber
\end{align}
It follows from \eqref{eq:taylor} and \eqref{eq:cs} that
\begin{equation}\label{eq:expand}
\begin{split}
& \sumstar_{\chi \in \mathcal S} L(\chi)M(\chi) \exp\Big(iuP_f(\chi)+ivP_g(\chi)\Big) \\
 &\qquad=\sum_{0\le j\le S}  \frac{i^j}{j!}\sum_{k=0}^j \binom{j}{k} u^k v^{j-k} \sumstar_{\chi \pamod q} L( \chi)M(\chi) P_f(\chi)^k P_g(\chi)^{j-k} \\
 &\qquad\qquad +O\bigg( q e^{-S} + q^{3/4} (\#\mathcal{S}^{c})^{1/4} \sum_{0\le j\le S} ( 4Z \sqrt{\log \log y})^{j}\sum_{k=0}^{j} \frac{1}{\sqrt{k!(j-k)!}}\bigg).
 \end{split}
\end{equation}

Using Lemma \ref{lem:largedev} with $V=C Z \sqrt{2 \log \log y}$ we have that
\[
\# \mathcal{S}^{c} \ll q e^{-\frac{2C^2 Z^2}{9} \log \log y}.
\]
Observe that by Cauchy-Schwarz's inequality $\sum_{k=0}^j\sqrt{\binom{j}{k}} \le 2^{j/2} \sqrt{j+1} $, where $\binom{j}{k}$ denotes the binomial coefficient. Using this estimate, along with  Stirling's formula we see that the sum in the error term in \eqref{eq:expand}  is
\[
\ll  \sum_{0\le j\le S} \frac{( 4Z \sqrt{\log \log y})^{j}}{\sqrt{j!}} 2^{j/2} \sqrt{j}\ll \exp\left( 17 Z^2 \log \log y \right).
\]
Hence using the two above estimates in \eqref{eq:expand} and applying Lemma \ref{lem:moments} leads to
\begin{equation} \label{eq:matched}
\begin{split}
&\frac{1}{\varphi^{\star}(q)} \sum_{\chi \in \mathcal S} L(\chi)M(\chi) \exp\Big(iu P_f(\chi)+ ivP_g(\chi) \Big)\\
&\qquad= \sum_{0\le j \le S} \frac{i^j}{j!}\sum_{k=0}^j \binom{j}{k} u^k v^{j-k} \mathbb E \Big( L(X)M(X) P_f(X)^k P_g(X)^{j-k} \Big)+O\left((\log q)^{-10} \right).
\end{split}
\end{equation}

Let $\mathcal{S}(X)$ denote the event corresponding to
\[
|P_f(X)|,|P_g(X)|\le C Z\log \log y
\]
and let $\mathcal{S}^{c}(X)$ denote its complement. By Remark \ref{rem:random} with $V=  C Z \sqrt{2 \log \log y}$ we obtain
$
\mathbb{P}\left (\mathcal S^c(X)\right) \ll e^{-\frac{2C^2 Z^2}{9}\log \log y}.
$
Also, analogously to \eqref{eq:cs}, applying Lemma \ref{lem:moments} we get
\[
\begin{split}
&\sum_{0\le j \le S}   \frac{i^j}{j!} \sum_{k=0}^j \binom{j}{k} u^k v^{j-k}  \mathbb E \Big( L(X) M(X) P_f(X)^k P_g(X)^{j-k} \Big)\\
&\qquad=\sum_{0\le j\le S} \frac{i^j}{j!}\sum_{k=0}^j  \binom{j}{k} u^k v^{j-k} \mathbb E  \Big( \mathbf{1}_{\mathcal{S}(X)}L(X)M(X) P_f(X)^k P_g(X)^{j-k} \Big)\\
&\qquad\qquad+O\bigg( (\log q)^{O(1)} \mathbb{P} \left(\mathcal S^c(X) \right)^{1/4} \sum_{0\le j\le S} (4Z\sqrt{\log \log y})^{j} \sum_{k=0}^j \frac{1}{\sqrt{k!(j-k)!}} \bigg)\\
&\qquad=\sum_{0\le j\le S} \frac{i^j}{j!}\sum_{k=0}^j  \binom{j}{k} u^k v^{j-k} \mathbb E  \Big( \mathbf{1}_{\mathcal{S}(X)}L(X)M(X) P_f(X)^k P_g(X)^{j-k} \Big)+O\left((\log q)^{-10} \right).
\end{split}
\]
In all outcomes in $\mathcal{S}(X)$, the analogue of \eqref{eq:taylor} holds,
\[
\exp\Big(iu P_f(X) +ivP_g(X)\Big)=\sum_{0\le j\le S} \frac{i^j}{j!}\sum_{k=0}^j \binom{j}{k} u^k v^{j-k} P_f(X)^kP_g(X)^{j-k}+O\left( e^{-S}\right),
\]
and therefore
\[
\begin{split}
&\sum_{0\le j\le S}   \frac{i^j}{j!}\sum_{k=0}^j \binom{j}{k} u^k v^{j-k} \mathbb E \Big( \mathbf{1}_{\mathcal{S}(X)} L(X) M(X) P_f(X)^k P_g(X)^{j-k} \Big)\\
&\qquad=\mathbb{E}\bigg( \mathbf{1}_{\mathcal{S}(X)} L(X)M(X) \exp\Big(iuP_f(X)+ivP_g(X)\Big)\bigg)+O\left( (\log q)^{-10}\right).
\end{split}
\]
By Cauchy-Schwarz's inequality  and Remark \ref{rem:momentbd}  we have that
\[
\begin{split}
\mathbb{E} \bigg( \mathbf{1}_{\mathcal{S}^{c}(X)} L(X)M(X) \exp\Big(iuP_f(X)+ivP_g(X)\Big)\bigg) & \ll \mathbb{E} \Big(|L(X)M(X)|^2 \Big)^{1/2}\mathbb{P} \Big(\mathcal{S}^{c}(X) \Big)^{1/4}\\
&\ll (\log q)^{-10}.
\end{split}
\]
 Combining the estimates above we have that
 \[
 \begin{split}
 &\sum_{0\le j\le S} \frac{i^j}{j!}\sum_{k=0}^j \binom{j}{k} u^k v^{j-k} \mathbb E \Big( L(X)M(X) P_f(X)^k P_g(X)^{j-k} \Big)\\
 &\qquad \qquad \qquad =\mathbb E\bigg(L(X)M(X)\exp\bigg(iu P_f(X)+ivP_g(X) \bigg) \bigg)+O((\log q)^{-10}).
 \end{split}
 \]
 Using this in \eqref{eq:matched} completes the proof.
\end{proof}

\section{A random computation} \label{sec:random}

For $\tmop{Re}(s)>1$ let
\[
\begin{split}
L^{f,g}(s,X)&=\sum_{m_1,m_2 \ge 1} \frac{\lambda_f(m_1)\lambda_g(m_2)}{(m_1m_2)^{s}} X(m_1) \overline{X(m_2)}=\prod_p \sum_{k_1,k_2 \ge 0} \frac{\lambda_{f}(p^{k_1})\lambda_g(p^{k_2})}{p^{(k_1+k_2)s}} X(p)^{k_1} \overline{X(p)^{k_2}} \\
&= \prod_p  \prod_{j=1,2}\bigg(1-\frac{\alpha_{f,j}(p) X(p)}{p^{s}} \bigg)^{-1}\bigg(1-\frac{\alpha_{g,j}(p) \overline{X(p)}}{p^{s}} \bigg)^{-1}=\prod_p L_p^{f,g}(s,X).
\end{split}
\]
Also, define
\[
L(s,X)=\frac{1}{2}\Big(L^{f,g}(s,X)+\varepsilon(f)\varepsilon(g) L^{g,f}(s,X)\Big).
\]

Our main results of this section are the following propositions.
\begin{proposition} \label{prop:random}
The function $\mathbb E(L(s,X) M(X))$ can be analytically continued to $\tmop{Re}(s)>0$. Moreover,
for $|u|,|v| \ll 1$  we have  that
\begin{align*}
&\mathbb E \bigg( L(X) M(X) \exp\Big(iuP_f(X)+ivP_g(X)\Big)\bigg)\\
&\qquad=\mathbb E \Big( L(\tfrac12,X) M(X) \Big) \exp\left(-\frac{(u^2+v^2)}{4}\log \log y\right)\Big(1+O(|u|+|v|)\Big)+O\left( (\log q)^{-10}\right).
\end{align*}
\end{proposition}

\begin{proposition} \label{prop:randombound}
We have that
\[
\mathbb E \Big( L(\tfrac12,X)  M(X) \Big) \asymp 1.
\]
\end{proposition}
\begin{remark}
\emph{As a consequence of Lemma \ref{lem:moments} with $j=k=0$, Proposition \ref{prop:random} with $u=v=0$ and Proposition \ref{prop:randombound} we have that
\begin{equation} \label{eq:mass}
\varphi_W^{\star}(q)=\sumstar_{\chi \pamod q} W(\chi)= \varphi^{\star}(q) \,  \mathbb E \Big( L(\tfrac12,X) M(X) \Big)+O(q (\log q)^{-10}) \asymp q.
\end{equation}}
\end{remark}

Before proceeding to the proofs
let us define the Rankin-Selberg $L$-function
\[
L(s, f \otimes g) =\zeta_N(2s)\sum_{n \ge 1} \frac{\lambda_f(n)\lambda_g(n)}{n^s}, \qquad \tmop{Re}(s)>1,
\]
where $\zeta_N(s)=\prod_{p} (1-\frac{\psi_0(p)}{p^s})^{-1}$ and $\psi_0$ denotes the principal character modulo $N$.
We write $L(s,f \otimes g)=\prod_p L_p(s,f \otimes g)$. Let $$\mathfrak C=N^2(|s+\kappa|+1)^2(|s|+1)^2$$ denote the analytic conductor of $f \otimes g$.

\subsection{Estimates for primes $p>x$}
The primes $p>x$ do not interact with our mollifier and their contribution is easy to understand. Estimating these terms precisely allows us to analytically continue $\mathbb E(L(s,X)M(X))$, since $M(X)$ is a Dirichlet polynomial with coefficients supported on integers with prime factors $ \le x$. 

\begin{lemma} \label{lem:eulerexp} 
For $\tmop{Re}(s)>-\frac12$ and $(f_1,f_2)=(f,g)$ or $(f_1,f_2)=(g,f)$ we have that
\[
\mathbb E\Big( L_p^{f_1,f_2}(s+\tfrac12,X) \Big)=\left( 1-\frac{\psi_0(p)}{p^{4s+2}}\right) L_p(2s+1, f \otimes g).
\]
\end{lemma}
\begin{remark} \label{rem:continuation}
\emph{
Using the dominated convergence theorem, we have for $\tmop{Re}(s)>\tfrac12$ and any $z \ge 2$  that
\[
\mathbb E\bigg( \prod_{p > z} L_p^{f,g}(s+\tfrac12,X) \bigg)=  \prod_{p > z} \mathbb E\bigg(L_p^{f,g}(s+\tfrac12,X) \bigg).
\]
In particular, for $\tmop{Re}(s)>\tfrac12$ Lemma \ref{lem:eulerexp} gives that
\begin{equation} \label{eq:largeprimes}
\mathbb E\bigg( \prod_{p > z} L_p^{f,g}(s+\tfrac12,X) \bigg)= \frac{L(2s+1, f \otimes g)}{\zeta_N(4s+2)}  \prod_{p \leq z}\left( 1-\frac{\psi_0(p)}{p^{4s+2}}\right)^{-1} L_p(2s+1, f \otimes g)^{-1}.
\end{equation}
This provides an analytic continuation of the left hand side to $\tmop{Re}(s)>-\frac12$. Clearly, the same applies to $ \mathbb \prod_{p >z }\mathbb E(L_p^{g,f}(s+\tfrac12,X))$. Therefore, by choosing $z=x$ we can analytically continue the function
\begin{equation*}
F(s;u,v):=\mathbb E \bigg(L(s+\tfrac12,X)M(X) \exp\Big(iu P_f(X)+iv P_g(X) \Big)\bigg)
\end{equation*}
to the half-plane $\tmop{Re}(s)>-\tfrac12$.}
\end{remark}

\begin{proof}[Proof of Lemma \ref{lem:eulerexp}]
We will only give the proof in the case $(f_1,f_2)=(f,g)$ since the argument in the other case is similar. For $\tmop{Re}(s)>-\tfrac12$ we have that
\[
\begin{split}
\mathbb E\Big(L_p^{f,g}(s+\tfrac12, X) \Big)=&\sum_{j,k\geq0} \frac{\lambda_f(p^{j})\lambda_g(p^{k})}{p^{(s+\frac12)(j+k)}} \mathbb E \left(X(p)^j \overline{X(p)^k} \right)\\
=&\sum_{j\geq 0} \frac{\lambda_f(p^{j})\lambda_g(p^{j})}{p^{(2s+1)j}}.
\end{split}
\]
The right hand side equals $(1-\frac{\psi_0(p)}{p^{4s+2}})L_p(2s+1,f \otimes g)$.
\end{proof}

In the next two lemmas we accomplish the main goal of this subsection and precisely estimate the contribution of the large primes $p>x$. The first result is a fairly standard lemma, which uses the zero-free region of Rankin-Selberg $L$-functions.
We use the notation $\Lambda_{f \otimes g}(n)$ to denote the $n$th coefficient of the Dirichlet series of $-L'/L(s,f \otimes g)$, that is,
\begin{equation} \label{eq:logdbd}
\Lambda_{f \otimes g}(p^m)= \Big( \alpha_{f,1}(p)^m+\alpha_{f,2}(p)^m\Big)\Big( \alpha_{g,1}(p)^m+\alpha_{g,2}(p)^m\Big) \log p.
\end{equation}

\begin{lemma} \label{lem:logbd}
There exists $0< c_1 <\tfrac15$ such that
for $\tmop{Re}(s) > 1-\frac{c_1}{\log \mathfrak C}$ and $ \log z\ge \log \log \mathfrak{C}$ we have that
\[
\log L(s,f \otimes g)=\sum_{n \le z} \frac{\Lambda_{ f \otimes g}(n)}{n^s \log n}+O\left(z^{c_1/\log\mathfrak{C}-1}\log z+z^{1-2c_1/\log \mathfrak{C}-\tmop{Re}(s)}(\log z)^{2}\right).
\]
\end{lemma}

\begin{proof}
By Perron's formula (see \cite[Lemma 3.12]{Titchmarsh}) we have that
\[
\begin{split}
\sum_{n \le z} \frac{\Lambda_{ f \otimes g}(n)}{n^s \log n}
=&\frac{1}{2\pi i}\int_{\sigma_1-iz}^{\sigma_1+iz}\log L(s+w,f \otimes g)z^w\frac{dw}{w}+O\left(z^{c_1/\log\mathfrak{C}-1}\log z\right),
\end{split}
\]
where we define $\sigma_1:=c_1/\log \mathfrak{C}+1/\log z>0$ and using the standard zero free region (see \cite[Proposition 2.11]{BFK}) we choose $c_1$ so that $L(s,f\otimes g) \neq 0$ for $\tmop{Re}(s) \ge 1-3c_1/\log \mathfrak C$. Shift the contour to $\tmop{Re} (w)=1-2c_1/\log\mathfrak{C}-\tmop{Re}(s)<0$. We have a simple pole at $w=0$, which contributes $\log L(s, f \otimes g)$ so that
\[
\sum_{n \le z} \frac{\Lambda_{ f \otimes g}(n)}{n^s \log n}
=\log L(s, f \otimes g)
+O\left(z^{c_1/\log\mathfrak{C}-1}\log z+ z^{1-2c_1/\log \mathfrak{C}-\tmop{Re}(s)}\log z \log \log \mathfrak{C} \right).
\]
In the error term we have used the fact that for $\tmop{Re}(s)\ge 1-2c_1/\log \mathfrak C$ we have
\begin{equation}\label{eq:logbd}
|\log L(s,f \otimes g)|\ll \log \log \mathfrak{C},
\end{equation}
which is proved by standard methods following the argument given in \cite[Theorem 11.4]{MV}, which uses the zero-free region of $L(s,f \otimes g)$ and the Ramanujan bound, which by \eqref{eq:logdbd} gives $|\Lambda_{f \otimes g}(n)| \le 4 \Lambda(n)$.
\end{proof}

\begin{lemma} \label{lem:largeprimeest}
There exists $c_2>0$ such that 
for $-\frac{c_2}{\log \mathfrak C} \le \tmop{Re}(s) \le 2$ and $z  \ge \exp(\log \mathfrak C \log \log \mathfrak C)$ we have for $(f_1,f_2)=(f,g)$ or $(f_1,f_2)=(g,f)$ that
\[
\mathbb E\bigg(\prod_{p > z} L_p^{f_1,f_2}(s+\tfrac12,X) \bigg)=1+O\left(z^{-c_2/\log \mathfrak C}(\log z)^{2}\right).
\]
\end{lemma}
\begin{proof}
Using Lemma \ref{lem:logbd}, setting $c_2=2c_1/3$, and applying Deligne's bound in \eqref{eq:logdbd} we have that
\[
\begin{split}
\sum_{p \le z} \log L_p(2s+1,f \otimes g)=&\sum_{n \le z} \frac{\Lambda_{f \otimes g}(n)}{n^{2s+1}\log n} +O\bigg( \sum_{p \le z} \sum_{j > \log z/\log p}\frac{1}{p^{(2\tmop{Re}(s)+1)j}} \bigg) \\
=&\log L(2s+1,f \otimes g)+O\left(z^{-1/2-2\tmop{Re}(s)}\right)+O\left(z^{-c_2/\log \mathfrak C}(\log z)^{2}\right),
\end{split}
\]
where we have estimated the first error term above by separately considering the contribution of the primes $\sqrt{z} < p \le z$  (so $j \ge 2$) and $p \le \sqrt{z}$.
Hence, using \eqref{eq:largeprimes}, \eqref{eq:logbd} and the elementary estimate $\prod_{p > z} ( 1-\frac{\psi_0(p)}{p^{2+4s}} )=1+O(z^{-1-4\tmop{Re}(s)})$ we obtain the lemma.
\end{proof}

\subsection{Estimates for primes $y < p \le x$}

We now analyze the case  $y< p \le x$. For primes in this range we need to  understand the interaction between the random $L$-series $L(s,X)$ and the mollifier $M(X)$. In this section we bound the contribution of these primes, which is needed when shifting contours in the proof of Proposition \ref{prop:random}. 

\begin{lemma} \label{lem:mediumprimeest}
Let $c_1$ be as in Lemma \ref{lem:logbd}. For
$ -\frac{c_1}{\log \mathfrak C} \le \tmop{Re}(s) \le 2$ and $|\tmop{Im}(s)| \le e^{\sqrt{\log q}}$ we have that
\begin{equation} \label{eq:lemform}
\mathbb E\bigg( \prod_{y< p \le x} L_p^{f,g}(s+\tfrac12,X) \prod_{j=1}^J M_j(X) \bigg)  \ll (\log \log q)^{O(1)}.
\end{equation}
\end{lemma}

    \begin{proof}
    Let $\beta_f(n)=\lambda_f(n)n^{-s}$ and $\gamma_f(n)=\lambda(n)a_{f;J}(n) \nu(n)$. Clearly, $\beta_f,\gamma_f$ are multiplicative functions. Also, a direct calculation gives for each $0 \le j \le J$ that
    \begin{align} \label{eq:expectedvalue}
  & \mathbb E \bigg( M_j(X) \prod_{p \in I_j}  L_p^{f,g}(s+\tfrac12, X) \bigg)\nonumber\\
    &\qquad=\mathbb E\bigg( \sum_{\substack{p|n_1n_2 \Rightarrow p \in I_j \\ \Omega(n_1), \Omega(n_2) \le \ell_j}} \frac{\gamma_f(n_1)\gamma_g(n_2)X(n_1) \overline{X(n_2)}}{\sqrt{n_1n_2}} \sum_{p|m_1m_2 \Rightarrow p\in I_j} \frac{\beta_f(m_1)\beta_g(m_2) X(m_1) \overline {X(m_2)}}{ \sqrt{m_1m_2}}\bigg)\nonumber \\
    &\qquad=\sum_{\substack{m_1,m_2, n_1,n_2 \\ p|m_1n_1 \Rightarrow p \in I_j \\ m_1n_1=m_2n_2 \\ \Omega(n_1), \Omega(n_2) \le \ell_j}} \frac{ \beta_f(m_1) \beta_g(m_2)\gamma_f(n_1)\gamma_g(n_2)}{m_1n_1}.
    \end{align}
   For any $r>0$ and $n, \ell \in \mathbb N$, 
    \[
    \mathbf{1}_{\Omega(n)=\ell}= \frac{1}{2\pi i} \int_{|z|=r} z^{\Omega(n)-\ell} \frac{dz}{z},
    \]
    so that for $r \neq 1$
    \[
    \mathbf{1}_{\Omega(n) \le \ell_j} =\frac{1}{2\pi i} \int_{|z|=r} z^{\Omega(n)} \frac{1-z^{-\ell_j-1}}{1-z^{-1}} \frac{dz}{z}.
    \]
    Therefore, for $1<r \le 2$ the right hand side of \eqref{eq:expectedvalue} equals
    \begin{equation}\label{eq:expanded}
    \frac{1}{(2\pi i)^2} \int_{|z|=r}\int_{|w|=r} \frac{1-z^{-\ell_j-1}}{1-z^{-1}} \frac{1-w^{-\ell_j-1}}{1-w^{-1}} \Sigma(z,w) \frac{dzdw}{zw},
    \end{equation}
    where
    \[
    \Sigma(z,w)=\sum_{\substack{m_1,m_2,n_1,n_2 \\ p|m_1n_1 \Rightarrow p \in I_j \\ m_1n_1=m_2n_2 }} \frac{z^{\Omega(n_1) }w^{\Omega(n_2)}\beta_f(m_1)\beta_g(m_2)\gamma_f(n_1)\gamma_g(n_2)}{m_1n_1},
    \]
   which can be seen to be absolutely convergent from the analysis below.

    Write $\gamma_{z,f}(n)=z^{\Omega(n)}\gamma_f(n)$ and $m=m_1n_1$ to see that 
    \begin{equation*} 
    \begin{split}
     \Sigma(z,w)=& \sum_{p|m \Rightarrow p \in I_j} \frac{(\beta_f\ast \gamma_{z,f})(m) (\beta_g\ast \gamma_{w,g})(m)}{m} \\
     =& \prod_{p \in I_j} \sum_{k\geq 0} \frac{(\beta_f\ast \gamma_{z,f})(p^k) (\beta_g\ast \gamma_{w,g})(p^k)}{p^k} \\
     =& \prod_{p \in I_j} \left(1+\frac{(p^{-s}-w_J(p)z)(p^{-s}-w_J(p)w )\lambda_f(p) \lambda_g(p)}{p} +O\left( \frac{1+p^{-4\tmop{Re}(s)}}{p^2}\right) \right).
     \end{split}
     \end{equation*}
    By Lemma \ref{lem:logbd} we have for $-c_1/\log \mathfrak C \le \tmop{Re}(s) \le 2$ and $|\tmop{Im}(s)|\le e^{\sqrt{\log q}}$ that
    $
     \sum_{p \in I_j} \frac{\lambda_f(p)\lambda_g(p)}{p^{2s+1}}=O(1),
     $
     and arguing similarly, using partial summation, we have that $ \sum_{p \in I_j} \frac{ w_J(p)\lambda_f(p)\lambda_g(p)}{p^{s+1}} =O(1)$ and $\sum_{p \in I_j} \frac{ w_J(p)^2\lambda_f(p)\lambda_g(p)}{p} =O(1)$. Consequently, we have that $|\Sigma(z,w)| \ll 1$ uniformly for $|z|,|w| \le 2$. Applying this bound in \eqref{eq:expanded} we have that the left hand side of \eqref{eq:expectedvalue} is $O(1)$, so the result follows upon applying this bound for each $1\leq j \leq J$ and noting that $J \asymp \log \log \log q$.
     \end{proof}
     
     Bounding the contribution of primes  $y < p \le x$ to $\mathbb E (L^{g,f}(s,X)M(X))$ requires a more subtle argument. Before proceeding to the proof let us recall from the definition of our mollifier that $\eta>0$ is sufficiently small and $\eta \le \theta_J \le e \eta$.
     
     \begin{lemma} \label{lem:mediumprimeestgf}
     Let $\sigma_0=\max\{ -\tmop{Re}(s),\frac{1}{\log q}\}$ and $c_1$ be as in Lemma \ref{lem:logbd}.
     For $\frac{-c_1}{ \log \mathfrak C} \le \tmop{Re}(s) \le 2$ and $|\tmop{Im}(s)|\le e^{\sqrt{\log q}}$ we have that
   \begin{equation} \label{eq:lemform2}
\mathbb E\bigg( \prod_{y< p \le x} L_p^{g,f}(s+\tfrac12,X) \prod_{j=1}^J M_j(X) \bigg)  \ll (\log \log q)^{O(1)} q^{21 \eta^{1/4} \sigma_0}.
\end{equation}
     \end{lemma}
     \begin{proof}
     We argue as in the previous proof and as before write $\beta_f(n)=\lambda_f(n)n^{-s}$, $\gamma_f(n)=\lambda(n)a_{f;J}(n) \nu(n)$ and $\gamma_{z,f}(n)=z^{\Omega(n)}\gamma_f(n)$. Repeating the argument leading up through \eqref{eq:expanded} we get for each $j=1,\ldots, J$ that 
     \begin{equation}\label{eq:expanded2}
     \mathbb E\bigg(M_j(X) \prod_{p \in I_j} L_p^{g,f}(s+\tfrac12,X) \bigg)=    \frac{1}{(2\pi i)^2} \int_{|z|=2}\int_{|w|=2} \frac{1-z^{-\ell_j-1}}{1-z^{-1}} \frac{1-w^{-\ell_j-1}}{1-w^{-1}} \widetilde \Sigma(z,w) \frac{dzdw}{zw},
     \end{equation}
    where
    \[
    \begin{split}
   & \widetilde \Sigma(z,w)=\sum_{\substack{m_1,m_2,n_1,n_2 \\ p|m_1n_1 \Rightarrow p \in I_j \\ m_1n_1=m_2n_2 }} \frac{z^{\Omega(n_1) }w^{\Omega(n_2)}\beta_g(m_1)\beta_f(m_2)\gamma_f(n_1)\gamma_g(n_2)}{m_1n_1}\\
    &= \prod_{p \in I_j} \sum_{k\geq 0} \frac{(\beta_g\ast \gamma_{z,f})(p^k) (\beta_f\ast \gamma_{w,g})(p^k)}{p^k}\\
    &=\prod_{p \in I_j} \left(1+\frac{(p^{-s}\lambda_g(p)-w_J(p)z \lambda_f(p))(p^{-s}\lambda_f(p)-w_J(p)w \lambda_g(p) )}{p} +O\left( \frac{1+p^{-4\tmop{Re}(s)}}{p^2}\right) \right).
    \end{split}
    \]
    Again, arguing as in the proof of Lemma \ref{lem:mediumprimeest}, using Lemma \ref{lem:logbd} we can write
    \[
    \widetilde \Sigma(z,w)=\exp\bigg(-z\sum_{p \in I_j} \frac{\lambda_f(p)^2 w_J(p)}{p^{1+s}}-w\sum_{p \in I_j} \frac{\lambda_g(p)^2 w_J(p)}{p^{1+s}} \bigg)H(z,w)
    \]
    where $H(z,w)$ is an analytic function for $|z|,|w| \le 2$ with $|H(z,w)| \ll 1$ uniformly for $s$ satisfying the hypotheses of the lemma.

     We now split the proof into two cases. First consider the case, $\theta_j < \frac{1}{\sigma_0 \log q}$. Then we have by Deligne's bound that     \begin{equation} \label{eq:smallj}
     \sum_{p \in I_j} \frac{\lambda_f(p)^2 w_J(p)}{p^{1+s}} \ll e^{\sigma_0 \theta_j \log q} \sum_{p \in I_j} \frac{1}{p} \ll 1,
     \end{equation}
     so that $\widetilde \Sigma(z,w) \ll1$ in this case and the left hand side of \eqref{eq:expanded2} is $\ll 1$.
    
    It remains to consider the case that $\theta_j \ge \frac{1}{\sigma_0 \log q}$. Using that $\widetilde \Sigma(z,w)$ is analytic in each variable, expanding $(1-z^{-\ell_j-1})/(1-z),(1-w^{-\ell_j-1})/(1-w)$ as geometric series and using Cauchy's integral formula we get that the left hand side of \eqref{eq:expanded2} equals
    \[
    \sum_{0 \le k_1,k_2 \le \ell_j} \frac{1}{k_1!k_2!} \frac{\partial^{k_1+k_2}}{\partial z^{k_1} \partial w^{k_2}} \widetilde \Sigma(z,w) \bigg|_{(z,w)=(0,0)}.
    \]
   Using Cauchy's integral formula once again we see that $\frac{\partial^{k_1+k_2}}{\partial z^{k_1} \partial w^{k_2}} H(z,w) |_{(z,w)=(0,0)}\ll k_1! k_2! 2^{-k_1-k_2}$, hence we can bound the above expression by 
    \[
    \ll 2^{2\ell_j} \bigg( \sum_{p \in I_j} \frac{\lambda_g(p)^2w_J(p) }{p^{1-\sigma_0}} \bigg)^{\ell_j} \bigg( \sum_{p \in I_j} \frac{\lambda_f(p)^2w_J(p)}{p^{1-\sigma_0}} \bigg)^{\ell_j},
    \]
    where the term $2^{2\ell_j}$ comes from applying the product rule.
    Using Deligne's bound and that $\ell_j \le 2\theta_j^{-3/4}$ we conclude that the left hand side of \eqref{eq:expanded2} is 
    \begin{equation} \label{eq:largej}
    \ll \bigg(8 \sum_{p \in I_j} \frac{1}{p^{1-\sigma_0}}\bigg)^{2\ell_j} \le 3^{4\ell_j} e^{4 \theta_j^{1/4} \sigma_0 \log q}.
    \end{equation}
    Write $J_1$ for the smallest $j=1,\ldots,J$ such that $\theta_j <\frac{1}{\sigma_0 \log q} $. Using \eqref{eq:smallj}, \eqref{eq:largej} and that in this case $\ell_j \le 2 (\sigma_0 \log q)^{3/4}$ we have that the left hand side of \eqref{eq:lemform2} is
    \[
    (\log \log q)^{O(1)} \prod_{J_1 \le j \le J}  3^{4\ell_j} e^{4 \theta_j^{1/4} \sigma_0 \log q} \ll (\log \log q)^{O(1)} e^{O((\sigma_0 \log q)^{3/4})}\,  e^{20 \eta^{1/4} \sigma_0 \log q}.
    \]
     \end{proof}

\subsection{The contribution of the small primes}
It remains to understand the contribution of the primes with $p \le y$. This involves understanding the interaction between $L^{f,g}(s,X)$, $M_0(X)$ and $e^{iuP_f(X)+ivP_g(X)}$. A key point is that since $M_0(X)$ consists of relatively small primes we can express it in terms of an Euler product with negligible loss since $\ell_0$ is large. This allows us to simplify our later analysis by reducing the problem to understanding the contribution from each prime $p \in I_0$ individually. Let 
     \[
    \widetilde M_0(X)=\sum_{\substack{p|n \Rightarrow p \in I_0}} \frac{\lambda(n)}{\sqrt{n}} (X a_{f,J} \nu \ast \overline X a_{g,J}\nu)(n)= \prod_{p \in I_0} \widetilde M_p(X),
    \]
    where
    \[
    \widetilde M_p(X)=\widetilde M_{p,f}(X) \widetilde M_{p,g}(\overline X), \qquad \widetilde M_{p,f}(X)=\sum_{k\geq0} \frac{(-1)^ka_{f,J}(p)^k}{k!p^{k/2}} X(p)^k.
    \]

    \begin{lemma} \label{lem:multiplicative} 
    For  $\tmop{Re}(s) \ge -\frac{( \log \log q)^2}{\log q}$, uniformly for $u,v \in \mathbb R$ we have for $(f_1,f_2)=(f,g)$ or $(f_1,f_2)=(g,f)$ that
    \[
    \begin{split}
    &\mathbb E\bigg( M_0(X) \exp\Big(iuP_f(X)+ivP_g(X)\Big)\prod_{p\in I_0}L_p^{f_1,f_2}(s+\tfrac12,X)  \bigg)\\
    &\qquad\qquad= \mathbb E\bigg(\widetilde M_0(X) \exp\Big(iuP_f(X)+ivP_g(X)\Big)\prod_{p\in I_0}L_p^{f_1,
    f_2}(s+\tfrac12,X)  \bigg)+O\left( (\log q)^{-10}\right).
    \end{split}
    \]

    \end{lemma}
    \begin{proof}
    We will only give the proof in the case $(f_1,f_2)=(f,g)$ since the arguments in both cases are similar.
    Recall that $\beta_f(n)=\lambda_f(n)n^{-s}$, $\gamma_f(n)=\lambda(n)a_{f,J}(n) \nu(n)$ and define
    \[
    R(X):=\widetilde M_0(X)-M_0(X)=\sum_{\substack{p|mn \Rightarrow p \in I_0 \\ \max\{ \Omega(m),\Omega(n)\} > \ell_0}}  \frac{\gamma_f(m)\gamma_g(n)}{\sqrt{mn}} X(m)\overline {X(n)}.
    \]
    Also, write $L_0^{f,g}(s,X)=\prod_{p \in I_0} L_p^{f,g}(s,X)$. Since $L_0^{f,g}(s,X)$ is a finite product this function is analytic for $\tmop{Re}(s)>0$.
    Applying Cauchy-Schwarz's inequality we have
    \[
\mathbb E\Big( \left|L_0^{f,g}(s+\tfrac12,X) R(X) \right|\Big)^2
   \le    \mathbb E \Big(|L_0^{f,
   g}(s+\tfrac12,X) |^2 \Big) \mathbb E\Big(|R(X)|^2\Big).
    \]
    
    Let us first analyze $\mathbb E(|L_0^{f,g}(s+\tfrac12,X)|^2)$. Write $(\beta_fX)(n):=\beta_f(n)X(n)$. We have that
    \begin{equation}\label{eq:momentrandombd}
    \begin{split}
    \mathbb E \Big(| L_0^{f,g}(s+\tfrac12,X)|^2\Big)=& \prod_{p \in I_0} \mathbb E \left( \bigg|\sum_{a \ge 0} \frac{\lambda_f(p^a)}{p^{a(s+1/2)}} X(p^a) \bigg|^2 \bigg|\sum_{a \ge 0} \frac{\lambda_g(p^a)}{p^{a(s+1/2)}} \overline X(p^a) \bigg|^2   \right) \\
    =&\prod_{p \in I_0} \mathbb E \left(\sum_{a\geq0}  \frac{(\beta_f X \ast \overline {\beta_g}X  \ast \overline {\beta_f X} \ast \beta_g\overline{ X )(p^a)} }{p^{a/2}}\right).
    \end{split}
    \end{equation}
    Let $\varrho(X;p^a)=(\beta_f X \ast \overline {\beta_g}X  \ast \overline {\beta_f X} \ast \beta_g\overline{ X} )(p^a)$ and observe that
    \[
    \mathbb E( \varrho(X;p))=0 \qquad \text{ and } \qquad  \mathbb E (\varrho(X;p^2))=O(p^{-2\tmop{Re}(s)}).
    \]
    We conclude that the left hand side of \eqref{eq:momentrandombd} equals
    \begin{equation}\label{eq:momentrandombd2}
    = \prod_{p \in I_0} \left(1+O\left(\frac{1}{p^{1+2\tmop{Re}(s)}} \right) \right).
    \end{equation}
    Using that $\tmop{Re}(s) \ge -\frac{(\log \log q)^2}{\log q}$ we have $\frac{1}{p^{1+2\tmop{Re}(s)}} \ll \frac{1}{p}$ for $p \in I_0$, so that
    the right hand side is $\ll (\log q)^{O(1)}$.

    We next estimate $\mathbb E( |R(X)|^2)$, which equals
    \[
    \sum_{\substack{p|m_1m_2n_1n_2 \Rightarrow p \in I_0 \\ \max\{ \Omega(m_1),\Omega(n_1)\} > \ell_0 \\ \max\{\Omega(m_2),\Omega(n_2)\} > \ell_0}} \frac{\gamma_f(m_1)\gamma_f(m_2) \gamma_g(n_1)\gamma_g(n_2)}{\sqrt{m_1m_2n_1n_2}} \mathbb E(X(m_1n_1)\overline {X(m_2n_2)})
    \]
  Therefore, using that $\max\{\Omega(m_1),\Omega(n_1)\} > \ell_0$ implies $2^{\Omega(m_1 n_1)-\ell_0} \ge 1$, writing $r=m_1n_1=m_2n_2$, we get upon applying Deligne's bound  $|a_{f,J}(n)|  \le 2^{\Omega(n)}$ that there exists $C>0$ such that 

    \begin{equation} \label{eq:Rerrorbd}
    \mathbb E\Big(|R(X)|^2\Big) \le \frac{1}{2^{\ell_0}} \sum_{p|r \Rightarrow p \in I_0} \frac{2^{\Omega(r)} (2^{\Omega} \ast 2^{\Omega})(r)^2}{r} \le \frac{1}{2^{\ell_0}} \sum_{p|r \Rightarrow p \in I_0 } \frac{C^{\Omega(r)}}{r} \ll \frac{(\log q)^{O(1)}}{2^{\ell_0}},
    \end{equation}
        where we have used that $c_0$ is sufficiently large so that the sum converges.
    Combining the two estimates above completes the proof.
     \end{proof}
     
     Now that we have replaced $M_0(X)$ with the Euler product $\widetilde M_0(X)$ our analysis reduces to estimates for each prime $p \in I_0$.
    Before continuing further let us introduce some notation. Let
    \[
    \mathcal G_p^{f,g}(s)=\mathbb E \left(L_p^{f,g}(s+\tfrac12, X) \widetilde M_p(X) \right)
    \]
    and
    \[
    \mathcal F_p^{f,g}(s)=\mathbb E \left(L_p^{f,g}(s+\tfrac12, X) \widetilde M_p(X) \tmop{Re}(X(p))\right).
    \]
    
    \begin{lemma} \label{lem:gfest} Let $a\in\mathbb{Z}$. Suppose that $p^{-\tmop{Re}(s)} \le 2$. Then the following statements hold:
    \begin{enumerate}
        \item $\displaystyle\mathcal G_p^{f,g}(s)=1+\frac{(p^{-s}-w_J(p))^2\lambda_f(p)\lambda_g(p)}{p} +O\left( \frac{1}{p^2}\right)$;
        \item  $\displaystyle\mathcal G_p^{g,f}(s)=1+\frac{(p^{-s}\lambda_f(p)-w_J(p)\lambda_g(p))(p^{-s}\lambda_g(p)-w_J(p)\lambda_f(p))}{p} +O\left( \frac{1}{p^2}\right)$.
        \end{enumerate}
        Additionally, for $(f_1,f_2)=(f,g)$ or $(f_1,f_2)=(g,f)$ we have each of the following:
        \begin{enumerate}
        \item[(3)] $\displaystyle\mathcal F_p^{f_1,f_2}(s)=\frac{(p^{-s}-w_J(p))(\lambda_f(p)+\lambda_g(p))}{2\sqrt{p}}+O\left(\frac{1}{p}\right)$;
       \item[(4)] $\displaystyle\mathbb E\left( L_p^{f_1,f_2}(s+\tfrac12, X) \widetilde M_p(X) X(p)^a \right) \ll \frac{ 5^{|a|} }{p^{|a|/2}}$.
    \end{enumerate}
    \end{lemma}
    \begin{proof}
    Let $h_1,h_2$ be level $N$ newforms (not necessarily distinct).
    For $\tmop{Re}(s)>0$ and $k$
an integer let
\[
n_p^{h_1,h_2}(s,k)=\sum_{\substack{k_1+k_2=k \\ k_1, k_2  \ge 0}} \frac{\lambda_{h_1}(p^{k_1}) a_{h_2,J}(p)^{k_2} (-1)^{k_2}}{p^{k_1 s+k_2/2}k_2!}.
\]
Observe that $n_p^{h_1,h_2}(s,k)=0$ if $k<0$ and
\[
\bigg(\sum_{k\ge 0} \frac{\lambda_{h_1}(p^{k} ) X(p)^{k}}{p^{k s}} \bigg) \,  \widetilde M_{p,h_2}(X) =\sum_{k \ge 0} n_p^{h_1,h_2}(s,k) X(p)^k.
\]
For any integer $a$
we have that
    \begin{align} \label{eq:expectation1}
    \mathbb E\left(L_p^{f,g}(s+\tfrac12, X) \widetilde M_p(X) X(p)^a\right)=&\sum_{k_1,k_2 \ge 0} n_p^{f,f}(s+\tfrac12,k_1) n_p^{g,g}(s+\tfrac12,k_2) \mathbb E\left( X(p)^{k_1+a} \overline {X(p)^{k_2}} \right)  \nonumber \\
    =& \sum_{k\geq 0} n_p^{f,f}(s+\tfrac12,k)n_p^{g,g}(s+\tfrac12,k+a).
    \end{align}
Additionally,
    \begin{align} \label{eq:expectationgf}
    \mathbb E\left(L_p^{g,f}(s+\tfrac12, X) \widetilde M_p(X) X(p)^a\right)=&\sum_{k_1,k_2 \ge 0} n_p^{g,f}(s+\tfrac12,k_1) n_p^{f,g}(s+\tfrac12,k_2) \mathbb E\left( X(p)^{k_1+a} \overline {X(p)^{k_2}} \right)  \nonumber \\
    =& \sum_{k\geq 0} n_p^{g,f}(s+\tfrac12,k)n_p^{f,g}(s+\tfrac12,k+a).
    \end{align}

    Clearly, $n_p^{h_1,h_2}(s+\tfrac12,0)=1$ and 
    \[
    n_p^{h_1,h_2}(s+\tfrac12,1)=\frac{(p^{-s}\lambda_{h_1}(p)-w_J(p) \lambda_{h_2}(p))}{\sqrt{p}}.
    \]
    Additionally, for each $k \ge 1$ using Deligne's bound we get that
    \[
    n_p^{h_1,h_2}(s+\tfrac12,k) \ll \frac{ (k+1)^2 2^k }{p^{k/2}} \max\{1,p^{-\tmop{Re}(s)k}\}.
    \]
    Using the above two estimates in \eqref{eq:expectation1} and \eqref{eq:expectationgf} all the claims follow (here we have also used that for a complex number $w$ with $w\overline w=1$ that $\overline w^a=w^{-a}$).
    \end{proof}
    
   \begin{lemma} \label{lem:smallprimeest}
   Let $u,v \in \mathbb R$ with $|u|,|v|\ll 1$.
  For $\tmop{Re}(s) \ge -\frac{( \log \log q)^2}{\log q}$ and $p \in I_0$ we have for $(f_1,f_2)=(f,g)$ or $(f_1,f_2)=(g,f)$ that
   \begin{align} \label{eq:taylorform}
   &\mathbb E \left( L_p^{f_1,f_2}(s+\tfrac12, X) \widetilde M_p(X) \exp\left(i u  \frac{\lambda_f(p)w_J(p)}{\sqrt{p}} \tmop{Re}(X(p))+iv \frac{\lambda_g(p)w_J(p)}{\sqrt{p}} \tmop{Re}(X(p))  \right)\right) \\
   &=\mathcal G_p^{f_1,f_2}(s)\left(1- \frac{(u\lambda_f(p)+v \lambda_g(p))^2w_J(p)^2}{4p} \right) +i \mathcal F_p^{f_1,f_2}(s) \frac{(u \lambda_f(p)+v \lambda_g(p))w_J(p)}{\sqrt{p}} +O\left(\frac{|u|^3+|v|^3}{p^{3/2}} \right).\nonumber
   \end{align}
   \end{lemma}
   
   \begin{remark} \label{rem:combine} \emph{Combining Lemmas \ref{lem:gfest} and \ref{lem:smallprimeest} we get for $\tmop{Re}(s)\ge -\frac{(\log \log q)^2}{\log q}$ and $p \in I_0$ for $|u|,|v|\ll 1$ that the left hand side of \eqref{eq:taylorform} is
$
=1+O(\frac{1}{p}),
$
  for $(f_1,f_2)=(f,g)$ or $(f_1,f_2)=(g,f)$. Hence, combining this with Lemma \ref{lem:multiplicative} we have for $(f_1,f_2)=(f,g)$ or $(f_1,f_2)=(g,f)$ that 
  \begin{equation} \notag
\mathbb E\bigg( M_0(X) \exp\Big(iuP_f(X)+ivP_g(X)\Big)\prod_{p\in I_0}L_p^{f_1,f_2}(s+\tfrac12,X)  \bigg) \ll (\log q)^{O(1)},  
\end{equation}
  for $\tmop{Re}(s) \ge -\frac{(\log \log q)^2}{\log q}$. 
  Hence, for $ -\frac{(\log \log q)^2}{\log q} \le \tmop{Re}(s)  \le 2$ using this estimate with Lemmas \ref{lem:largeprimeest}, \ref{lem:mediumprimeest}, \ref{lem:mediumprimeestgf} gives
  \begin{equation} \label{eq:fupperbd1}
  \mathbb E\bigg( L(s+\tfrac12,X) M(X) \exp\Big(iuP_f(X)+ivP_g(X)\Big)  \bigg) \ll (\log q)^{O(1)}+ (\log q)^{O(1)}e^{21 \eta^{1/4} (\log \log q)^2}.
  \end{equation}
  }
    \end{remark}
   
   \begin{proof}[Proof of Lemma \ref{lem:smallprimeest}]
    Taylor expanding gives that the exponential on the left hand side of \eqref{eq:taylorform} equals
   \[
    \begin{split}
=\sum_{k\geq0} \frac{i^k (u\lambda_f(p)+v\lambda_g(p))^kw_J(p)^k}{k!p^{k/2}} \tmop{Re}(X(p))^k.
     \end{split}
   \]
   The terms with $k=0,1,2$ account for the main terms upon noting
    that Lemma \ref{lem:gfest} implies
    \[
    \mathbb E \left( L_p^{f_1,f_2}(s+\tfrac12, X) \widetilde M_p(X) \tmop{Re}(X(p))^2\right)=\frac{\mathcal G_p^{f_1,f_2}(s)}{2} +O\left(\frac{1}{p} \right),
    \]
    for $p \in I_0$.
    To complete the proof, use Lemma \ref{lem:gfest} to bound the contribution to the left hand side of \eqref{eq:taylorform} from the terms with $k \ge 3$ in the Taylor expansion above, since $c_0$ is sufficiently large.

   \end{proof}
   
 \subsection{Proof of Proposition \ref{prop:random}}
 
 We are now ready to complete the proof of Proposition \ref{prop:random}.
 
 \begin{proof}[Proof of Proposition \ref{prop:random}]
 For $u,v \in \mathbb R$ and $\tmop{Re}(s)>\tfrac12 $ define
  \begin{equation} \notag
   F(s;u,v)=\mathbb E \bigg(L(s+\tfrac12,X)M(X) \exp\Big(iuP_f(X)+ivP_g(X) \Big)\bigg).
   \end{equation}
   Recall that in Remark \ref{rem:continuation} we saw that $F(s;u,v)$
   admits an analytic continuation to $\tmop{Re}(s)>-\tfrac12$. Moreover, by \eqref{eq:fupperbd1} we have that 
     \begin{equation} \label{eq:fbd1}
   |F(s;u,v)| \ll e^{22 \eta^{1/4} (\log \log q)^2}
   \end{equation}
   uniformly for $|u|,|v|\ll 1$, $ - \frac{(\log \log q)^2}{ \log q} \le \tmop{Re}(s) \le 2$, $|\tmop{Im}(s)| \le B \log q$ for any fixed $B>0$. Additionally, in \eqref{eq:momentrandombd2} we showed that $\mathbb E(|L_p^{f,g}(s,X)|^2)=1+O(p^{-2\tmop{Re}(s)})$ hence it is not hard to see that for $\tmop{Re}(s) \ge \tfrac12+\varepsilon$ we have  $\mathbb E\left( |L^{f,g}(s,X)|^2\right)=O_{\varepsilon}(1)$ and by repeating this argument we also have that $\mathbb E\left( |L^{g,f}(s,X)|^2\right)=O_{\varepsilon}(1)$ in the same range. Also, by Lemma \ref{lem:mollifiermomentbd} we have
   $\mathbb E(|M(X)|^2) \ll (\log q)^{O(1)}$. Applying these estimates along with Cauchy-Schwarz's inequality we have that 
   \begin{equation}\label{eq:fbd2}
   |F(s;u,v)| \ll (\log q)^{O(1)}\
   \end{equation}
   in the region $\tmop{Re}(s) \ge \tfrac12+\varepsilon$. 
   
   Applying Mellin inversion we see that
   \begin{equation} \label{eq:mellin}
   \begin{split}
   &\mathbb E \bigg(L(X)M(X) \exp\Big( iu P_f(X)+iv P_g(X) \Big)\bigg) \\
   &\qquad  = \frac{1}{2\pi i} \int_{(2)}   \frac{L_{\infty}(s+\tfrac12,f)L_{\infty}(s+\tfrac12,g)}{L_{\infty}(\tfrac12,f)L_{\infty}(\tfrac12,g)}  \left( q^2 N^2 \right)^sF(s;u,v) \,  \frac{(\cos( \frac{\pi s}{12}))^{-48}}{s}  \, ds.
   \end{split}
   \end{equation}
   Since for fixed $\sigma>0$ Stirling's formula gives that $|\Gamma(\sigma+it)|\ll (|t|+1)^{\sigma-\frac12} e^{-\pi|t|}$, by \eqref{eq:fbd2} we may truncate the integral in \eqref{eq:mellin} to $|\tmop{Im}(s) |\le B \log q$ at the cost of an error term of size $O(q^{-1})$ where $B$ is a sufficiently large absolute constant. We now shift contours to $\tmop{Re}(s)=-\frac{(\log \log q)^2}{\log q}$,  pick up a simple pole at $s=0$, estimate the horizontal and left contours using the bound \eqref{eq:fbd1} to get that
   \[
   \begin{split}
  & \mathbb E \bigg(L(X)M(X) \exp\Big( iuP_f(X)+ivP_g(X) \Big)\bigg)\\
   & \qquad \qquad =\mathbb E \bigg(L(\tfrac12,X)M(X) \exp\Big( iuP_f(X)+ivP_g(X) \Big)\bigg)+O\left((\log q)^{-10}\right),
   \end{split}
   \]
   since $\eta>0$ is sufficiently small.

   For $(f_1,f_2)=(f,g)$ or $(f_1,f_2)=(g,f)$
   by Lemma \ref{lem:gfest},  $|\mathcal F_p^{f_1,f_2}(0)| \ll \frac{\log p}{\sqrt{p} \log x}+\frac{1}{p}$ and $\mathcal G_p^{f_1,f_2}(0)=1+O(1/p)$. Hence, using Lemmas  \ref{lem:largeprimeest},
   \ref{lem:multiplicative}, and \ref{lem:smallprimeest} we have that
   \[
   \begin{split}
   &\mathbb E \bigg(L^{f_1,f_2}(\tfrac12,X)M(X) \exp\Big( iuP_f(X)+ivP_g(X) \Big)\bigg) \\
   & = \prod_{p  \in I_0}\left(\mathcal G_p^{f_1,f_2}(0)\left(1- \frac{(u\lambda_f(p)+v \lambda_g(p))^2w_J(p)^2}{4p} \right)+O\left( (|u|+|v|) \Big(\frac{ \log p}{p \log x} +\frac{1}{p^{3/2}}\Big)\right) \right)\\
   & \qquad \qquad\qquad \times \prod_{j=1}^J \mathbb E\bigg( M_j(X) \prod_{p \in I_j} L_p^{f_1,f_2}(\tfrac12, X)\bigg) \prod_{p \le c_0} \mathbb E\Big(L_p^{f_1,
   f_2}(\tfrac12,X)\Big)+O\left( (\log q)^{-10}\right) \\
  & = \mathbb E \Big( L^{f_1,f_2}(\tfrac12, X)M(X) \Big) \exp\bigg(- \sum_{p \in I_0} \frac{(u\lambda_f(p)+v \lambda_g(p))^2w_J(p)^2}{4p}\bigg)\Big(1+O(|u|+|v|)\Big)\\
  &\qquad\qquad\qquad+O\left((\log q)^{-10}\right),
   \end{split}
   \]
   where we also used Lemmas \ref{lem:mediumprimeest} and \ref{lem:mediumprimeestgf} to estimate the error terms.
By the Prime Number Theorem for Rankin-Selberg $L$-functions (see \cite[Corollary 2.15]{BFK})   \begin{equation} \label{eq:pntsum}
   \sum_{p \in I_0} \frac{(u\lambda_f(p)+v \lambda_g(p))^2w_J(p)^2}{4p}=\frac{u^2+v^2}{4} \log \log y+O(|u|^2+|v|^2),
   \end{equation}
which completes the proof.   \end{proof}

\subsection{Proof of Proposition \ref{prop:randombound}}  We are now ready to prove Proposition \ref{prop:randombound}.
    \begin{proof}[Proof of Proposition \ref{prop:randombound}]
        Using \eqref{eq:expectedvalue} with $s=0$ and recalling $\gamma_f(n)=\lambda(n) a_{f,J}(n) \nu(n)$,
    we have for each $0\leq j\leq J$ that
    \[
       \mathbb E \bigg( M_j(X) \prod_{p \in I_j} L_p^{f,g}(\tfrac12, X) \bigg)= \sum_{\substack{m_1,m_2,n_1,n_2 \\ p |m_1n_1 \Rightarrow p \in I_j \\ m_1n_1=m_2n_2 \\ \Omega(n_1) , \Omega(n_2) \le \ell_j}} \frac{\lambda_f(m_1)\lambda_g(m_2)\gamma_f(n_1)\gamma_g(n_2) }{m_1 n_1}.
       \]
       We now wish to remove the condition $\Omega(n_1),\Omega(n_2) \le \ell_j$ so that we can express the sum as an Euler product. Arguing as in \eqref{eq:Rerrorbd} we see that there exists $C>0$ such that the sum on the right hand side is
       \begin{equation} \label{eq:toomanyprimes}
       \begin{split}
       =& \sum_{\substack{ p |m \Rightarrow p \in I_j  }} \frac{(\lambda_f\ast \gamma_f)(m)(\lambda_g \ast \gamma_g)(m) }{m} +O\bigg(\frac{1}{2^{\ell_j}} \sum_{p | m \Rightarrow p \in I_j} \frac{C^{\Omega(m)}}{m}\bigg) \\
       =&\prod_{p \in I_j}\left(1+\frac{(1-w_J(p))^2 \lambda_f(p) \lambda_g(p)}{p}+O\left(\frac{1}{p^2} \right) \right)+O \left( \frac{\mathbf{1}_{j=0} (\log q)^{O(1)}+1}{2^{\ell_j}}\right).
       \end{split}
       \end{equation}

   Since the product on the right hand side above is $\asymp 1$, the product over $1\leq j\leq J$ is
    \begin{equation} \label{eq:waytoomanyprimes}
    =\left(1+O\left(y^{-1}\right)+O \bigg( \sum_{j=1}^{J} \frac{1}{2^{\ell_j}}\bigg)\right)\prod_{j=1}^{J}\prod_{p\in I_j} \left(1+\frac{(1-w_J(p))^2 \lambda_f(p) \lambda_g(p)}{p} \right).
    \end{equation}
    Using that $2^{-\ell_j} \ll 1/\ell_j \asymp \theta_j^{3/4}$, and summing the geometric sum the error term is $\ll \theta_J^{3/4} \ll \eta^{3/4}$.
    Hence, combining \eqref{eq:toomanyprimes}, \eqref{eq:waytoomanyprimes}, and Lemmas \ref{lem:eulerexp}, \ref{lem:largeprimeest} we get that
    \[
    \begin{split}
    \mathbb E\Big( L^{f,g}(\tfrac12,X)M(X)\Big)&=\left(1
    +O\left(\frac{1}{c_0}\right)+O(\eta^{3/4})\right)\prod_{p \le c_0} \left(1-\frac{\psi_0(p)}{p^{2}} \right) L_p(1,f \otimes g) \\
    & \qquad\qquad  \times \prod_{c_0 < p \le x}  \left(1+\frac{(1-w_J(p))^2 \lambda_f(p) \lambda_g(p)}{p} \right).
    \end{split}
    \]
    Since $w_J(p)=1+O(\log p/\log x)$, using Mertens' Theorem and the bound $| \lambda_f(p)\lambda_g(p)| \le 4$  the product over $c_0<p \le x$ is $ \asymp 1$, hence the right hand side above is $\ll 1$. To get a lower bound, note that by
 using Deligne's bound we have $|L_p(1,f \otimes g)| \ge |1+ \frac{1}{p}|^{-4}$, so the product over $p \le c_0$ is $ \gg (\log c_0)^{- 4}$, by Mertens' Theorem, which shows that the right hand side above is $\gg 1$ by choosing $\eta$ to be sufficiently small in terms of $c_0$.
    
Using Lemma \ref{lem:largeprimeest} to estimate the contribution of the primes $p>x$ and Lemmas \ref{lem:multiplicative} and \ref{lem:gfest} to estimate that of the primes $c_0 < p \le y$ gives 
        \[
    \begin{split}
  & \mathbb E\Big( L^{g,f}(\tfrac12,X)M(X)\Big)= (1+o(1))  \prod_{p \le c_0}\mathbb E\Big( L_p^{g,f}(\tfrac12,X)\Big) \, \times \prod_{y< p \le x} \mathbb E\bigg(  L_p^{g,f}(\tfrac12,X) \prod_{j=1}^J M_j(X) \bigg) \\
   &\times   \bigg( 
\prod_{c_0 < p \le y}  \left(1+\frac{(\lambda_f(p)-w_J(p)\lambda_g(p)) (\lambda_g(p)-w_J(p)\lambda_f(p))}{p}+O\bigg(\frac{1}{p^2} \bigg)\right) +O((\log q)^{-10})\bigg).
    \end{split}
    \]
    The product over $p \le c_0$ on the right hand side above is $\ll 1$ by Lemma \ref{lem:eulerexp}. Using Lemma \ref{lem:mediumprimeestgf} the product over $y< p \le x$ is $\ll (\log \log q)^{O(1)}$.
   To estimate the product over $c_0< p \le y$ we again use the Prime Number Theorem for Rankin-Selberg $L$-functions, which implies \[
   \sum_{c_0 < p \le y} \frac{\lambda_f(p)\lambda_g(p)}{p}=O(1) \qquad  \text{ and } \qquad \sum_{c_0 < p \le y} \frac{\lambda_f(p)^2}{p}=\log \log y+O(1),
   \]
   to see that this product is $\ll (\log y)^{-2}$.
   Hence, we have that
\[
\mathbb E\Big( L^{g,f}(\tfrac12,X)M(X)\Big)=O_\varepsilon( (\log q)^{-2+\varepsilon})
\]
for any $\varepsilon>0$.
Combining the above result along with our previous estimate that $\mathbb E( L^{f,g}(\tfrac12,X)M(X)) \asymp 1$ completes the proof.
    \end{proof}

\section{Proof of Theorem \ref{thm:cltjoint} and Corollary \ref{cor:ratios}}  \label{sec:proofcltjoint}

We will first prove Theorem \ref{thm:cltjoint}. Let
\[
\Phi_q(u,v)=\frac{1}{\varphi_W^{\star}(q)}\ \sumstar_{\chi \pamod q} W(\chi)  
e\bigg( -iu \frac{P_f(\chi)}{\sqrt{\frac12 \log \log q}}-iv \frac{P_g(\chi)}{\sqrt{\frac12 \log \log q}}\bigg).
\]
Using the the main results from the previous sections, Lemma \ref{lem:matchrandom} and Proposition \ref{prop:random},
we get that by rescaling $(u,v) \rightarrow \big(\frac{-2\pi u}{\sqrt{\frac12 \log \log q}},\frac{-2\pi v}{\sqrt{\frac12 \log \log q}}\big)$ that
\begin{equation}\label{thm:mainresult2}
 \Phi_q(u,v) = e^{-2 \pi^2 (u^2+v^2)}\left(1+O\left( \frac{(|u|+|v|) (\log \log \log q)^{1/2}}{(\log \log q)^{1/2}} \right)  \right)+O\left((\log q)^{-10}\right)
\end{equation}
for $u,v\in\mathbb{R}$ with $|u|, |v| \ll \sqrt{\frac{\log \log q}{\log \log \log q}} $.  The expression above for the characteristic function is the key input into the proof of Theorem \ref{thm:cltjoint}. Before proceeding to the proof we require some additional results.

\begin{lemma} \label{lem:cheby} Assume GRH.
Let $\Lambda \ge 1$. Also, let $B>0$ be sufficiently large.
Then for all primitive characters $\chi$ modulo $q$ outside an exceptional set of size $\ll \frac{q}{\Lambda^2}+\frac{q}{(\log \log q)^{10}}$ the following statements hold:
\begin{enumerate}
    \item  $ \displaystyle |L_f(\chi)M_f(\chi)| \le \Lambda$;
    \item $\displaystyle \prod_{j=1}^J |M_{f,j}(\chi)| \le (\log \log q)^B$;
    \item $ \displaystyle \frac{1}{\sqrt{\frac12\log \log q}}\bigg| \sum_{c_0< p \le y} \frac{\lambda_f(p) w_J(p) \chi(p)}{ \sqrt{p}} \bigg|\le \log \log \log q $.
\end{enumerate}
\end{lemma}

\begin{remark}
\emph{In the proof we use GRH when applying Proposition \ref{prop:mollifiedf} to bound
\[
\sumstar_{\chi \pamod q} |L_f(\chi)M_f(\chi)|^2.
\]
Using work of Blomer et. al. \cite{BFK} such an estimate can be proved unconditionally so the assumption of GRH can be removed at the cost of a longer argument. }
\end{remark}

\begin{proof} [Proof of Lemma \ref{lem:cheby}]

By Chebyshev's inequality and Proposition \ref{prop:mollifiedf}
\[
\sumstar_{\substack{\chi \pamod q \\ |L_f(\chi)M_f(\chi)| > \Lambda} } 1  \le \frac{1}{\Lambda^2}\ \sumstar_{\chi \pamod q} |L_f(\chi)M_f(\chi)|^2 \ll \frac{q}{\Lambda^2}.
\]

Similarly,
\[
\begin{split}
\sumstar_{\substack{\chi \pamod q \\  \prod_{j=1}^J |M_{f,j}(\chi)| > (\log \log q)^B} } 1 \le & \frac{1}{(\log \log q)^{2B}} \ \sumstar_{\chi \pamod q} \prod_{j=1}^J |M_{f,j}(\chi)|^2  \\
\ll &  \frac{q}{(\log \log q)^{2B}} \prod_{y< p \le x}\left( 1+O\left( \frac{1}{p}\right)\right) \ll \frac{q}{(\log \log q)^{10}},
\end{split}
\]
where the second step follows by using \eqref{eq:trivial} and Lemma \ref{lem:mollifiermomentbd} since $B$ is sufficiently large.

Finally, we note that the argument given in the proofs of Lemmas \ref{lem:moments} and \ref{lem:largedev} shows that the conclusion of Lemma \ref{lem:largedev} holds with $P_f(\chi)$ replaced by $\sum_{c_0< p \le y} \frac{\lambda_f(p) w_J(p) \chi(p)}{ \sqrt{p}} $ (this follows immediately since in the proof of Lemma \ref{lem:moments} we used that $|\tmop{Re}(z)|\le |z|$), so that using this result with $V=\log \log \log q(\frac{\log \log q}{\log \log y})^{1/2}$ gives that
\[
\# \bigg\{ \chi \pamod q : \bigg|\sum_{c_0< p \le y} \frac{\lambda_f(p) w_J(p) \chi(p)}{ \sqrt{p}} \bigg| \ge \log \log \log q \sqrt{\frac12\log \log q} \bigg\} \ll \frac{q}{(\log \log q)^{10}}. 
\]
\end{proof}

For $f \in L^1(\mathbb R)$ we denote by $\widehat{f}$ the Fourier transform of $f$,
\[
\widehat f(\xi)=\int_{\mathbb R} f(x) e(-\xi x) \, dx.
\]
Let us quote the following result due independently to Beurling and Selberg (see \cite[Section 7]{LLR} and Vaaler \cite{vaaler}).

\begin{lemma}\label{lem:BS}
Let $\Delta>0$ and $I=[a,b] \subset \mathbb R$ be an interval. Then there exists an entire function $F_{I,\Delta}(z)$ such that each of the following holds:
\begin{enumerate}
    \item $\displaystyle 0 \le \mathbf{1}_{I}(x)- F_{I,\Delta}(x) \le \left( \frac{\sin(\pi \Delta(x-a)) }{\pi \Delta (x-a)} \right)^2+\left( \frac{\sin(\pi \Delta(b-x)) }{\pi \Delta (b-x)} \right)^2$, $\forall x \in \mathbb R$;
    \item $ \displaystyle \widehat {F_{I,\Delta}}(\xi) =\begin{cases}
    \widehat {\mathbf{1}_{I}}(\xi)+O\left(\frac{1}{\Delta} \right) & \text{ if } |\xi| < \Delta, \xi \in \mathbb R, \\
    0 & \text{ if } |\xi| \ge \Delta, \xi \in \mathbb R.
    \end{cases}$
\end{enumerate}
\end{lemma}

 We also need to establish an unweighted analogue of \eqref{thm:mainresult2}, which is much easier to prove. 
\begin{lemma} \label{lem:clteasy}
For $u,v\in\mathbb{R}$ with $|u|, |v| \ll \sqrt{\frac{\log \log q}{\log \log \log q}} $ we have that
\begin{align*}
\Psi_q(u,v)&:=\frac{1}{\varphi^{\star}(q)} \sumstar_{\chi \pamod q} \exp\bigg(-2\pi iu \frac{P_f(\chi)}{\sqrt{\tfrac12\log \log q}}-2\pi iv \frac{P_g(\chi)}{\sqrt{\tfrac12\log \log q}}\bigg)\\
&=e^{-2\pi^2(u^2+v^2)} \left(1+O\left( \frac{(|u|+|v|) (\log \log \log q)^{1/2}}{(\log \log q)^{1/2}} \right)  \right)+O\left((\log q)^{-10}\right).
\end{align*}
\end{lemma}
\begin{proof}
A straightforward line-by-line modification of the proof of Lemma \ref{lem:matchrandom} gives that
\[
\frac{1}{\varphi^{\star}(q)} \sumstar_{\chi \pamod q} \exp\Big(iu P_f(\chi)+iv P_g(\chi)\Big)=\mathbb E\bigg( \exp\Big(iu P_f(X)+iv P_g(X)\Big)\bigg)+O\left((\log q)^{-10}\right),
\]
for $|u|,|v| \ll 1$.
Taylor expanding and using \eqref{eq:pntsum}, the main term equals
\[
\begin{split}
&\prod_{p \le y} \mathbb E \left(\exp\left(iu \frac{ \lambda_f(p) w_J(p)\tmop{Re}(X(p))}{\sqrt{p}}+iv \frac{ \lambda_g(p)w_J(p)\tmop{Re}( X(p))}{\sqrt{p}}\right) \right)\\
&\qquad=\prod_{p \le y} \left(1-\frac{(u \lambda_f(p)+v \lambda_g(p))^2 w_J(p)^2 }{4p}+O\left(\frac{|u|^3+|v|^3}{p^{3/2}} \right)\right) \\
&\qquad=(\log y)^{-\frac{(u^2+v^2)}{4}} \Big(1+O\left(|u|+|v| \right) \Big),
\end{split}
\]
and the lemma follows upon rescaling as in \eqref{thm:mainresult2}.
\end{proof}

\begin{proof}[Proof of Theorem \ref{thm:cltjoint}]
Let $\mathcal{S}_{f}$ denote the set of primitive characters $\chi$ modulo $q$ satisfying the properties in Lemma \ref{lem:cheby} 
as well as
\[
|L_f(\chi)M_f(\chi)|\ge \frac{1}{\Lambda}
\]
with $\Lambda=\log \log q$. Then for $\chi \in \mathcal{S}_{f}$ we have that
\[
\begin{split}
\log|L_f(\chi)|=&- \log |M_f(\chi)|+O(\log \Lambda) \\
=& -\log |M_{f,0}(\chi)|+O(\log \log \log q).
\end{split}
\]
Also, for any complex number $|s| \le S/e^2$ we have from \eqref{exptruncation} that
\[
\sum_{j=0}^S \frac{s^j}{j!}=e^{s}\Big(1+O\left( e^{-S/2}\right)\Big),
\]
which implies for $\chi \in \mathcal{S}_{f}$ that
\[
M_{f,0}(\chi)=\exp\bigg( -\sum_{c_0 < p \le y} \frac{\lambda_f(p)w_J(p)}{\sqrt{p}} \chi(p) \bigg)\bigg(1+O\left((\log \log q)^{-10}\right)\bigg).
\]
We conclude that for  $\chi \in \mathcal{S}_{f}$,
\begin{equation} \label{eq:clean}
P_f(\chi)=\log |L_f(\chi)|+O(\log \log \log q).
\end{equation}
In particular, this implies that for $\chi \in \mathcal S_f$,
\begin{equation} \label{eq:lfbound}
\frac{|\log |L_f(\chi)|| }{\sqrt{\frac12 \log \log q}}\ll \log \log \log q.
\end{equation}

Let $\mathcal L_f(\chi)= \frac{\log |L_f(\chi)|}{\sqrt{\frac12 \log \log q}}$, $\mathcal P_f(\chi)=\frac{P_f(\chi)}{\sqrt{\frac12\log \log q}}$, and $I_1,I_2 \subset \mathbb R$ be intervals. Let us first consider the case $I_1,I_2$ are closed.
By \eqref{eq:lfbound}, for $\chi \in \mathcal  S_f$ we have that $\mathcal L_f(\chi) \in I_1$ if and only if $\mathcal L_f(\chi) \in \mathcal I_1$ where $\mathcal I_1=I_1 \cap [-A \log \log \log q, A \log \log \log q]$, where $A>0$ is sufficiently large. Write $\mathcal I_1=[a_1,b_1]$.
Using \eqref{eq:clean} we have for  
all $\chi \in \mathcal{S}_{f}$ that there exists $C>0$ sufficiently large such that for $\delta=C \frac{\log \log \log q}{\sqrt{\log \log q}}$
\[
|\mathbf{1}_{\mathcal I_1}(\mathcal L_f(\chi))-\mathbf{1}_{\mathcal I_1}(\mathcal P_f(\chi))| \le \mathbf{1}_{[a_1-\delta,a_1+\delta]}(\mathcal P_f(\chi))+\mathbf{1}_{[b_1-\delta,b_1+\delta]}(\mathcal P_f(\chi)).
\]
For $\chi \in \mathcal S_g$ we arrive at a similar inequality and can replace $I_2$ with $\mathcal I_2$, which is defined analogously to $\mathcal I_1$, with $\mathcal I_2=[a_2,b_2]$ and $|b_2-a_2| \ll \log \log \log q$.
Write $I_{\delta,j}=[a_j-\delta,a_j+\delta] \cup  [b_j-\delta,b_j+\delta]$ for $j=1,2$.
Consequently, using the above inequality for $f$ and $g$, H\"older's inequality with exponents $\frac{1}{e_1}+\frac{1}{e_2}=1$ where $e_1=1+\varepsilon$, so $e_2 \asymp 1/\varepsilon$, and applying Proposition \ref{prop:mollified} we have that
\begin{equation} \label{eq:ltop}
\begin{split}
&\frac{1}{\varphi_W^{\star}(q)}\ \sumstar_{\chi \in \mathcal{S}_{f} \cap \mathcal S_g} W(\chi) \, \mathbf{1}_{I_1}(\mathcal L_f(\chi)) \mathbf{1}_{ I_2} (\mathcal L_g(\chi))= \frac{1}{\varphi_W^{\star}(q)} \sumstar_{\chi \in \mathcal{S}_{f} \cap \mathcal S_g} W(\chi) \, \mathbf{1}_{\mathcal I_1}(\mathcal P_f(\chi)) \mathbf{1}_{\mathcal I_2} (\mathcal P_g(\chi))
\\
&\qquad \qquad \qquad \qquad \qquad +O_\varepsilon\bigg(q^{1/e_2-1} \bigg(\ \,\sumstar_{\chi \pamod q} \Big(\mathbf{1}_{I_{\delta,1}}(\mathcal P_f(\chi))+\mathbf{1}_{I_{\delta,2}}(\mathcal P_g(\chi))\Big) \bigg)^{1/e_1} \bigg),
\end{split}
\end{equation}
where we have also used that $(\mathbf{1}_{I_{\delta,1}}(\mathcal P_f(\chi))+\mathbf{1}_{I_{\delta,2}}(\mathcal P_g(\chi)))^{e_1} \le 2^{\varepsilon}( \mathbf{1}_{I_{\delta,1}}(\mathcal P_f(\chi))+\mathbf{1}_{I_{\delta,2}}(\mathcal P_g(\chi)))$
and $\mathbf{1}_{[a_1-\delta,a_1+\delta]}(\mathcal P_f(\chi))+\mathbf{1}_{[b_1-\delta,b_1+\delta]}(\mathcal P_f(\chi)) \le 2\mathbf{1}_{I_{\delta,1}}(\mathcal P_f(\chi)) $.
Using Cauchy-Schwarz's inequality, the sum on the left hand side  as well as the sum in the main term on the right hand side  above can be extended to all primitive characters modulo $q$ at the cost of an error term of size $O(1/\log \log q)$.

Let $K(x)=(\frac{\sin \pi x}{\pi x} )^2$ and note that $\widehat K(u)=\max\{0,1-|u|\}$. 
By Lemma \ref{lem:BS},  for $I=[a,b]$ and $\Delta>0$ we have $|\mathbf{1}_I(x)-F_{I,\Delta}(x)| \le 2$, so $|\mathbf{1}_I(x)-F_{I,\Delta}(x)|^{e_1} \le 2^{\varepsilon} ( K(\Delta(x-a))+K(\Delta(b-x)))$. We now take $\Delta=\sqrt{\frac{\log \log q}{\log \log \log q}}$.
By H\"older's inequality and Proposition \ref{prop:mollified} we get that
\begin{align}\label{eq:smoothtosharp}
&\sumstar_{\chi \pamod q} |W(\chi)| \Big|\mathbf{1}_{\mathcal I_1}(\mathcal P_f(\chi))-F_{\mathcal I_1,\Delta}(\mathcal P_f(\chi))\Big|\nonumber\\
&\qquad\qquad \ll_\varepsilon q^{1/e_2}  \bigg(\,\, \sumstar_{\chi \pamod q} \Big( K(\Delta(\mathcal P_f(\chi)-a_1))+K(\Delta(b_1-\mathcal P_f(\chi)))\Big)  \bigg)^{1/e_1}.
\end{align}
By Fourier inversion and Lemma \ref{lem:clteasy} we find uniformly for any $\alpha \in \mathbb R$ that
\begin{equation} \label{eq:errterm}
\frac{1}{\varphi^{\star}(q)}\ \sumstar_{\chi \pamod q} K(\Delta (\mathcal P_f(\chi)-\alpha))=\frac{1}{\Delta} \int_{-\Delta}^{\Delta} \left(1-\frac{|u|}{\Delta} \right) e(-\alpha u) \Psi_q(u,0) \, du \ll \frac{1}{\Delta}.
\end{equation}
Hence, using \eqref{eq:mass}, \eqref{eq:smoothtosharp},
 and \eqref{eq:errterm} we can replace $\mathbf{1}_{\mathcal I_1}(\mathcal P_f(\chi))$ by $F_{\mathcal I_1,\Delta}(\mathcal P_f(\chi))$ in the main term in \eqref{eq:ltop} at the cost of an error term of size $O_\varepsilon((\log \log q)^{-1/2+\varepsilon})$. Similarly, we can replace $\mathbf{1}_{\mathcal I_2}(\mathcal P_g(\chi))$ with $F_{\mathcal I_2,\Delta}(\mathcal P_g(\chi))$ at the cost of an error term of the same size.  Using Fourier inversion, \eqref{thm:mainresult2} and Lemma \ref{lem:BS} we get that up to an error term of size $O_\varepsilon((\log \log q)^{-1/2+\varepsilon})$ the main term on the right hand side of \eqref{eq:ltop} equals
\begin{equation}  \notag
\begin{split}
&\frac{1}{\varphi_W^{\star}(q)}\ \sumstar_{\chi \pamod q} W(\chi) F_{\mathcal I_1,\Delta}(\mathcal P_f(\chi)) F_{\mathcal I_2,\Delta}(\mathcal P_g(\chi))= \int_{\mathbb R^2} \widehat {F_{\mathcal I_1, \Delta}}(u) \widehat {F_{\mathcal I_2,\Delta}}(v) \Phi_q(u,v) \, du \, dv \\
& = \int_{\mathbb R^2} \widehat {F_{\mathcal I_1, \Delta}}(u) \widehat {F_{\mathcal I_2,\Delta}}(v) e^{-2\pi^2(u^2+v^2)} \, du \, dv+O_\varepsilon\bigg( |\mathcal I_1| |\mathcal I_2| ((\log \log q)^{-1/2+\varepsilon}+  \Delta^2 (\log q)^{-10})\bigg).
\end{split}
\end{equation}
Applying Plancherel and then using Lemma \ref{lem:BS} the main term above equals
\begin{equation} \label{eq:plancherel}
\begin{split}
\frac{1}{2\pi} \int_{\mathbb R^2} F_{\mathcal I_1,\Delta}(x)F_{\mathcal I_2,\Delta}(y) e^{-(x^2+y^2)/2} \, dx \, dy = \frac{1}{2\pi} \int_{I_1 \times I_2} e^{-(x^2+y^2)/2} \, dx \, dy+O_\varepsilon\left((\log \log q)^{-1/2+\varepsilon} \right),
\end{split}
\end{equation}
where we have also used the rapid decay of the Gaussian to bound the portion of the integral with $ (I_1 \times  I_2) \setminus (\mathcal I_1 \times \mathcal I_2)$ by $\ll (\log \log q)^{-10}$.

To estimate the error term in \eqref{eq:ltop}, we see that by arguing as above we have 
\begin{equation} \notag
   \frac{1}{\varphi^{\star}(q)}\ \sumstar_{\chi \pamod q} \mathbf{1}_{I_{\delta,1}}(\mathcal P_f(\chi)) =\frac{1}{\sqrt{2\pi}} \int_{I_{\delta,1}} e^{-x^2/2} \, dx +O_\varepsilon\left((\log \log q)^{-1/2+\varepsilon} \right) \ll_\varepsilon (\log \log q)^{-1/2+\varepsilon}.
\end{equation}
Applying this with \eqref{eq:plancherel} in \eqref{eq:ltop} completes the proof in the case $I_1,I_2$ are closed. The other cases follow in the exact same way since the conclusions of Lemma \ref{lem:BS} hold for any finite interval, since $F_{I,\Delta}$ is continuous.
\end{proof}

\begin{proof}[Proof of Corollary \ref{cor:ratios}]
Let $\omega=\frac{\Lambda}{ \sqrt{\log \log q)}}$.
By a similar yet slightly easier argument to the one given in the proof of Theorem 1.4, which we will omit, we have that
\[
\begin{split}
 \frac{1}{\varphi_W^{\star}(q)}&\sumstar_{\chi \pamod q} W(\chi) \mathbf 1_{[-\omega, \omega]}\bigg( \frac{\log |L(\tfrac12,f\otimes \chi)|}{\sqrt{\frac12 \log \log q}}-\frac{\log |L(\tfrac12,g\otimes \chi)|}{\sqrt{\frac12 \log \log q}}\bigg) \\
 &=\frac{1}{2 \sqrt{\pi}} \int_{-\omega}^{\omega} e^{-u^2/4} \, du +O_\varepsilon((\log \log q)^{-1/2+\varepsilon})
\end{split}
\]
with the main difference in the proof being that after applying Fourier inversion in place of $\Phi_q(u,v)=e^{-2\pi^2 (u^2+v^2)}+o(1)$ we will have $\Phi_q(u,-u)=e^{-4\pi^2 u^2}+o(1)$ and the Fourier transform of the Gaussian $\mathbf g(u)=e^{-4\pi^2 u^2}$ is $\widehat {\mathbf g}(u)=\frac{1}{2\sqrt{\pi}} e^{-u^2/4}$. Hence, by H\"older's inequality with exponents $\frac{1}{e_1}+\frac{1}{e_2}=1$ the formula above yields 
\[
q \omega \ll \bigg(\, \, \, \sumstar_{\chi \pamod q} |W(\chi)|^{e_1} \bigg)^{1/e_1}   \bigg(\# \bigg\{ \chi \neq \chi_0 : L(\tfrac12,g \otimes \chi) \neq 0, e^{-\Lambda}  \le \bigg| \frac{L(\tfrac12,f\otimes \chi)}{L(\tfrac12, g \otimes \chi)} \bigg| \le e^{\Lambda} \bigg\}\bigg) ^{1/e_2},
\]
where for brevity we wrote $\chi \neq \chi_0$ to mean that $\chi$ is a nonprincipal character modulo $q$.
Applying Proposition 1.3 gives that
\[
  \# \bigg\{ \chi \neq \chi_0 : L(\tfrac12,g \otimes \chi) \neq 0, \, e^{-\Lambda}  \le \bigg| \frac{L(\tfrac12,f\otimes \chi)}{L(\tfrac12, g \otimes \chi)} \bigg| \le e^{\Lambda} \bigg\}  \gg \omega^{e_2} q
\]
taking $e_2=1+\varepsilon$ gives the claim.

\end{proof}

\section{Upper bounds for mollified moments}\label{upperbounds}

In this section we prove Proposition \ref{prop:mollified}, which follows immediately from the following proposition and an application of Cauchy-Schwarz's inequality.

\begin{proposition}\label{prop:mollifiedf}
Assume GRH. Let $k > 0$ and suppose that $\eta \lceil k \rceil$ is sufficiently small. Then
\[
\sumstar_{\chi\pamod q} |L(\tfrac12, f\otimes\chi) M_f(\chi)|^{2k}\ll_{f,k} q.
\]
\end{proposition}
We assume GRH for $L(s, f\otimes \chi),L(s,\tmop{Sym}^2 f \otimes \chi^2)$ and $L(s,\chi^2)$ for all characters $\chi \pmod q$.

\subsection{Preliminaries} 
Note that by using the inequality $t^{2k} \le 1+t^{2 \lceil k \rceil}$, which holds for $t \ge 0$ and $k>0$, it suffices to prove the result for $k \in \mathbb N$.  
For an interval $I$ and a completely multiplicative function $a(n)$ we define
\[
P_{I}(\chi; a )=\sum_{\substack{p \in I }} \frac{a(p) \chi(p)}{\sqrt{p}}.
\]
Recalling that $\sum_{p_1\cdots p_{\ell}} 1 =\ell! \nu(n)$ we have
\begin{equation} \label{eq:PID}
\begin{split}
P_I(\chi; a)^\ell=&\sum_{p_1, \ldots, p_\ell \in I} \frac{a(p_1 \cdots p_\ell)\chi(p_1 \cdots p_\ell)}{\sqrt{p_1 \cdots p_\ell}}=\sum_{\substack{p|n \Rightarrow p \in I \\ \Omega(n)=\ell}} \frac{a(n) \chi(n)}{\sqrt{n}} \sum_{p_1\cdots p_\ell=n} 1 \\
=& \ell! \sum_{\substack{p|n \Rightarrow p \in I \\ \Omega(n)=\ell}} \frac{a(n)\nu(n) \chi(n)}{\sqrt{n}} .
\end{split}
\end{equation}

For $\ell$ a positive even integer and $t \in \mathbb R$, let
\[
E_{\ell}(t)=\sum_{j \le \ell} \frac{t^j}{j!}.
\]
By \cite[Lemma 1]{Radziwill-Soundararajan} we see that $E_{\ell}(t)>0$ for any $t \in \mathbb R$ if $\ell$ is even and for $t \le \ell/e^2$  that 
\begin{equation} \label{eq:taylor2}
e^t \le (1+e^{-\ell})E_{\ell}(t).
\end{equation}
For any real number $k \neq 0$ we have
\begin{equation} \label{eq:id}
\begin{split}
E_{\ell}\big(2k \tmop{Re}(P_{I}(\chi;a))\big)=&\sum_{j \le \ell} \frac{k^j}{j!}  \big(P_{I}(\chi;a)+P_{I}(\overline \chi;a)\big)^j\\
=&\sum_{j\le \ell} k^j \sum_{r=0}^j \sum_{\substack{p|m \Rightarrow p \in I \\ \Omega(m) = r}} \frac{  a(m) \nu(m)\chi(m)}{\sqrt{m}}   \sum_{\substack{p|n \Rightarrow p \in I \\ \Omega(n) = j-r}} \frac{  a(n) \nu(n)\overline\chi(n)}{\sqrt{n}}   \\
=& \sum_{\substack{p|n \Rightarrow p \in I \\ \Omega(n) \le \ell}} \frac{k^{\Omega(n)}a(n)}{\sqrt{n}} (\nu \chi \ast \nu \overline \chi)(n).
\end{split}
\end{equation}

Additionally, for each $0\leq j\leq J$ let
\begin{equation} \label{eq:ddef}
D_j(\chi;k)=\prod_{r=0}^j (1+e^{-\ell_r}) E_{\ell_r}\big(2k \tmop{Re}(P_{I_r}(\chi;a_{f,j}))\big)
\end{equation}
and note that $D_j(\chi;k)>0$. 

\subsection{The random model} As in Subsection \ref{subsectionrandom}, let $\{X(p)\}_{p}$ be i.i.d. random variables that are uniformly distributed on the unit circle and let
$
X(n) =\prod_{p^a || n } X(p)^a.$
Define
\[
P_I(X;a)=\sum_{p \in I} \frac{a(p)}{\sqrt{p}} X(p)
\]
and just as in \eqref{eq:id} we have that
\begin{equation} \label{eq:EXexpand}
E_{\ell}\big(2k \tmop{Re}(P_{I}(X;a))\big)= \sum_{\substack{p|n \Rightarrow p \in I \\ \Omega(n) \le \ell}} \frac{k^{\Omega(n)}a(n)}{\sqrt{n}} (\nu X \ast \nu \overline X)(n).
\end{equation}
We also let
\begin{equation*}
D_j(X;k)=\prod_{r=0}^j (1+e^{-\ell_r}) E_{\ell_r}\big(2k \tmop{Re}(P_{I_r}(X;a_{f,j}))\big).
\end{equation*}

\subsection{Preliminary lemmas}
We use the conventions that $D_{-1}=1$ and $P_{I_{J+1}}=1$, $\theta_{J+1}=1$.
\begin{lemma}\label{lem:diagonal}
Let $0 \le j \le J+1$ and let $b(n)$ be a completely multiplicative function.
For $t \in \mathbb Z$ with $0 \le t \le \frac{2}{5 \theta_{j+1}}$
we have that
\[
\begin{split}
\frac{1}{\varphi(q)}\ \sumstar_{\chi \pamod q} &D_{j-1}(\chi;k)\big(\tmop{Re}(P_{I_{j}}(\chi;b))\big)^{2t}|M_f(\chi)|^{2k}\\
&=\mathbb E\left( D_{j-1}(X;k) \big(\tmop{Re}(P_{I_{j}}(X,b))\big)^{2t}|M_f(X)|^{2k} \right)+O(q^{-1/10}).
\end{split}
\]
\end{lemma}
\begin{remark}\label{rem:indepexp}
\emph{Since $\{X(p)\}_p$ are independent random variables and the intervals $I_j$ are disjoint, we have that
\[
\mathbb E\Big( D_J(X;k)|M_f(X)|^{2k} \Big)=\prod_{0 \le j \le J} (1+e^{-\ell_j}) \mathbb E\Big( E_{\ell_j}\big(2k \tmop{Re}(P_{I_j}(X;a_{f,J}))\big)|M_{f,j}(X)|^{2k}\Big).
\]}
\end{remark}
\begin{proof}
The error term arises from pointwise bounding the principal character contribution (this is why the formula is not an identity). By \eqref{eq:trivial} the two expressions are equal when the sum is over all $\chi$ modulo $q$. The contribution from the term with $\chi=\chi_0$ is bounded by  $\ll q^{-1} q^{1/10} (q^{\theta_{j}})^{\frac{4}{5\theta_{j}}} \ll q^{-1/10} $. 
\end{proof}

\begin{lemma} \label{lem:mollifiedeuler}
Let $0 \le j \le J$ and let $b(n)$ be a completely multiplicative function with $b(p) \ll 1$. Then
\begin{equation} \label{eq:expectedbd}
\begin{split}
&\mathbb E\Big( E_{\ell_j}\big(2k \tmop{Re}(P_{I_j}(X;b))\big)|M_{f,j}(X)|^{2k}\Big)\\
&\qquad =\left(1+O\left(\frac{\mathbf{1}_{j=0} (\log q)^{O(1)}+1}{2^{\ell_j}} \right)\right) \prod_{p \in I_j} \left(1+\frac{k^2(a_{f,J}(p) -b(p))^2}{p}+O\left( \frac{1}{p^2}\right) \right)
\end{split}
\end{equation}
and for $t\in \mathbb Z$ with $t\ge 4k \ell_j$ we have that
\begin{equation*}
\mathbb E\Big(  \big(\tmop{Re} (P_{I_j}(X;b))\big)^{2t} |M_{f,j}(X)|^{2k} \Big) \ll\Big(\mathbf{1}_{j=0} (\log q)^{O(1)}+1 \Big) \frac{(2t)!}{2^{2t}\lfloor\tfrac 34 t\rfloor!} \bigg(\sum_{p \in I_j} \frac{b(p)^2}{p} \bigg)^{t}.
\end{equation*}
\end{lemma}

\begin{proof}
Using \eqref{eq:EXexpand} we have that
\begin{equation} \label{eq:switch2}
\begin{split}
\mathbb E & \Big( E_{\ell_j}\big(2k \tmop{Re}(P_{I_j}(X;b))\big|M_{f,j}(X)|^{2k}\Big)\\
&\qquad=\sum_{\substack{p|mn \Rightarrow p \in I_j  \\ \Omega(m) \le \ell_j, \Omega(n) \le 2k \ell_j}} \frac{k^{\Omega(m)} b(m) \lambda(n) a_{f,J}(n)}{\sqrt{mn}} \mathbb E\Big( (\nu X \ast \nu \overline X)(m) \left(\nu_{k;\ell_j}X\ast \nu_{k;\ell_j} \overline X \right)(n)\Big)\\
&\qquad=\sum_{\substack{p|mn \Rightarrow p \in I_j  \\ \Omega(m) \le \ell_j, \Omega(n) \le 2k \ell_j}} \frac{k^{\Omega(m)}b(m) \lambda(n)a_{f,J}(n)}{\sqrt{mn}}\sum_{\substack{c_1d_1=m \\ c_2d_2=n \\ c_1c_2=d_1d_2}} \nu(c_1) \nu(d_1) \nu_{k;\ell_j}(c_2) \nu_{k;\ell_j}(d_2).
\end{split}
\end{equation}
Since $\tmop{max}\{\Omega(m),\Omega(n)\} > \ell_j$ implies $2^{\Omega(mn)}/2^{\ell_j}> 1$, there exists fixed $C>0$ sufficiently large such that 
\begin{equation}\label{eq:extendsum} 
\begin{split}
&\mathbb E  \Big( E_{\ell_j}\big(2k \tmop{Re}(P_{I_j}(X;b))\big)|M_{f,j}(X)|^{2k}\Big)
\\
&\qquad=\sum_{p|mn \Rightarrow p \in I_j } \frac{k^{\Omega(m)}b(m) \lambda(n)a_{f,J}(n)}{\sqrt{mn}}\sum_{\substack{c_1d_1=m \\ c_2d_2=n \\ c_1c_2=d_1d_2}} \nu(c_1) \nu(d_1) \nu_{k}(c_2) \nu_{k}(d_2)
\\&\qquad\qquad+O\bigg( \frac{1}{2^{\ell_j}} \sum_{p|r_1r_2 \Rightarrow p\in I_j} \frac{C^{\Omega(r_1r_2)}}{r_1r_2}\bigg).
\end{split}
\end{equation}
Here we also used that $\nu_{k;\ell}(n)=\nu_{k}(n)$ if $\Omega(n) \le \ell$.
The error term in \eqref{eq:extendsum} is
\begin{equation} \label{eq:extendsumET}
\ll \frac{\mathbf{1}_{j=0} (\log q)^{O(1)}+1}{2^{\ell_j}},
\end{equation}
where we have used that $c_0$ is sufficiently large. Define the multiplicative functions $h_1,h_2$ and $h$ by
\[
h_1(n)=k^{\Omega(n)}b(n) \nu(n), \quad h_2(n)=\lambda(n)a_{f,J}(n)\nu_k(n), \quad \text{ and } \quad h(n)=(h_1\ast h_2)(n),
\] 
and note that
\[
\begin{split}
h(p)&=b(p)k-a_{f,J}(p)\nu_k(p)=k(b(p)-a_{f,J}(p)),\\
|h(p^a)|&\leq \frac{(Ck)^a}{a!}
\end{split}
\]
for some constant $C>0$. Therefore, writing $mn=r$, $c_1c_2=r_1$, $d_1d_2=r_2$, grouping terms appropriately and using that $a_{f,J}(n),k^{\Omega(n)},b(n),\lambda(n)$ are completely multiplicative the main term in \eqref{eq:extendsum} is
\begin{equation} \label{eq:extendsumMT}
\sum_{p|r \Rightarrow p \in I_j} \frac{1}{\sqrt{r}} \sum_{\substack{r_1r_2=r \\ r_1=r_2}}  h(r_1)h(r_2)=\sum_{\substack{p|n  \Rightarrow p \in I_j} }\frac{h(n)^2}{n}=\prod_{p \in I_j}\left( 1+\frac{ k^2(b(p)-a_{f,J}(p))^2}{p}+O\left( \frac{1}{p^2}\right)\right).
\end{equation}
Hence, using \eqref{eq:extendsumET} and \eqref{eq:extendsumMT}  in \eqref{eq:extendsum} completes the proof of \eqref{eq:expectedbd}.

For the second statement, we note that by using \eqref{eq:PID} we have that
\[
\begin{split}
\big(\tmop{Re}( P_{I_j}(X;b))\big)^{2t}&=2^{-2t}\sum_{s=0}^{2t}\binom{2t}{s}P_{I_j}(X;b)^sP_{I_j}(\overline X;b)^{2t-s}\\
&=2^{-2t}(2t)!\sum_{s=0}^{2t}\sum_{\substack{p|m  \Rightarrow p \in I_j \\\Omega(m)=s}}\frac{b(m)\nu(m)}{\sqrt{m}}X(m)\sum_{\substack{p|n  \Rightarrow p \in I_j \\\Omega(n)=2t-s}}\frac{b(n)\nu(n)}{\sqrt{n}}\overline X(n)\\
&=2^{-2t}(2t)!\sum_{\substack{p|n\Rightarrow p\in I_j\\\Omega(n)=2t}}\frac{b(n)}{\sqrt{n}}(\nu X\ast \nu\overline X)(n).
\end{split}
\]
Therefore, similarly to \eqref{eq:switch2} we have that
\[
\begin{split}
\frac{2^{2t}}{(2t)!} & \mathbb E \Big ( \big(\tmop{Re}( P_{I_j} (X;b))\big)^{2t} |M_{f,j}(X)|^{2k} \Big)\\
&=\sum_{\substack{p|mn\Rightarrow p\in I_j \\ \Omega(m)=2t,  \Omega(n)\le 2k  \ell_j}}  \frac{b(m) \lambda(n) a_{f,J}(n)}{\sqrt{mn}}
\sum_{\substack{c_1d_1=m \\ c_2d_2=n \\ c_1c_2=d_1d_2}}  \nu(c_1)\nu(d_1) \nu_{k;\ell_j}(c_2)\nu_{k; \ell_j}(d_2).
\end{split}
\]
Taking absolute values inside the sum and recalling that $\nu_{k;\ell}(n)\le\nu_{k}(n)$, we have the bound
\[
\ll\sum_{\substack{p|c_1c_2d_1d_2\Rightarrow p\in I_j \\ \Omega(c_1d_1)=2t, \Omega(c_2d_2)\le 2k \ell_j \\c_1c_2=d_1d_2}} \frac{|b(c_1d_1) a_{f,J}(c_2d_2)|}{\sqrt{c_1c_2d_1d_2}}  \nu(c_1) \nu(d_1) \nu_{k}(c_2)\nu_k(d_2).
\]
We now write $c_i=g_ic'_i$ and $d_i=g_id'_i$ with $g_i:=(c_i,d_i)$ and $(c'_i,d'_i)=1$ for $i=1,2$. Then, since $c_1c_2=d_1d_2$, we have that $c'_1c'_2=d'_1d'_2$ and therefore $c'_1=d'_2$ and $d'_1=c'_2$. Relabelling $c'_i$ as $c_i$ and writing $\gamma_k(n):=a_{f,J}(n)\nu_k(n)$, the above is bounded by
\begin{equation}\label{eq:expandedexp}
\sum_{\substack{p|g_1g_2c_1c_2 \Rightarrow p\in I_j \\ \Omega(c_1c_2g_1^2)=2t, \Omega(c_1c_2)\le 2k \ell_j}}\frac{b(g_1)^2 \gamma_k^2(g_2)|b(c_1c_2) \gamma_k(c_1)\gamma_k(c_2)|}{g_1g_2c_1c_2} \nu(g_1)^2,
\end{equation}
where we have used that $\nu(n)\leq1$ and $\nu_k(mn) \le \nu_k(m)\nu_k(n)$. 

We now bound the sum over $g_2$. Using that $|\lambda_f(p)|\le 2$, we have 
\[
\sum_{p|g_2\Rightarrow p\in I_j}\frac{\gamma_k^2(g_2)}{g_2}
\ll\prod_{p\in I_j}\left(1+\frac{\lambda_f(p)^2\nu_{k}(p)^2}{p}\right)
\ll
\begin{cases}
\big(\frac{\theta_j}{\theta_{j-1}}\big)^{4k^2}\asymp_k 1,&\text{ if }j\neq0,\\
(\log q)^{O_k(1)},&\text{ if }j=0.
\end{cases}
\]
Hence  we may bound \eqref{eq:expandedexp} by
\[
\begin{split}
\ll\Big(\mathbf{1}_{j=0}(\log q)^{O(1)}+1\Big)\sum_{\substack{p|c_1c_2\Rightarrow p\in I_j \\ \Omega(c_1c_2)\le 2k \ell_j \\ 2| \Omega(c_1c_2)}} & \frac{|b(c_1c_2) \gamma_k (c_1)\gamma_k(c_2)|}{c_1c_2}\sum_{\substack{p|g\Rightarrow p\in I_j\\\Omega(g)=t-\Omega(c_1c_2)/2}}\frac{b(g)^2\nu(g)^2}{g}\\
\ll\Big(\mathbf{1}_{j=0}(\log q)^{O(1)}+1\Big)\sum_{\substack{p|c_1c_2 \Rightarrow p\in I_j \\ \Omega(c_1c_2)\le 2k \ell_j \\ 2| \Omega(c_1c_2)}} & \frac{|b(c_1c_2) \gamma_k (c_1)\gamma_k(c_2)|}{c_1c_2} \bigg(\frac{1}{\big(t-\frac{\Omega(c_1c_2)}{2}\big)!}\bigg(\sum_{p\in I_j}\frac{b(p)^2}{p}\bigg)^{t-\frac{\Omega(c_1c_2)}{2}}\bigg).
\end{split}
\]
Since by assumption $t\ge 4k\ell_j$ and $\Omega(c_1c_2)\le 2k\ell_j$, we have that $\frac{3}{4}t\le t-\Omega(c_1c_2)/2$. Thus the final bracketed term above is bounded by
\[
\frac{1}{\lfloor\frac{3}{4}t\rfloor!}\bigg(\sum_{p\in I_j}\frac{b(p)^2}{p}\bigg)^t.
\]

For the remaining sum over $c_1,c_2$, there exists some $C>0$ such that this sum is bounded by $\sum_{p|n\Rightarrow p\in I_j}\frac{C^{\Omega(n)}}{n}\ll \mathbf{1}_{j=0}(\log q)^{O(1)}+1$. Therefore, we have that
\[
\mathbb E\Big(\big(\tmop{Re}( P_{I_j}(X;b))\big)^{2t} |M_{f,j}(X)|^{2k}  \Big) 
\ll\Big(\mathbf{1}_{j=0}(\log q)^{O(1)}+1\Big)\frac{(2t)!}{2^{2t}\lfloor\frac{3}{4}t\rfloor!}\bigg(\sum_{p\in I_j}\frac{b(p)^2}{p}\bigg)^t,
\]
as claimed.
\end{proof}

\begin{lemma} \label{lem:L1}
Assume GRH. Suppose $Y \ge (\log q)^3$. Then the following statements hold:
\begin{enumerate}
    \item $\displaystyle\log L(1,\tmop{Sym}^2 f \otimes \chi^2)=\sum_{p \le Y} \frac{\lambda_f(p^2) \chi(p)^2}{p}+O(1)$;
    \item $\displaystyle\log L(1, \chi^2)=\sum_{p \le Y} \frac{\chi(p)^2}{p}+O(1)$, for $\chi^2 \neq \chi_0$;
   \item $\displaystyle \frac{1}{\varphi^{\star}(q)}\, \sumstar_{\substack{\chi\ \pamod q \\ \chi^2 \neq \chi_0}} \bigg|\frac{L(1, \tmop{Sym}^2 f \otimes \chi^2)}{L(1, \chi^2)} \bigg|^{2k} \ll 1$.
\end{enumerate}

\end{lemma}
\begin{proof}
Applying \cite[Lemma 5]{BB} gives (1) and (2). The last claim can be shown to follow unconditionally, however it is easy to prove using GRH and we give a sketch below. Using (1) and (2) we have
\[
\bigg|\frac{L(1, \tmop{Sym}^2 f \otimes \chi^2)}{L(1, \chi^2)} \bigg|^{2k} \ll \exp\bigg(2k \tmop{Re} \sum_{p \le (\log q)^3} \frac{(\lambda_f(p^2)-1) \chi^2(p)}{p} \bigg).
\]
Since $\sum_{p \le (\log q)^3} \frac{(\lambda_f(p^2)-1) \chi^2(p)}{p} \ll \log \log \log q$, we can apply \eqref{eq:taylor2} to the right hand side above with $\ell=2 \lfloor (\log \log \log q)^2 \rfloor$. Using the non-negativity of $E_{\ell}(\cdot)$ we extend the sum to all $\chi$ modulo $q$ to bound the left hand side of (3) by
\[
\ll \sum_{j=0}^{\ell} \frac{(2k)^j}{j!} \frac{1}{\varphi^{\star}(q)} \sum_{\substack{\chi \pamod q  }} \bigg(\tmop{Re}\bigg( \sum_{p \le (\log q)^3} \frac{(\lambda_f(p^2)-1) \chi^2(p)}{p}\bigg) \bigg)^j.
\]
Arguing as in the proof of \eqref{eq:moments2} the sum over $\chi\ \pamod q$ can be seen to be $ \ll q (C j)^j$ for each $1 \le j \le \ell$ and $C>0$ sufficiently large. This gives that the right hand side above is $$\ll \sum_{j=1}^{\ell} \bigg(\frac{C}{j}\bigg)^{j/2} \ll 1, $$ where $C=C(k)>0$ is sufficiently large.
\end{proof}

\begin{lemma} \label{lem:Lbd}
Assume GRH. Suppose $\chi^2\ne\chi_0$. Then either
\[
\tmop{Re}(P_{I_0}(\chi;a_{f,j}))\ge\frac{\ell_0}{ke^2}
\]
for some $0\le j\le J$ or
\[
\begin{split}
&|L(\tfrac12,f\otimes \chi)|^{2k} L(1,\tmop{Sym}^2 f \otimes \chi^2)^{-k}  L(1,\chi^2)^{k} \\
&\qquad\ll D_J(\chi;k)+\sum_{\substack{0 \le j \le J-1 \\ j+1 \le u \le J}} \exp\left(\frac{6k}{\theta_j} \right) D_j(\chi;k) \left( \frac{e^2 k \tmop{Re}(P_{I_{j+1}}(\chi;a_{f,u}))}{\ell_{j+1}}\right)^{t_{j+1}}
\end{split}
\]
for any sequence of nonnegative integers $(t_j)$.
\end{lemma}
\begin{proof}
By \cite[Theorem 2.1]{C} we have that for $Y \ge 8$
\begin{equation}\label{eq:chandeebd}
|L(\tfrac 12, f\otimes \chi)|^{2k}
\le \exp\bigg(2k\tmop{Re}\sum_{p^n\le Y} \frac{ (\alpha_{f,1}(p)^n+\alpha_{f,2}(p)^n)\chi(p^n)}{np^{n(\frac 12 +\frac{1}{\log Y})}}\frac{\log Y/p^n}{\log Y}+6k \frac{\log q}{\log Y} +O(1)\bigg).
\end{equation}
In the sum over prime powers, the contribution of the powers $n\ge 3$ is $O(1)$. For the prime squares, we note that $\alpha_{f,1}(p)^2+\alpha_{f,2}(p)^2=\lambda_f(p^2)-1$ and therefore the contribution is
\begin{equation}\label{eq:primesquare}
k\sum_{p\le \sqrt{Y}}\frac{(\lambda_f(p^2)-1)\chi(p)^2}{p}\frac{\log Y/p^2}{p^{2/\log Y} \log Y}=k\sum_{p\le \sqrt{Y}}\frac{(\lambda_f(p^2)-1)\chi(p)^2}{p}+O(1),
\end{equation}
where in the error term we used that $\sum_{p \le \sqrt{Y}} \frac{\log p}{p \log Y} \ll 1$.
Applying Lemma \ref{lem:L1}, (1) and (2) we see that \eqref{eq:primesquare} equals
\begin{equation}\label{eq:primesquare2}
k \log L(1,\tmop{Sym}^2 f \otimes \chi^2)-k \log L(1, \chi^2)+O(1),
\end{equation}
provided $ Y \ge (\log q)^6$.

For $0\leq j\leq J$ let $\mathcal{S}_j$ be the set of characters $\chi$ modulo $q$ such that $\chi^2\ne\chi_0$ and
\[
\max_{j\le r\le J} \tmop{Re}P_{I_j}(\chi;a_{f,r})<\frac{\ell_j}{k e^2}.
\]
For each such $\chi$ we must have either (i) $\chi\notin\mathcal{S}_0$; (ii) $\chi\in\mathcal{S}_j$ for all $0\le j\le J$; (iii) there exists $0\le j\le J-1$ such that $\chi\in\mathcal{S}_r$ for $0\le r\le j$ and $\chi\notin\mathcal{S}_{j+1}$. In particular, for each $\chi^2\ne\chi_0$ we have either
\begin{align}
\max_{0\le r\le J} \tmop{Re}P_{I_0}(\chi;a_{f,r})&\ge\frac{\ell_0}{k e^2};\label{eq:case1}\\
\max_{j\le r\le J}\tmop{Re}P_{I_j}(\chi;a_{f,r})&<\frac{\ell_j}{k e^2} \text{ for each }0\leq j\leq J\label{eq:case2};
\end{align}
or for some $0\leq j\leq J-1$
\begin{equation}\label{eq:case3}
\begin{split}
\max_{r\le u\le J} \tmop{Re}P_{I_r}(\chi;a_{f,u}) & <\frac{\ell_r}{ke^2}\text{ for each } 0\le r\le j\\
\text{and }\max_{j+1\le u\le J} \tmop{Re}P_{I_{j+1}}(\chi;a_{f,u}) & \ge\frac{\ell_{j+1}}{ke^2},
\end{split}
\end{equation}
with \eqref{eq:case1} corresponding to (i), \eqref{eq:case2} to (ii) and \eqref{eq:case3} to (iii). 

If we have \eqref{eq:case1}, then we can conclude. If \eqref{eq:case2} holds, then we set $Y=q^{\theta_J}$ in \eqref{eq:chandeebd} so that $\frac{\log q}{\log Y}\ll 1$. Also applying \eqref{eq:primesquare2} to handle the contribution of the squares of primes, we have that
\[
\begin{split}
|L(\tfrac12,f\otimes \chi)|^{2k} L(1,\tmop{Sym}^2 f \otimes \chi^2)^{-k}  L(1,\chi^2)^{k} &\ll \exp\bigg( 2k \sum_{p\le q^{\theta_J}}\frac{a_{f,J}(p)\chi(p)}{\sqrt{p}}\bigg)\\
&\le \prod_{j=0}^J (1+e^{-\ell_j}) E_{\ell_j}\big (2k\text{Re}( P_{I_j}(\chi; a_{f,J}))\big)\\
&= D_J(\chi;k),
\end{split}
\]
by the definition \eqref{eq:ddef}. Lastly, if we are in the case \eqref{eq:case3}, then we take $Y=q^{\theta_j}$ in \eqref{eq:chandeebd} so that $\frac{\log q}{\log Y}=\frac{1}{\theta_j}$. For $0\le r\le j$ we argue as in the previous case to bound the contribution of the primes by $D_j(\chi;k)$ and use \eqref{eq:primesquare2} to estimate the  contribution from the prime squares.
Therefore  we have
\[
\begin{split}
&|L(\tfrac12,f\otimes \chi)|^{2k} L(1,\tmop{Sym}^2 f \otimes \chi^2)^{-k} L(1,\chi^2)^{k} \\
&\qquad\qquad\ll \exp\left(\frac{6k}{\theta_j}\right) D_j(\chi;k) \\
&\qquad\qquad\ll \exp\left(\frac{6k}{\theta_j}\right) D_j(\chi;k) \max_{j+1\le u\le J}\left(\frac{e^2k \tmop{Re} (P_{I_{j+1}}(\chi; a_{f,u}))}{\ell_{j+1}}\right)^{t_{j+1}},
\end{split}
\]
where in the last step we have trivially applied
\[
\max_{j+1\le u\le J}\left(\frac{e^2k \tmop{Re} (P_{I_{j+1}}(\chi; a_{f,u}))}{\ell_{j+1}}\right)^{t_{j+1}} \ge 1.
\]
\end{proof}

\begin{lemma}\label{lem:Expectations}
For $0\leq j\leq J$ let $t_j\in\mathbb{Z}$ such that $4k\ell_j\le t_j\le\frac{2}{5\theta_j}$. Let $b(n)$ be a completely multiplicative function with $b(p) \ll 1$. Then
\begin{align*}
&\sumstar_{\chi\pamod q} D_j(\chi;k)\big(\tmop{Re}(P_{I_{j+1}}(\chi;b))\big)^{2t_{j+1}}|M_f(\chi)|^{2k}\ll q e^{4k^2(J-j)}\frac{(2t_{j+1})!}{2^{2t_{j+1}}\lfloor\frac{3}{4}t_{j+1}\rfloor!}\bigg(\sum_{p\in I_{j+1}}\frac{b(p)^2}{p}\bigg)^{t_{j+1}}
\end{align*}
for each $0\leq j\leq J-1$ and 
\[
\sumstar_{\chi\pamod q}\big(\tmop{Re}(P_{I_0}(\chi;b))\big)^{2t_0}|M_f(\chi)|^{2k}
\ll q(\log q)^{O(1)}\frac{(2t_0)!}{2^{2t_0}\lfloor\frac{3}{4}t_0\rfloor!}\bigg(\sum_{p\in I_0}\frac{b(p)^2}{p}\bigg)^{t_0}.
\]
\end{lemma}

\begin{proof}
First, by Lemma \ref{lem:diagonal} and arguing as in Remark \ref{rem:indepexp}, we have that 
\begin{equation}\label{eq:expectprod}
\begin{split}
&\frac{1}{\varphi(q)}\ \sumstar_{\chi\pamod q} D_j(\chi;k)\big(\tmop{Re}(P_{I_{j+1}}(\chi;b))\big)^{2t_{j+1}}|M_f(\chi)|^{2k}\\
&\qquad=\prod_{0\le r_1\le j}(1+e^{-\ell_{r_1}/2})\mathbb{E}\Big(E_{\ell_{r_1}}\big(2k\tmop{Re}(P_{I_{r_1}}(X;a_{f,j}))\big)|M_{f,r_1}(X)|^{2k}\Big)\\
&\qquad\qquad\qquad\times\prod_{j+1\le r_2\le J}\mathbb{E}\Big(\big(\tmop{Re}(P_{I_{j+1}}(X,b))\big)^{2t_{j+1}}|M_{f,r_2}(X)|^{2k}\Big)+O\left(q^{-1/10}\right).
\end{split}
\end{equation}
For the product over $0\le r_1\le j$, we apply Lemma \ref{lem:mollifiedeuler} to get the bound
\[
\ll \prod_{0 \le r_1 \le j}\left(1+O\left(\frac{\mathbf{1}_{r_1=0} (\log q)^{O(1)}+1}{2^{\ell_{r_1}}} \right)\right) \prod_{c_0 < p \le q^{\theta_j}} \left(1+\frac{k^2(a_{f,J}(p)-a_{f,j}(p))^2}{p}+O\left( \frac{1}{p^2}\right) \right).
\]
We have that
\begin{equation}\label{eq:aJajsum}
\sum_{c_0<p\leq q^{\theta_j}}\frac{(a_{f,J}(p)-a_{f,j}(p))^2}{p}
\ll\sum_{c_0< p\leq q^{\theta_j}}\frac{\log p}{\theta_jp\log q}
\ll1,
\end{equation}
hence the product above is bounded by $\ll 1$. 

In the case $r_2=j+1$, we apply Lemma \ref{lem:mollifiedeuler} to obtain the bound
\[
\ll\frac{(2t_{j+1})!}{2^{2t_{j+1}}\lfloor\frac{3}{4}t_{j+1}\rfloor!}\bigg(\sum_{p\in I_{j+1}}\frac{b(p)^2}{p}\bigg)^{t_{j+1}}.
\]

It remains to bound the product over $j+1<r_2\le J$ in \eqref{eq:expectprod}. Using \eqref{eq:mfmoment} we have
\begin{equation}\label{eq:rJcase}
\prod_{j+1< r_2\le J}\mathbb{E}\Big(|M_{f,r_2}(X)|^{2k}\Big)
\ll\prod_{q^{\theta_j}< p\le q^{\theta_J}}\left(1+\frac{k^2\lambda_f(p)^2}{p}\right)\ll e^{4k^2(J-j)},
\end{equation}
where we have used that $|\lambda_f(p)|\le 2$. Combining these estimates, we obtain the first statement. 

For the second statement, by Lemma \ref{lem:diagonal} we need to bound 
\[
\prod_{0\le j\le J}\mathbb{E}\left(\big(\tmop{Re}(P_{I_0}(X;b))\big)^{2t_0}|M_{f,j}(X)|^{2k}\right).
\]
For the $j=0$ term, by Lemma \ref{lem:mollifiedeuler} we have the bound
\[
\ll(\log q)^{O(1)}\frac{(2t_0)!}{2^{2t_0}\lfloor\frac{3}{4}t_0\rfloor!}\bigg(\sum_{p\in I_0}\frac{b(p)^2}{p}\bigg)^{t_0}.
\]
For the remaining terms $1\le j\le J$, by \eqref{eq:rJcase} we have the bound $e^{4k^2J}\ll(\log\log q)^{O(1)}\ll(\log q)^{O(1)}$, completing the proof.
\end{proof}

\subsection{The Proof of Proposition \ref{prop:mollifiedf}}
First, note that
\[
\sumstar_{\chi \pamod q} |L(\tfrac12, f\otimes \chi)M_f(\chi)|^{2k} 
=\sumstar_{\substack{\chi \pamod q \\ \chi^2\neq \chi_0}}  |L(\tfrac12, f\otimes \chi)M_f(\chi)|^{2k}+|L(\tfrac12, f\otimes \chi_1) M_f(\chi_1)|^{2k},
\]
where $\chi_1$ is the non-principal character mod $q$ satisfying $\chi_1^2=\chi_0$. The $\chi_1$ term is negligible using bounds on GRH for the $L$-functions and a pointwise bound for the mollifier. Defining 
\[
A(\chi)=\frac{L(1,\chi^2)}{L(1,\tmop{Sym}^2 f \otimes \chi^2)}.
\]
It suffices to show that
\begin{equation}\label{eq:Aupperbound}
\sumprime_{\chi \pamod q} |L(\tfrac 12 ,f\otimes\chi) M_f(\chi)|^{2k}|A(\chi)|^k\ll q,
\end{equation}
where $\sumprime$ denotes a sum over primitive characters $\chi^2\neq \chi_0$. Applying Cauchy-Schwarz's inequality followed by \eqref{eq:Aupperbound} and Lemma \ref{lem:L1} establishes Proposition \ref{prop:mollifiedf}.

We now prove \eqref{eq:Aupperbound}. We split the sum over $\chi$ modulo $q$ according to whether or not 
\begin{equation}\label{eq:Pbound}
\tmop{Re}( P_{I_0}(\chi;a_{f,j}))\ge\frac{\ell_0}{ke^2}
\end{equation}
for some $0\le j\le J$. In the case that \eqref{eq:Pbound} holds, we apply Chebyshev's inequality and Cauchy-Schwarz's inequality to see that
\begin{equation} \label{eq:final}
\begin{split}
&\sumprime_{\substack{\chi\pamod q\\ \tmop{Re}(P_{I_0}(\chi;a_{f,j}))\ge\frac{\ell_0}{ke^2}}}  |L(\tfrac 12 ,f\otimes\chi)  M_f(\chi)|^{2k}|A(\chi)|^k\\
&\ll\sumprime_{\chi \pamod q }  |L(\tfrac 12 ,f\otimes\chi) M_f(\chi)|^{2k}|A(\chi)|^k\left(\frac{ke^2\tmop{Re}(P_{I_0}(\chi;a_{f,j}))}{\ell_0}\right)^{2t_0}\\
&\le\bigg(\ \sumprime_{\chi \pamod q}  |L(\tfrac 12 ,f\otimes\chi)|^{4k}|A(\chi)|^{2k}\bigg)^{1/2} \bigg(\frac{(ke^2)^{4t_0}}{\ell_0^{4t_0}}\sumprime_{\chi \pamod q }\big(\tmop{Re}(P_{I_0}(\chi;a_{f,j}))\big)^{4t_0} |M_f(\chi)|^{4k}\bigg)^{1/2}.
\end{split}
\end{equation}
The argument given by Soundararajan \cite{S} carries over to give that \[\sumstar_{\chi \pmod q} |L(\tfrac12,f\otimes \chi)|^{2k} \ll q(\log q)^{k^2+\varepsilon}\] for any $k>0$. Using Cauchy-Schwarz's inequality, this estimate and Lemma \ref{lem:L1} the first sum on the right hand side of \eqref{eq:final} is $\ll q(\log q)^{O(1)}$. Applying Lemma \ref{lem:Expectations} with $t_0=\lfloor \frac{1}{5\theta_0}\rfloor$ followed by Stirling's formula, the second sum on the right hand side of \eqref{eq:final} is bounded by
\[
\ll q(\log q)^{O(1)}\left(\frac{k^2e^4}{\ell_0^2}\right)^{2t_0}\frac{(4t_0)!}{2^{4t_0}\lfloor \tfrac 32 t_0\rfloor !} \bigg( 5 \log \log q \bigg)^{2t_0}
\ll q(\log q)^{O(1)}\exp\Big(-(\log\log q)^4\Big),
\]
noting we can extend the summation to all primitive $\chi$ modulo $q$ using non-negativity. Combining these estimates, the contribution of the $\chi$ modulo $q$ satisfying \eqref{eq:Pbound} to \eqref{eq:Aupperbound} is $\ll qJ(\log q)^{-10}$.

For the remaining characters $\chi$ modulo $q$, we apply Lemma \ref{lem:Lbd} to see that these characters contribute
\begin{equation}\label{eq:Lbdapplication}
\begin{split}
&\ll\sumstar_{\chi\pamod q}D_J(\chi;k) |M_f(\chi)|^{2k}\\
&\qquad+\sum_{\substack{0 \le j \le J-1 \\ j+1 \le r \le J}}  \exp\left(\frac{6k}{\theta_j} \right) \sumstar_{\chi\pamod q}D_j(\chi;k) \left( \frac{e^2 k \tmop{Re}(P_{I_{j+1}}(\chi;a_{f,r}))}{\ell_{j+1}}\right)^{2t_{j+1}} |M_f(\chi)|^{2k}.
\end{split}
\end{equation}
Again we have extended the sum to all primitive characters using non-negativity. By Lemma \ref{lem:diagonal} and Lemma \ref{lem:mollifiedeuler} we have that
\[
\begin{split}
&\frac{1}{\varphi(q)}\sumstar_{\chi\pamod q}D_J(\chi;k)|M_f(\chi)|^{2k}\ll \mathbb{E}\Big(D_J(X;k) |M_f(X)|^{2k}\Big)\\
&\qquad\ll \prod_{0\le j\le J}\left(1+O\left(\frac{\mathbf{1}_{j=0} (\log q)^{O(1)}+1}{2^{\ell_j}} \right)\right) \prod_{p \in I_j} \left(1+\frac{k^2(a_{f,J}(p)-a_{f,J}(p))^2}{p}+O\left( \frac{1}{p^2}\right) \right)\\
&\qquad\ll 1,
\end{split}
\]
where the last bound follows by \eqref{eq:aJajsum}. It remains to show that the second term of \eqref{eq:Lbdapplication} is $\ll q$. By Lemma \ref{lem:Expectations} with $t_j=\lfloor \frac{2}{5\theta_j}\rfloor$, the second term is bounded by
\[
\ll q\sum_{\substack{0 \le j \le J-1 \\ j+1 \le r \le J}}  \exp\left(\frac{6k}{\theta_j} \right)e^{4k^2(J-j)}\left(\frac{ke^2}{\ell_{j+1}}\right)^{2t_{j+1}} \frac{(2t_{j+1})!}{2^{2t_{j+1}}\lfloor\frac{3}{4}t_{j+1}\rfloor!}\bigg(\sum_{p\in I_{j+1}}\frac{a_{f,r}(p)^2}{p}\bigg)^{t_{j+1}}.
\]
We estimate the inner sum over primes trivially as $\le 5 \log(\frac{\theta_{j+1}}{\theta_j})=5$. The sum over $r$ then trivially contributes $J-j$, so that the above is bounded by
\[
\ll q\sum_{0 \le j \le J-1 } (J-j) \exp\left(\frac{6k}{\theta_j} \right)e^{4k^2(J-j)}\left(\frac{3ke^2}{\ell_{j+1}}\right)^{2t_{j+1}} \frac{(2t_{j+1})!}{2^{2t_{j+1}}\lfloor\frac{3}{4}t_{j+1}\rfloor!}.
\]
We now apply Stirling's formula and use that $t_{j+1}^{5/8}/\ell_{j+1}\ll\theta_{j+1}^{1/8}$ to get the bound
\[
\ll q \sum_{0 \le j \le J-1 } (J-j) e^{4k^2(J-j)}\exp\left(-\frac{c}{\theta_j}\log\frac{1}{\theta_j}+O\left(\frac{1}{\theta_j}\right)\right)
\]
for some $c>0$. Noting that $\theta_j=e^{j-J}\theta_J$ and relabelling, this is
\[
\ll q \sum_{1 \le j \le J} j e^{4k^2j}\exp\left(-\frac{cje^{j}}{2\theta_J}\right)
\ll q,
\]
as claimed. As both terms of \eqref{eq:Lbdapplication} are $\ll q$ this completes the proof.
\qed

\section{Weighted central limit theorem for Dirichlet $L$-functions} \label{sec:twistedD}

In this section we sketch a proof of Theorem \ref{thm:CLTDirichlet}. Recall that
\begin{align*}
\mathcal{M}_j(\chi) = \sum_{\substack{p\mid n \Rightarrow p \in I_j \\ \Omega(n)\leq \ell_j}} \frac{\lambda(n) \nu(n) \chi(n)}{\sqrt{n}}
\end{align*}
and
\begin{align}\label{defmolD}
\mathcal{M}(\chi) = \prod_{j=0}^J \mathcal{M}_j(\chi).
\end{align}
We can write
\begin{align*}
\mathcal{M}(\chi) &= \sum_{n\leq q^\vartheta} \frac{\gamma(n) \chi(n)}{\sqrt{n}}
\end{align*}
with $\vartheta=1/1000$, $\gamma(1) = 1$ and $|\gamma(n)| \leq 1$ otherwise.
It is not difficult to check that
\begin{align*}
\varphi_\mathcal{W}^\star(q) = \sideset{}{^*}\sum_{\chi\ (\text{mod}\ q)} \mathcal{W}(\chi) = \varphi(q) (1+O(q^{-\delta}))
\end{align*}
for some absolute constant $\delta > 0$. Additionally, let
\[
P(\chi)= \sum_{c_0<p \le y}\frac{\chi(p)}{\sqrt{p}}.
\]

The proof is similar in many respects to the proof of Theorem \ref{thm:cltjoint}. As in the proof of Theorem \ref{thm:cltjoint}, we consider a set $\mathcal{S}$ of characters $\chi$ such that
\begin{enumerate}
\item $\displaystyle \Lambda^{-1} \le |L(\tfrac{1}{2},\chi)\mathcal{M}(\chi)| \leq \Lambda$;
\item $\displaystyle\prod_{j=1}^J |\mathcal{M}_j(\chi)| \leq (\log \log q)^{B}$;
\item $\displaystyle \frac{|P(\chi)|}{\sqrt{\frac12 \log \log q}} \leq \log \log \log q$,
\end{enumerate}
where $\Lambda = \log \log q$. 
Like \eqref{eq:clean} we obtain
\begin{align*}
\tmop{Re}(P(\chi)) = \log |L(\tfrac{1}{2},\chi)| + O(\log \log \log q)
\end{align*}
for $\chi \in \mathcal{S}$. 

Write $\mathcal L(\chi)=\frac{\log |L(\frac{1}{2},\chi)|}{\sqrt{\frac12 \log \log q}}$ and $\mathcal P(\chi)= \frac{\tmop{Re}(P(\chi))}{\sqrt{\frac12 \log \log q}}$.
Taking $\delta=\frac{C \log \log \log q}{\sqrt{\log \log q}}$ for $C$ sufficiently large it follows for $\chi \in \mathcal S$ and $I=[a,b]$ that
\[
|\mathbf{1}_{I}(\mathcal L(\chi))-\mathbf{1}_{I}(\mathcal P(\chi))| \le \mathbf{1}_{[a-\delta,a+\delta]}(\mathcal P(\chi))+\mathbf{1}_{[b-\delta,b+\delta]}(\mathcal P(\chi)).
\]
Setting $I_{\delta}=[a-\delta,a+\delta] \cup [b-\delta,b+\delta]$, an
 application of Cauchy-Schwarz's inequality and Proposition \ref{upperboundsecondDirichlet} gives
\begin{equation} \label{eq:dirclttransfer}
\frac{1}{\varphi(q)} \sumstar_{\chi \in \mathcal S} \mathcal{W}(\chi) \mathbf{1}_{I}(\mathcal L(\chi))=\frac{1}{\varphi(q)} \sumstar_{\chi \in \mathcal S} \mathcal{W}(\chi) \mathbf{1}_{I}(\mathcal P(\chi))+O\bigg(\frac{1}{\sqrt{q}} \bigg(\ \sumstar_{\chi \pamod q} \mathbf 1_{I_{\delta}}(\mathcal P(\chi))\bigg)^{1/2} \bigg).
\end{equation}
By another application of Cauchy-Schwarz's inequality and Proposition \ref{upperboundsecondDirichlet} the sums can be extended to all primitive characters modulo $q$ at the cost of an error term of size $O((\log \log q)^{-1})$. Following the proof of Theorem \ref{thm:cltjoint}, an analogous argument shows that the error term is $\ll \frac{(\log \log \log q)^{1/2}}{(\log \log q)^{1/4}}$.

It remains to estimate 
\[
\frac{1}{\varphi_\mathcal{W}^{\star}(q)}\, \sumstar_{\chi \pamod q} \mathcal{W}(\chi) \mathbf{1}_{I}(\mathcal P(\chi)).
\]
Following the argument from the proof of Theorem \ref{thm:cltjoint}, using Proposition \ref{upperboundsecondDirichlet} in place of Proposition \ref{prop:mollified}, we see that Theorem \ref{thm:CLTDirichlet} follows once we have shown that
\begin{equation} \label{eq:goal}
\frac{1}{\varphi_\mathcal{W}^{\star}(q)}\, \sumstar_{\chi \pamod q} \mathcal{W}(\chi) e^{iu \tmop{Re}(P(\chi))}=e^{\frac{-u^2}{4} \log \log y}(1+O(u^2))+O((\log q)^{-10}),
\end{equation}
which is analogous to \eqref{thm:mainresult2}. 
Now we want to relate everything to a random setup. It is helpful to normalize by $\varphi_\mathcal{W}^\star(q)$. Recalling that $\varphi_\mathcal{W}^\star(q) = \varphi^{\star}(q) (1+O(q^{-\delta}))$, we see that
\begin{align*}
\frac{1}{\varphi_\mathcal{W}^\star(q)}\ \sideset{}{^*}\sum_{\chi\ (\text{mod}\ q)} \mathcal{W}(\chi) e^{iu \tmop{Re}(P(\chi))} = \frac{1}{\varphi^{\star}(q)}\ \sideset{}{^*}\sum_{\chi\ (\text{mod}\ q)} \mathcal{W}(\chi) e^{iu \tmop{Re}(P(\chi))} + O \bigg(q^{-1-\delta}\ \sideset{}{^*}\sum_{\chi\ (\text{mod}\ q)} |\mathcal{W}(\chi)| \bigg).
\end{align*}
We bound the error term using Cauchy-Schwarz's inequality and Proposition \ref{upperboundsecondDirichlet}, and see that it is acceptably small.

We argue as in Lemma \ref{lem:matchrandom} to get
\begin{align*}
\frac{1}{\varphi^{\star}(q)}\ \sideset{}{^*}\sum_{\chi\ (\text{mod}\ q)} \mathcal{W}(\chi) e^{iu \text{Re}(P(\chi))} &= \sum_{0\leq j \leq J} \frac{(iu/2)^j}{j!} \sum_{k=0}^j {j\choose k} \frac{1}{\varphi(q)}\ \sideset{}{^*}\sum_{\chi\ (\text{mod}\ q)} \mathcal{W}(\chi) P(\chi)^k P(\overline{\chi})^{j-k}\\
&\qquad\qquad + O((\log q)^{-10}),
\end{align*}
where we have used the binomial theorem and the fact that $\text{Re} (z) = \frac{1}{2}(z+\overline{z})$. We open everything up to see the sum over $\chi$ is equal to
\begin{align*}
 \sum_{n\leq q^\vartheta} \frac{\gamma(n)}{\sqrt{n}}\sum_{\substack{c_0 < p_1,\ldots,p_k \leq y \\ c_0 < q_1,\ldots,q_{j-k} \leq y}} \frac{1}{\sqrt{p_1\cdots p_kq_1\cdots q_{j-k}}} \frac{1}{\varphi(q)}\ \sideset{}{^*}\sum_{\chi\ (\text{mod}\ q)} L(\tfrac{1}{2},\chi) \chi(np_1\cdots p_k) \overline{\chi(q_1\cdots q_{j-k})}.
\end{align*}
Since $np_1\cdots p_k,q_1\cdots q_{j-k}$ are small we may show that
\begin{align*}
\frac{1}{\varphi^{\star}(q)}\ \sideset{}{^*}\sum_{\chi\ (\text{mod}\ q)} L(\tfrac{1}{2},\chi) \chi(np_1\cdots p_k) \overline{\chi(q_1\cdots q_{j-k})} = \frac{\mathbf{1}_{mnp_1\cdots p_k= q_1\cdots q_{j-k}}}{\sqrt{m}} + O(q^{-\delta}),
\end{align*}
for some absolute $\delta > 0$. Since $q_i \leq y$ we note that $m,n$ can only be composed of primes $\leq y$, i.e. only $M_0(\chi)$ contributes anything here in the mollifier. It follows that, by rewinding everything as in the proof of Lemma \ref{lem:matchrandom}, we have
\begin{align*}
\frac{1}{\varphi^{\star}(q)}\ \sideset{}{^*}\sum_{\chi\ (\text{mod}\ q)} \mathcal{W}(\chi) e^{iu \text{Re}(P(\chi))} = \mathbb{E}\bigg(L(X)\mathcal{M}_0(X) \exp\Big(iu \text{Re}(P(X))\Big) \bigg) + O((\log q)^{-10}),
\end{align*}
where
\begin{align*}
L(X) &= \sum_{p \mid m \Rightarrow p \leq y} \frac{1}{\sqrt{m}}X(m), \\
\mathcal{M}_0(X) &=\sum_{\substack{p\mid n \Rightarrow c_0 < p\leq y \\ \Omega(n)\leq \ell_0}} \frac{\lambda(n) \nu(n)}{\sqrt{n}} X(n), \\
P(X) &= \sum_{c_0 < p\leq y} \frac{1}{\sqrt{p}}X(p).
\end{align*}
We can replace $\mathcal{M}_0(X)$ by
\begin{align*}
\widetilde{\mathcal{M}_0}(X) = \sum_{\substack{p\mid n \Rightarrow c_0 < p\leq y}} \frac{\lambda(n) \nu(n) }{\sqrt{n}}X(n)
\end{align*}
via Cauchy-Schwarz's inequality and trivial estimations as in Lemma \ref{lem:multiplicative}. We then wish to compute
\begin{align*}
\mathbb{E} \bigg(L(X)\widetilde{ \mathcal{M}_0}(X) \exp\Big(iu \text{Re}(P(X))\Big) \bigg),
\end{align*}
and by independence this is equal to
\begin{align}\label{1000}
\prod_{c_0 < p \leq y} \mathbb{E}\left( \bigg(\sum_{a\geq 0} \frac{X(p)^a}{p^{a/2}} \bigg)\bigg(\sum_{b\geq 0} \frac{(-1)^b X(p)^b}{b! p^{b/2}} \bigg)\exp\bigg(iu \frac{\text{Re}(X(p))}{p^{1/2}} \bigg) \right).
\end{align}

We work with each local factor individually. We write
\begin{align*}
\text{Re}(X(p)) = \frac{1}{2}\left(X(p) + \overline{X(p)} \right)
\end{align*}
and then use Taylor expansion and the binomial theorem to see that
\begin{align*}
&\bigg(\sum_{a\geq 0} \frac{X(p)^a}{p^{a/2}} \bigg)\bigg(\sum_{b\geq 0} \frac{(-1)^b X(p)^b}{b! p^{b/2}} \bigg)\exp\bigg(iu \frac{\text{Re}(X(p))}{p^{1/2}} \bigg)\\
&\qquad\qquad=\sum_{\substack{a,b,c\geq 0\\0\leq d \leq c}} \frac{(-1)^b (iu/2)^c}{b! c!p^{(a+b+c)/2}}{c \choose d} X(p)^{a+b+d} \overline{X(p)^{c-d}}.
\end{align*}
Taking expectations, we see we have no contribution unless $a+b+d = c-d$. We add $c+d$ to both sides to get $a+b+c+2d = 2c$. Writing $2k = a+b+c$ so that $c=k+d$, and then $a=k-b-d$.
After some simplification we see that 
\begin{align*}
&\mathbb{E}\left( \bigg(\sum_{a\geq 0} \frac{X(p)^a}{p^{a/2}} \bigg)\bigg(\sum_{b\geq 0} \frac{(-1)^b X(p)^b}{b! p^{b/2}} \bigg)\exp\bigg(iu \frac{\text{Re}(X(p))}{p^{1/2}} \bigg) \right)\\
&\qquad\qquad=\sum_{k \geq 0}\frac{(iu/2)^k}{p^k}\sum_{\substack{b,d\geq 0 \\ b+d\leq k}} \frac{(-1)^b (iu/2)^{d}}{b! (k+d)!} {{k + d} \choose d}.
\end{align*}
The contribution from $k = 0$ is obviously $1$. The contribution from $k = 1$ is, after some work, seen to be
\begin{align*}
\frac{(iu/2)^2}{p} = -\frac{u^2}{4p}.
\end{align*}
The contribution from $k \geq 2$ is $O(u^2/p^2)$. 
Therefore \eqref{1000} is equal to
\begin{align*}
& \prod_{c_0 < p \leq y} \left(1 - \frac{u^2}{4p} +O\left(\frac{u^2}{p^2} \right)\right) = \prod_{c_0 < p \leq y} \left(1 - \frac{u^2}{4p}\right) \left(1+O\left(\frac{u^2}{p^2} \right)\right) \\
&\qquad\qquad= \exp \left( - \frac{u^2}{4}\log \log y\right)(1+O(u^2)),
\end{align*}
where in the last step we have used Taylor series expansions and Mertens' theorem. Collecting everything together establishes \eqref{eq:goal}. Hence, following the argument of Theorem \ref{thm:cltjoint} we get that
\[
\frac{1}{\varphi_\mathcal{W}^{\star}(q)} \sumstar_{\chi \pamod q} \mathcal{W}(\chi) \mathbf{1}_{I}\bigg(\frac{\log |L(\tfrac12,\chi)|}{\sqrt{\tfrac12 \log \log q}} \bigg)= \frac{1}{\sqrt{2\pi}}\int_I e^{-u^2/2} \, du+O \bigg( \frac{(\log \log \log q)^{1/2}}{(\log \log q)^{1/4}}\bigg),
\]
which finishes the proof. \qed

\section{Upper bound for the mollified second moment} \label{sec:upperboundsD}

In this section we prove Proposition \ref{upperboundsecondDirichlet} for the family of even characters. The estimate for the odd characters follows from a similar argument. We quote the following twisted second moment of Dirichlet $L$-functions (see \cite{IS} and \cite[Theorem 1.1]{BPRZ}).
\begin{lemma}\label{secondDirichlet}
Let $(x_n)$ be a sequence of real numbers supported on $1\leq n\leq L$ such that $x_n\ll n^\varepsilon$. Then we have
\begin{align*}
&\frac{1}{\varphi^{+}(q)}\sum_{m,n\leq L}\frac{x_mx_n}{\sqrt{mn}}\ \sumplus_{\chi \pamod q} L(\tfrac12+\alpha, \chi)L(\tfrac12+\beta,\overline{\chi})\chi(m) \overline{\chi(n)}\\
&\qquad\quad=\zeta(1+\alpha+\beta)\sum_{\substack{hm,hn\leq L\\(m,n)=1}} \frac{x_{hm}x_{hn}}{hm^{1+\alpha}n^{1+\beta}}\\
&\qquad\qquad\qquad+\Big(\frac{q}{\pi}\Big)^{-(\alpha+\beta)}\frac{\Gamma(\frac{\frac12-\alpha}{2})\Gamma(\frac{\frac12-\beta}{2})}{\Gamma(\frac{\frac12+\alpha}{2})\Gamma(\frac{\frac12+\beta}{2})}\zeta(1-\alpha-\beta)\sum_{\substack{hm,hn\leq L\\(m,n)=1}} \frac{x_{hm}x_{hn}}{hm^{1-\beta}n^{1-\alpha}}+O_{\varepsilon}(q^{-1/2+\varepsilon}L)
\end{align*}
uniformly for $|\alpha|,|\beta|\ll (\log q)^{-1}$, where $\varphi^{+}(q)$ is the number of even
primitive characters modulo $q$.
\end{lemma}

\begin{proof}[Proof of Proposition \ref{upperboundsecondDirichlet}]
Since the expression in Proposition \ref{upperboundsecondDirichlet} is holomorphic in $\alpha$ and $\beta$, it suffices by the maximum modulus principle to prove the proposition uniformly over any fixed annuli such that $|\alpha|,|\beta| \asymp (\log q)^{-1}$, $|\alpha + \beta| \gg (\log q)^{-1}$. Applying Lemma \ref{secondDirichlet} we get
\begin{align}\label{1001}
&\frac{1}{\varphi^{+}(q)}\ \sumplus_{\chi \pamod q}L(\tfrac12+\alpha, \chi)L(\tfrac12+\beta,\overline{\chi})|\mathcal{M}(\chi)|^2=\zeta(1+\alpha+\beta)M(\alpha,\beta)\\
&\qquad\qquad +\Big(\frac{q}{\pi}\Big)^{-(\alpha+\beta)}\frac{\Gamma(\frac{\frac12-\alpha}{2})\Gamma(\frac{\frac12-\beta}{2})}{\Gamma(\frac{\frac12+\alpha}{2})\Gamma(\frac{\frac12+\beta}{2})}\zeta(1-\alpha-\beta)M(-\beta,-\alpha)+O_\varepsilon(q^{-1/2+\vartheta+\varepsilon}),\nonumber
\end{align}
where
\begin{align*}
M(\alpha,\beta)&=\sum_{\substack{hm,hn\leq q^{\vartheta}\\(m,n)=1}} \frac{\gamma(hm)\gamma(hn)}{hm^{1+\alpha}n^{1+\beta}}=\sum_{\substack{hdm,hdn\leq q^{\vartheta}}} \frac{\mu(d)\gamma(hdm)\gamma(hdn)}{hd^{2+\alpha+\beta}m^{1+\alpha}n^{1+\beta}}.
\end{align*}
By multiplicativity and the definition of the mollifier in \eqref{defmolD} we have
\begin{align}\label{formulaM1uvw}
&M(\alpha,\beta)=\prod_{0\leq j\leq J}\sum_{\substack{p|hdmn\Rightarrow p\in I_j\\\Omega(hdm),\Omega(hdn)\leq\ell_j}}\frac{\mu(d) \lambda(mn)\nu(hdm)\nu(hdn)}{hd^{2+\alpha+\beta}m^{1+\alpha}n^{1+\beta}}.
\end{align}

We now estimate the inner sum on the right hand side of \eqref{formulaM1uvw}. This is
\begin{align*}
&\sum_{\substack{p|hdmn\Rightarrow p\in I_j}}\frac{\mu(d)\lambda(mn)\nu(hdm)\nu(hdn)}{hd^{2+\alpha+\beta}m^{1+\alpha}n^{1+\beta}}+O\bigg(\frac{1}{2^{\ell_j}}\sum_{\substack{p|hdmn\Rightarrow p\in I_j}}\frac{ \nu(hdm)\nu(hdn)}{hd^{2}mn}\bigg).
\end{align*}
The error term is
\begin{align}\label{errorsecond}
&\ll \frac{1}{2^{\ell_j}}\prod_{\substack{p\in I_j}}\left(1+O\left(\frac{1}{p}\right)\right) \ll \frac{\mathbf{1}_{j=0}(\log q)^{O(1)}+1}{2^{\ell_j}}.
\end{align}
For the main term we write it as an Euler product
\begin{align*}
\prod_{p \in I_j}\sum_{h,d,m,n=0}^{\infty}\frac{\mu(p^d) \lambda(p^{m+n})\nu(p^{h+d+m})\nu(p^{h+d+n})}{p^{h+(2+\alpha+\beta)d+(1+\alpha)m+(1+\beta)n}}&=\prod_{p \in I_j}\bigg(1+\frac{1}{p}-\frac{1}{p^{1+\alpha}}-\frac{1}{p^{1+\beta}}+O\bigg(\frac{1}{p^{2}}\bigg)\bigg)\\
&=\prod_{p \in I_j}\bigg(1-\frac{1}{p}+O\bigg(\frac{\log p}{p\log q}\bigg)+O\bigg(\frac{1}{p^{2}}\bigg)\bigg).
\end{align*}
As
\[
\prod_{p \in I_j}\bigg(1-\frac{1}{p}+O\bigg(\frac{\log p}{p\log q}\bigg)+O\bigg(\frac{1}{p^{2}}\bigg)\bigg)^{-1}\ll \mathbf{1}_{j=0}(\log q)^{O(1)}+1,
\]
combining with \eqref{errorsecond} we obtain that
\begin{align*}
M(\alpha,\beta)&=\prod_{c_0<p\leq x}\bigg(1-\frac{1}{p}+O\bigg(\frac{\log p}{p\log q}\bigg)+O\bigg(\frac{1}{p^{2}}\bigg)\bigg)\prod_{0\leq j\leq J}\left(1+O\bigg(\frac{\mathbf{1}_{j=0}(\log q)^{O(1)}+1}{2^{\ell_j}}\bigg)\right).
\end{align*}
Note that the first product is $\asymp (\log q)^{-1}$ and the second product is $\asymp 1$, and hence
\[
M(\alpha,\beta)\asymp (\log q)^{-1}.
\]
In view of \eqref{1001} we obtain the proposition.
\end{proof}

\section*{Acknowledgments}
We would like to thank Sandro Bettin for insightful comments regarding Corollary \ref{cor:nonvanishing}.
S.L. is partially supported by EPSRC Standard Grant EP/T028343/1. N.E. is supported by the Engineering and Physical Sciences Research Council [EP/R513106/1].

\bibliographystyle{amsplain}
\bibliography{references4}

\providecommand{\bysame}{\leavevmode\hbox to3em{\hrulefill}\thinspace}
\providecommand{\MR}{\relax\ifhmode\unskip\space\fi MR }
\providecommand{\MRhref}[2]{%
  \href{http://www.ams.org/mathscinet-getitem?mr=#1}{#2}
}
\providecommand{\href}[2]{#2}
\begin{thebibliography}{10}

\bibitem{ABBRS19}
Louis-Pierre Arguin, David Belius, Paul Bourgade, Maksym Radziwi{\l}{\l}, and
  Kannan Soundararajan, \emph{Maximum of the {R}iemann zeta function on a short
  interval of the critical line}, Comm. Pure Appl. Math. \textbf{72} (2019),
  no.~3, 500--535. \MR{3911893}

\bibitem{ABR20}
Louis-Pierre Arguin, Paul Bourgade, and Maksym Radziwi{\l}{\l}, \emph{The
  {F}yodorov--{H}iary--{K}eating {C}onjecture. {I}}, preprint,
  arXiv:2007.00988.

\bibitem{BB}
Valentin Blomer and Farrell Brumley, \emph{Simultaneous equidistribution of
  toric periods and fractional moments of {L}-functions}, preprint,
  arXiv:2009.07093.

\bibitem{BFK}
Valentin Blomer, \'Etienne Fouvry, Emmanuel Kowalski, Philippe Michel, Djordje
  Mili{\'c}evi{\'c}, and Will Sawin, \emph{The second moment theory of families
  of {L}-functions}, to appear in Mem. Amer. Math. Soc., arXiv:1804.01450.

\bibitem{Bombieri-Hejhal}
E.~Bombieri and D.~A. Hejhal, \emph{On the distribution of zeros of linear
  combinations of {E}uler products}, Duke Math. J. \textbf{80} (1995), no.~3,
  821--862. \MR{1370117}

\bibitem{BFK21}
Hung~M. Bui, Alexandra Florea, and Jonathan~P. Keating, \emph{The {R}atios
  {C}onjecture and upper bounds for negative moments of {L}-functions over
  function fields}, preprint.

\bibitem{bui-keating}
Hung~M. Bui and Jonathan~P. Keating, \emph{On the mean values of {D}irichlet
  {$L$}-functions}, Proc. Lond. Math. Soc. (3) \textbf{95} (2007), no.~2,
  273--298. \MR{2352562}

\bibitem{BPRZ}
Hung~M. Bui, Kyle Pratt, Nicolas Robles, and Alexandru Zaharescu,
  \emph{Breaking the {$\frac12$}-barrier for the twisted second moment of
  {D}irichlet {$L$}-functions}, Adv. Math. \textbf{370} (2020), 107175, 40.
  \MR{4099826}

\bibitem{C}
Vorrapan Chandee, \emph{Explicit upper bounds for {$L$}-functions on the
  critical line}, Proc. Amer. Math. Soc. \textbf{137} (2009), no.~12,
  4049--4063. \MR{2538566}

\bibitem{Chinta}
Gautam Chinta, \emph{Analytic ranks of elliptic curves over cyclotomic fields},
  J. Reine Angew. Math. \textbf{544} (2002), 13--24. \MR{1887886}

\bibitem{darbar-lumley}
Pranendu Darbar and Allysa Lumley, \emph{Selberg's central limit theorem for
  quadratic dirichlet $l$-functions over function fields}, preprint,
  arXiv:2105.10863.

\bibitem{DFL20}
Chantal David, Alexandra Florea, and Matilde Lalin, \emph{Non-vanishing for
  cubic $l$-functions}, preprint, arXiv:2006.15661.

\bibitem{DPSS15}
Martin Dickson, Ameya Pitale, Abhishek Saha, and Ralf Schmidt, \emph{Explicit
  refinements of {B}\"{o}cherer's conjecture for {S}iegel modular forms of
  squarefree level}, J. Math. Soc. Japan \textbf{72} (2020), no.~1, 251--301.
  \MR{4055095}

\bibitem{Fazzari1}
Alessandro Fazzari, \emph{A weighted central limit theorem for {$\log
  |\zeta(1/2+it)|$}}, Mathematika \textbf{67} (2021), no.~2, 324--341.
  \MR{4220993}

\bibitem{Fazzari2}
\bysame, \emph{Weighted value distributions of the {R}iemann zeta function on
  the critical line}, Forum Math. \textbf{33} (2021), no.~3, 579--592.
  \MR{4250489}

\bibitem{FHK12}
Yan~V. Fyodorov, Ghaith~A. Hiary, and Jonathan~P. Keating, \emph{Freezing
  transition, characteristic polynomials of random matrices, and the {R}iemann
  zeta function}, Phys. Rev. Lett. \textbf{108} (2012), 170601.

\bibitem{FK14}
Yan~V. Fyodorov and Jonathan~P. Keating, \emph{Freezing transitions and extreme
  values: random matrix theory, and disordered landscapes}, Philos. Trans. R.
  Soc. Lond. Ser. A Math. Phys. Eng. Sci. \textbf{372} (2014), no.~2007,
  20120503, 32. \MR{3151088}

\bibitem{gonek-keating-hughes}
S.~M. Gonek, C.~P. Hughes, and Jonathan~P. Keating, \emph{A hybrid
  {E}uler-{H}adamard product for the {R}iemann zeta function}, Duke Math. J.
  \textbf{136} (2007), no.~3, 507--549. \MR{2309173}

\bibitem{Har19}
Adam~J. Harper, \emph{On the partition function of the {R}iemann zeta function,
  and the {F}yodorov--{H}iary--{K}eating conjecture}, preprint,
  arXiv:1906.05783.

\bibitem{heap-radziwill-soundararajan}
Winston Heap, Maksym Radziwi{\l}{\l}, and Kannan Soundararajan, \emph{Sharp
  upper bounds for fractional moments of the {R}iemann zeta function}, Q. J.
  Math. \textbf{70} (2019), no.~4, 1387--1396. \MR{4045106}

\bibitem{heap-soundararajan}
Winston Heap and Kannan Soundararajan, \emph{Lower bounds for moments of zeta
  and {$L$}-functions revisited}, preprint, arXiv:2007.13154.

\bibitem{Hough}
Bob Hough, \emph{The distribution of the logarithm in an orthogonal and a
  symplectic family of {$L$}-functions}, Forum Math. \textbf{26} (2014), no.~2,
  523--546. \MR{3176641}

\bibitem{Hou16}
\bysame, \emph{The angle of large values of {$L$}-functions}, J. Number Theory
  \textbf{167} (2016), 353--393. \MR{3504052}

\bibitem{ILS}
Henryk Iwaniec, Wenzhi Luo, and Peter Sarnak, \emph{Low lying zeros of families
  of {$L$}-functions}, Inst. Hautes \'{E}tudes Sci. Publ. Math. (2000), no.~91,
  55--131 (2001). \MR{1828743}

\bibitem{IS}
Henryk Iwaniec and Peter Sarnak, \emph{Dirichlet {$L$}-functions at the central
  point}, Number theory in progress, {V}ol. 2 ({Z}akopane-{K}o\'{s}cielisko,
  1997), de Gruyter, Berlin, 1999, pp.~941--952. \MR{1689553}

\bibitem{JLS}
Jesse J\"a\"asaari, Stephen Lester, and Abhishek Saha, \emph{On fundamental
  fourier coefficeints of siegel cusp forms of degree $2$}, to appear in J.
  Inst. Math. Jussieu, arXiv:2012.09563.

\bibitem{KS00}
Jonathan~P. Keating and Nina~C. Snaith, \emph{Random matrix theory and
  {$L$}-functions at {$s=1/2$}}, Comm. Math. Phys. \textbf{214} (2000), no.~1,
  91--110. \MR{1794267}

\bibitem{LLR}
Youness Lamzouri, Stephen Lester, and Maksym Radziwi{\l}{\l}, \emph{An
  effective universality theorem for the {R}iemann zeta function}, Comment.
  Math. Helv. \textbf{93} (2018), no.~4, 709--736. \MR{3880225}

\bibitem{LR21}
Stephen Lester and Maksym Radziwi{\l}{\l}, \emph{Signs of {F}ourier
  coefficients of half-integral weight modular forms}, Math. Ann. \textbf{379}
  (2021), no.~3-4, 1553--1604. \MR{4238273}

\bibitem{MV}
Hugh~L. Montgomery and Robert~C. Vaughan, \emph{Multiplicative number theory.
  {I}. {C}lassical theory}, Cambridge Studies in Advanced Mathematics, vol.~97,
  Cambridge University Press, Cambridge, 2007. \MR{2378655}

\bibitem{Naj18}
Joseph Najnudel, \emph{On the extreme values of the {R}iemann zeta function on
  random intervals of the critical line}, Probab. Theory Related Fields
  \textbf{172} (2018), no.~1-2, 387--452. \MR{3851835}

\bibitem{Radziwill-Soundararajan-rh-notes}
Maksym Radziwi{\l}{\l}, \emph{Some analogies},
  https://aimath.org/wp-content/uploads/bristol-2018-slides/Radziwill-talk.pdf,
  Perspectives on the Riemann Hypothesis (2018).

\bibitem{Radziwill-Soundararajan}
Maksym Radziwi{\l}{\l} and Kannan Soundararajan, \emph{Moments and distribution
  of central {$L$}-values of quadratic twists of elliptic curves}, Invent.
  Math. \textbf{202} (2015), no.~3, 1029--1068. \MR{3425386}

\bibitem{RS17}
\bysame, \emph{Selberg's central limit theorem for {$\log{|\zeta(1/2+it)|}$}},
  Enseign. Math. \textbf{63} (2017), no.~1-2, 1--19. \MR{3832861}

\bibitem{Rohrlich2}
David~E. Rohrlich, \emph{On {$L$}-functions of elliptic curves and
  anticyclotomic towers}, Invent. Math. \textbf{75} (1984), no.~3, 383--408.
  \MR{735332}

\bibitem{Rohrlich1}
\bysame, \emph{On {$L$}-functions of elliptic curves and cyclotomic towers},
  Invent. Math. \textbf{75} (1984), no.~3, 409--423. \MR{735333}

\bibitem{selberg-dir}
Atle Selberg, \emph{Contributions to the theory of {D}irichlet's
  {$L$}-functions}, Skr. Norske Vid.-Akad. Oslo I \textbf{1946} (1946), no.~3,
  62. \MR{22872}

\bibitem{Selberg-contributions}
\bysame, \emph{Contributions to the theory of the {R}iemann zeta-function},
  Arch. Math. Naturvid. \textbf{48} (1946), no.~5, 89--155. \MR{20594}

\bibitem{Selberg-old-new}
\bysame, \emph{Old and new conjectures and results about a class of {D}irichlet
  series}, Proceedings of the {A}malfi {C}onference on {A}nalytic {N}umber
  {T}heory ({M}aiori, 1989), Univ. Salerno, Salerno, 1992, pp.~367--385.
  \MR{1220477}

\bibitem{SoundNonvanishing}
Kannan Soundararajan, \emph{Nonvanishing of quadratic {D}irichlet
  {$L$}-functions at {$s=\frac12$}}, Ann. of Math. (2) \textbf{152} (2000),
  no.~2, 447--488. \MR{1804529}

\bibitem{S}
\bysame, \emph{Moments of the {R}iemann zeta function}, Ann. of Math. (2)
  \textbf{170} (2009), no.~2, 981--993. \MR{2552116}

\bibitem{Titchmarsh}
E.~C. Titchmarsh, \emph{The theory of the {R}iemann zeta-function}, second ed.,
  The Clarendon Press, Oxford University Press, New York, 1986, Edited and with
  a preface by D. R. Heath-Brown. \MR{882550}

\bibitem{vaaler}
Jeffrey~D. Vaaler, \emph{Some extremal functions in {F}ourier analysis}, Bull.
  Amer. Math. Soc. (N.S.) \textbf{12} (1985), no.~2, 183--216. \MR{776471}

\bibitem{young}
Matthew~P. Young, \emph{The fourth moment of {D}irichlet {$L$}-functions}, Ann.
  of Math. (2) \textbf{173} (2011), no.~1, 1--50. \MR{2753598}

\bibitem{zacharias}
Rapha\"{e}l Zacharias, \emph{Mollification of the fourth moment of {D}irichlet
  {$L$}-functions}, Acta Arith. \textbf{191} (2019), no.~3, 201--257.
  \MR{4017531}

\end{thebibliography}

\end{document}